\documentclass[a4paper,11pt]{amsart}
\usepackage{mathrsfs}
\usepackage{cases}
\usepackage{epic}
\usepackage{amsfonts}
\usepackage{graphicx}
\usepackage{amsmath}
\usepackage{amssymb}
\usepackage{pdflscape}
\usepackage[all]{xypic}
\usepackage[all]{xy}
\usepackage{color}
\usepackage{colordvi}
\usepackage{multicol}
\usepackage[linktocpage=true]{hyperref}
\hypersetup{colorlinks,linkcolor=blue,urlcolor=cyan,citecolor=blue}
\newcommand{\nc}{\newcommand}
\nc{\brown}[1]{\textcolor{brown}{#1}}
\nc{\green}[1]{\textcolor{green}{#1}}
\nc{\red}[1]{\textcolor{red}{#1}}
\nc{\blue}[1]{\textcolor{blue}{#1}}

\begin{document}
\input xy
\xyoption{all}

\newcommand{\iadd}{\operatorname{iadd}\nolimits}
\newcommand{\res}{\operatorname{res}\nolimits}

\renewcommand{\mod}{\operatorname{mod}\nolimits}
\newcommand{\proj}{\operatorname{proj}\nolimits}
\newcommand{\inj}{\operatorname{inj.}\nolimits}
\newcommand{\rad}{\operatorname{rad}\nolimits}
\newcommand{\soc}{\operatorname{soc}\nolimits}
\newcommand{\ind}{\operatorname{inj.dim}\nolimits}
\newcommand{\Ginj}{\operatorname{Ginj}\nolimits}
\newcommand{\Mod}{\operatorname{Mod}\nolimits}
\newcommand{\R}{\operatorname{R}\nolimits}
\newcommand{\End}{\operatorname{End}\nolimits}
\newcommand{\colim}{\operatorname{colim}\nolimits}
\newcommand{\inc}{\operatorname{inc}\nolimits}
\newcommand{\gldim}{\operatorname{gl.dim}\nolimits}
\newcommand{\cone}{\operatorname{cone}\nolimits}
\newcommand{\rep}{\operatorname{rep}\nolimits}
\newcommand{\Ext}{\operatorname{Ext}\nolimits}
\newcommand{\Tor}{\operatorname{Tor}\nolimits}
\newcommand{\Hom}{\operatorname{Hom}\nolimits}
\newcommand{\Top}{\operatorname{top}\nolimits}
\newcommand{\Coker}{\operatorname{Coker}\nolimits}
\newcommand{\gid}{\operatorname{grinj.dim}\nolimits}
\newcommand{\thick}{\operatorname{thick}\nolimits}
\newcommand{\rank}{\operatorname{rank}\nolimits}
\newcommand{\Gproj}{\operatorname{Gproj}\nolimits}
\newcommand{\irr}{\operatorname{irr}\nolimits}
\newcommand{\RHom}{\operatorname{RHom}\nolimits}
\renewcommand{\deg}{\operatorname{deg}\nolimits}
\renewcommand{\Im}{\operatorname{Im}\nolimits}
\newcommand{\Ker}{\operatorname{Ker}\nolimits}
\newcommand{\coh}{\operatorname{coh}\nolimits}
\newcommand{\Id}{\operatorname{Id}\nolimits}
\newcommand{\Span}{\operatorname{Span}\nolimits}
\newcommand{\CM}{\operatorname{CM}\nolimits}
\newcommand{\For}{\operatorname{{\bf F}or}\nolimits}
\newcommand{\coker}{\operatorname{Coker}\nolimits}
\renewcommand{\dim}{\operatorname{dim}\nolimits}
\renewcommand{\div}{\operatorname{div}\nolimits}
\newcommand{\Ab}{{\operatorname{Ab}\nolimits}}
\newcommand{\diag}{{\operatorname{diag}\nolimits}}

\renewcommand{\Vec}{{\operatorname{Vec}\nolimits}}
\newcommand{\pd}{\operatorname{proj.dim}\nolimits}
\newcommand{\gr}{\operatorname{gr}\nolimits}
\newcommand{\id}{\operatorname{inj.dim}\nolimits}
\newcommand{\Gd}{\operatorname{G.dim}\nolimits}
\newcommand{\Ind}{\operatorname{Ind}\nolimits}
\newcommand{\add}{\operatorname{add}\nolimits}
\newcommand{\pr}{\operatorname{pr}\nolimits}
\newcommand{\oR}{\operatorname{R}\nolimits}
\newcommand{\oL}{\operatorname{L}\nolimits}
\newcommand{\Perf}{{\mathfrak Perf}}
\newcommand{\cc}{{\mathcal C}}
\newcommand{\gc}{{\mathcal GC}}
\newcommand{\ce}{{\mathcal E}}
\newcommand{\cs}{{\mathcal S}}
\newcommand{\cf}{{\mathcal F}}
\newcommand{\cx}{{\mathcal X}}
\newcommand{\ct}{{\mathcal T}}
\newcommand{\cu}{{\mathcal U}}
\newcommand{\cv}{{\mathcal V}}
\newcommand{\cn}{{\mathcal N}}
\newcommand{\mcr}{{\mathcal R}}
\newcommand{\ch}{{\mathcal H}}
\newcommand{\ca}{{\mathcal A}}
\newcommand{\cb}{{\mathcal B}}
\newcommand{\ci}{{\mathcal I}}
\newcommand{\cj}{{\mathcal J}}
\newcommand{\cm}{{\mathcal M}}
\newcommand{\cp}{{\mathcal P}}
\newcommand{\cg}{{\mathcal G}}
\newcommand{\cw}{{\mathcal W}}
\newcommand{\co}{{\mathcal O}}
\newcommand{\cq}{{\mathcal Q}}
\newcommand{\cd}{{\mathcal D}}
\newcommand{\ck}{{\mathcal K}}
\newcommand{\calr}{{\mathcal R}}
\newcommand{\ol}{\overline}
\newcommand{\ul}{\underline}
\newcommand{\st}{[1]}
\newcommand{\ow}{\widetilde}
\renewcommand{\P}{\mathbf{P}}
\newcommand{\pic}{\operatorname{Pic}\nolimits}
\newcommand{\Spec}{\operatorname{Spec}\nolimits}
\newtheorem{theorem}{Theorem}[subsection]
\newtheorem{acknowledgement}[theorem]{Acknowledgement}
\newtheorem{algorithm}[theorem]{Algorithm}
\newtheorem{axiom}[theorem]{Axiom}
\newtheorem{case}[theorem]{Case}
\newtheorem{claim}[theorem]{Claim}
\newtheorem{conclusion}[theorem]{Conclusion}
\newtheorem{condition}[theorem]{Condition}
\newtheorem{conjecture}[theorem]{Conjecture}
\newtheorem{construction}[theorem]{Construction}
\newtheorem{corollary}[theorem]{Corollary}
\newtheorem{criterion}[theorem]{Criterion}
\newtheorem{definition}[theorem]{Definition}
\newtheorem{example}[theorem]{Example}
\newtheorem{exercise}[theorem]{Exercise}
\newtheorem{lemma}[theorem]{Lemma}
\newtheorem{notation}[theorem]{Notation}
\newtheorem{problem}[theorem]{Problem}
\newtheorem{proposition}[theorem]{Proposition}
\newtheorem{remark}[theorem]{Remark}
\newtheorem{solution}[theorem]{Solution}
\newtheorem{summary}[theorem]{Summary}
\newtheorem*{thm}{Theorem}

\def \bp{{\mathbf p}}
\def \bA{{\mathbf A}}
\def \bL{{\mathbf L}}
\def \bF{{\mathcal F}}
\def \F{{\mathbf F}}

\def \bS{{\mathbf S}}
\def \bC{{\mathbf C}}

\def \Z{{\Bbb Z}}
\def \D{{\Bbb D}}
\def \C{{\Bbb C}}
\def \N{{\Bbb N}}
\def \Q{{\Bbb Q}}
\def \G{{\Bbb G}}
\def \X{{\Bbb X}}
\def \P{{\Bbb P}}
\def \K{{\Bbb K}}
\def \E{{\Bbb E}}
\def \A{{\Bbb A}}
\def \L{{\Bbb L}}
\def \BH{{\Bbb H}}
\def \T{{\Bbb T}}
\newcommand {\lu}[1]{\textcolor{red}{$\clubsuit$: #1}}
\newcommand {\zhu}[1]{\textcolor{red}{$\spadesuit$: #1}}

\numberwithin{equation}{subsection}
\allowdisplaybreaks

\title[Singularity categories of Gorenstein monomial algebras]{Singularity categories of Gorenstein monomial algebras\footnote{Supported by the National Natural Science Foundation of China (No. 11401401, No. 11601441, No. 12031007)}}

\author[Ming Lu]{Ming Lu}
\address{Department of Mathematics, Sichuan University, Chengdu 610064, P.R.China}
\email{luming@scu.edu.cn}

\author[Bin Zhu]{Bin Zhu}
\address{Department of Mathematical Sciences,
Tsinghua University,
Beijing 100084, P.R.China}
\email{zhu-b@tsinghua.edu.cn}


\begin{abstract}
In this paper, we consider the  singularity category $D_{sg}(\mod A)$ and the $\mathbb{Z}$-graded singularity category $D_{sg}(\mod^{\mathbb Z} A)$ for a Gorenstein monomial algebra $A$.
Firstly, for a positively graded $1$-Gorenstein algebra, we prove that its ${\mathbb Z}$-graded singularity category admits silting objects. Secondly, for $A=KQ/I$ being a Gorenstein monomial algebra, we prove that $D_{sg}(\mod^{\mathbb Z} A)$ has tilting objects. As a consequence, $D_{sg}(\mod^{\mathbb Z}A)$ is triangulated equivalent to the derived category of a hereditary algebra $H$ which is of finite representation type. Finally, we give a characterization of $1$-Gorenstein monomial algebras, and describe their singularity categories clearly by using the triangulated orbit categories of type ${\mathbb A}$.
\end{abstract}

\maketitle
\tableofcontents
\section{Introduction}

The singularity category of an algebra is defined to be the Verdier quotient of the bounded derived category with respect to the thick subcategory formed by complexes isomorphic to bounded complexes of finitely generated projective modules \cite{Bu}; see also \cite{Ha3}. The singularity category measures the homological singularity of an algebra \cite{Ha3}: an algebra has finite global dimension if and only if its singularity category is trivial. Recently, D. Orlov's global version \cite{Or1} attracted a lot of interest in algebraic geometry and theoretical physics.

The concept of the Gorenstein (also called Iwanaga-Gorenstein) algebras is inspired from commutative ring theory. A fundamental result of R. Buchweitz \cite{Bu} and D. Happel \cite{Ha3} states that for a Gorenstein algebra $A$, its singularity category is triangulated equivalent to the stable category of Gorenstein projective (also called (maximal) Cohen-Macaulay) $A$-modules, which generalizes J. Rickard's result \cite{Ri} on self-injective algebras. It is worth noting that Gorenstein
algebras, especially $1$-Gorenstein algebras are playing an important role in representation theory of finite-dimensional algebras, which include  some important classes of algebras, for examples, the cluster-tilted algebras \cite{BMR1,BMR2}, $2$-CY-tilted algebras \cite{KR, Rei}, or more general the endomorphism algebras of cluster tilting objects in triangulated categories \cite{KZ}, the class of $1$-Gorenstein algebras defined by  C. Geiss, B. Leclerc and J. Schr\"{o}er via quivers with relations associated with symmetrizable Cartan matrices \cite{GLS}.

In general, it is difficult to describe the singularity categories and the Gorenstein projective modules. Many people are trying to describe them for some special classes of algebras; see e.g. \cite{Chen,CSZ,CGLu,IO,Ka,Ki1,Ki2,Ki3,Lu,Rin,Shen1,Shen,Ya}. In this paper, we focus on describing the (graded) singularity categories of Gorenstein monomial algebras. For a monomial algebra $A$, X-W. Chen, D. Shen and G. Zhou in \cite{CSZ} give an explicit classification of indecomposable Gorenstein projective $A$-modules, which unifies the results in \cite{Rin} and \cite{Ka} to some extent. For some special monomial algebras, their singularity categories are also described explicitly; see \cite{Rin} for Nakayama algebras, \cite{Ka} for gentle algebras and \cite{Chen} for quadratic monomial algebras.

Recently K. Yamaura \cite{Ya} proved that for a finite-dimensional positively graded self-injective algebra $A=\bigoplus_{i\geq0}A_i$ with $\gldim A_0<\infty$, its stable category of the $\Z$-graded modules admits a tilting object. Inspired by this result, we consider the $\Z$-graded singularity category $D_{sg}(\mod^\Z A)$ for a positively graded Gorenstein algebra $A=\bigoplus_{i\geq0}A_i$. To state our results, we need some notations.

Following \cite{Ya}, we define the truncation functors $$(-)_{\geq i}:\mod^\Z A\rightarrow\mod^\Z A,\quad (-)_{\leq i}:\mod^\Z A\rightarrow \mod^\Z A$$
as follows. For a $\Z$-graded $A$-module $X$, $X_{\geq i}$ is a $\Z$-graded sub $A$-module of $X$
defined by
$$(X_{\geq i})_j:=\left\{ \begin{array}{ccc} 0& \mbox{ if }j<i \\ X_j& \mbox{ if }j\geq i,\end{array} \right.$$
and $X_{\leq i}$ is a $\Z$-graded factor $A$-module $X/X_{\geq i+1}$ of $X$.
Now we define a $\Z$-graded $A$-module by
\begin{equation}\label{def:T}
T:=\bigoplus_{i\geq0} A(i)_{\leq0},
\end{equation}
where $(i)$ is the degree shift functor.

First, if $A=\bigoplus_{i\geq0}A_i$ is a finite-dimensional positively graded $1$-Gorenstein algebra with $\gldim A_0<\infty$, then $T$ is a silting object in $D_{sg}(\mod^\Z A)$; see Theorem \ref{theorem silting object}. Recall that for a triangulated category $\cc$, an object $T$ is called to be a \emph{silting object}, if $\Hom_\cc(T,T[m])=0$ for any $m>0$, and $T$ generates $\cc$; see \cite{KV,AI}. Second, for any positively graded Gorenstein algebra $A=\bigoplus_{i\geq0}A_i$ with $\gldim A_0<\infty$, if $T=\bigoplus_{i\geq0} A(i)_{\leq0}$ is a Gorenstein projective $A$-module, then $T$ is a tilting object in $D_{sg}(\mod^\Z A)$; see Proposition \ref{corollary for T CM}. It is worth noting that the first author uses this result to describe the singularity categories of quiver representations over local rings $K[X]/(X^k)$, $k>1$ (see \cite{Lu}), which generalises the results in \cite{RS,RZ}.

A natural question is that when does $D_{sg}(\mod^\Z A)$ admit a tilting object for a positively graded algebra $A=\bigoplus_{i\geq0} A_i$.  We refer to \cite{HIMO,IO,IT,Ki1,Ki2,Ki3,Lu,MU,U,Ya} for recent results related to this question. In this case, $A_0$ should satisfy that $\gldim A_0<\infty$; see Lemma \ref{lemma finite global dimension}. For Gorenstein monomial algebras, we prove that it is true, i.e., their graded singularity categories admit tilting objects; see Theorem \ref{main theorem}. As a consequence, $D_{sg}(\mod^\Z A)$ is triangulated equivalent to $D^b(H)$, where $H$ is a hereditary algebra of finite representation type; see Proposition \ref{proposition equivalent to hereditary algebra}. As corollaries, for $A$ being a Gorenstein quadratic monomial algebra, there exists a hereditary algebra $B=KQ$ with $Q$ a union of quivers of type $\A_1$ such that $D_{sg}(\mod^\Z A)\simeq D^b(\mod B)$; for $A$ being a Gorenstein Nakayama algebra, then there exists a hereditary algebra $B=KQ$ with $Q$ a union of quivers of type $\A$ such that $D_{sg}(\mod^\Z A)\simeq D^b(\mod B)$.

Finally, we consider the $1$-Gorenstein monomial algebra, which is a special kind of monomial algebras. First, we characterize $1$-Gorenstein monomial algebra $A=KQ/I$ by using the minimal generators of the two-sided ideal $I$; see Theorem \ref{proposition characterize of 1 Gorenstein monomial algebras}, which generalises \cite[Proposition 3.1]{CL} for gentle algebras. Second, there exists a hereditary algebra $B=KQ^B$ with $Q^B$ a union of quivers of type $\A$ such that $D_{sg}(\mod^\Z A)\simeq D^b(\mod B)$; see Theorem \ref{singularity category of 1-Gorenstein algebras}. Third, we describe explicitly its ordinary singularity category $D_{sg}(\mod A)$ by using the triangulated orbit categories of type $\A$ in the sense of \cite{Ke}; see Theorem \ref{singularity category of 1-gorenstein monomial algebras}. In fact, in order to prove this result, we consider a kind of gluing algebras, which is defined by T. Br\"{u}stle in \cite{Br}, and called the \emph{Br\"{u}stle's gluing algebras}. We prove that singularity categories, and also Gorenstein properties are invariant under Br\"{u}stle's gluing process; see Theorem \ref{theorem singularity equivalence of gluing vertices} and Proposition \ref{proposition of gorenstein properties of gluiing algebras}.

The paper is organized as follows. In Section \ref{section 2}, we collect some materials on positively graded algebras, Gorenstein algebras, (graded) Gorenstein projective modules and (graded) singularity categories. In Section \ref{section 3}, we prove that there is a silting object in $D_{sg}(\mod^\Z A)$ for a $1$-Gorenstein positively graded algebra $A= \bigoplus_{i\geq0}A_i$ if $\gldim A_0<\infty$. In Section \ref{section 4}, we prove that $D_{sg}(\mod^\Z A)$ admits a tilting object for a Gorenstein monomial algebra $A$. In Section \ref{section 5}, we characterize $1$-Gorenstein monomial algebras. In Section \ref{section 6}, we describe the singularity categories for $1$-Gorenstein monomial algebras.

\vspace{0.2cm} \noindent{\bf Acknowledgments.}
The work was done during the stay of the first author at the Department of Mathematics,
University of Bielefeld. He is deeply indebted to Professor Henning Krause for his kind
hospitality, inspiration and continuous encouragement. Both authors deeply thank Professor Xiao-Wu Chen for helpful discussions. We thank the anonymous referee for his/her careful readings and suggestions which help make this paper more readable.

After finishing a first version of this paper appeared in arXiv, we were informed kindly from Professor Osamu Iyama and Professor Kota Yamaura that H. Minamoto, Y. Kimura and K. Yamaura \cite{MKY} has also obtained the same result with Theorem \ref{theorem silting object}. Kota Yamaura pointed out that $A$ should be of Gorenstein parameter $l$ to make the statements in Theorem \ref{theorem endomorphism ring of gorenstein projective modules} hold.  We are deeply indebted to them for these and for their helpful comments!

\section{Preliminaries}\label{section 2}
Throughout this paper $K$ is an algebraically closed field and algebras are finite-dimensional $K$-algebras. % unless otherwise specified.
We denote by $D$ the $K$-dual, i.e., $D(-)=\Hom_K(-,K)$.

Let $A$ be a $K$-algebra. We denote by $\mod A$ the category of finitely generated (left) $A$-modules, by $\proj A$ the category of finitely generated projective $A$-modules.

For an additive category $\ca$, we use $\Ind\ca$ to denote the set of all non-isomorphic indecomposable objects in $\ca$.

Let $\ca$ be an abelian category, $\cx$ a full additive subcategory of $\ca$. For any $M\in\ca$, a \emph{right $\cx$-approximation} of $M$ is a morphism $f:X\rightarrow M$ such that $X\in\cx$ and any morphism $X'\rightarrow M$ from an object $X'\in\cx$ factors through $f$. Dually one has the notion of \emph{left $\cx$-approximation}. The subcategory $\cx\subseteq \ca$ is said to be \emph{contravariantly finite} (resp. \emph{covariantly finite}) provided that each object in $\ca$ has a right (resp. left) $\cx$-approximation. The subcategory $\cs\subseteq \ca$ is said to be \emph{functorially finite} provided it is both contravariantly finite and covariantly finite.

\subsection{Positively graded algebras}
Let $A$ be a $\Z$-graded algebra, i.e., $A=\bigoplus_{i\in\Z} A_i$, where $A_i$ is the degree $i$ part of $A$.
A $\Z$-graded $A$-module $X$ is of the form $\bigoplus_{i\in \Z}X_i$, where $X_i$ is the degree $i$ part of $X$.
The category $\mod ^\Z A$ of (finitely generated) $\Z$-graded $A$-modules is defined as follows.
\begin{itemize}
\item The objects are graded $A$-modules,
\item For graded $A$-modules $X$ and $Y$, the morphism space from $X$ to $Y$ in $\mod^\Z A$ is $$\Hom_{\mod^\Z A}(X,Y):=\Hom_A(X,Y)_0:=\{ f\in \Hom_A(X,Y)\mid f(X_i)\subseteq Y_i\mbox{ for any }i\in \Z\}.$$
\end{itemize}
We denote by $\proj^\Z A$ the full subcategory of $\mod^\Z A$ consisting of projective objects.

For $i\in \Z$, we denote by $(i):\mod^\Z A\rightarrow \mod^\Z A$ the \emph{degree shift functor}.
Then for any two graded $A$-modules $X,Y$,
\begin{equation}\label{equation morphism}
\Hom_A(X,Y)=\bigoplus_{i\in \Z}\Hom_A(X,Y(i))_0.
\end{equation}

Denote by $J_A$ the \emph{Jacobson radical} of $A$.
\begin{proposition}[\text{\cite[Proposition 2.16]{Ya}}]\label{proposition characterize of simple porjective injective modules of graded algebras}
Assume that $J_A=J_{A_0}\oplus (\bigoplus_{i\neq0}A_i)$. We take a set $\overline{PI}$ of idempotents of $A_0$ such that $\{A_0e|e\in\overline{PI}\}$ is a complete list of indecomposable projective $A_0$-modules. Then the following assertions hold.
\begin{itemize}
\item[(i)] Any complete set of orthogonal primitive idempotents of $A_0$ is that of $A$.
\item[(ii)] A complete list of simple objects in $\mod^\Z A$ is given by
$$\{ S(i)\mid i\in \Z,S \mbox{ is a simple }A_0\mbox{-module}\}.$$
\item[(iii)] A complete list of indecomposable projective objects in $\mod^\Z A$ is given by
$$\{ A(i)e\mid i\in \Z,e\in\overline{PI}\}.$$
\item[(iv)] A complete list of indecomposable injective objects in $\mod^\Z A$ is given by
$$\{ D(eA)(i)\mid i\in \Z,e\in\overline{PI}\}.$$
\end{itemize}
\end{proposition}

Let $A$ be a positively graded algebra, i.e., $A=\bigoplus_{i\geq0} A_i$.
Note that for a positively graded algebra $A$, the equation
$$J_A=J_{A_0}\oplus (\bigoplus_{i\neq0}A_i)$$
always holds. So $A$ satisfies the assumption of Proposition \ref{proposition characterize of simple porjective injective modules of graded algebras}; see \cite[Proposition 2.18]{Ya}.

\subsection{Gorenstein algebras and singularity categories}
Let $A$ be an algebra.
A complex $$P^\bullet:\cdots\rightarrow P^{-1}\rightarrow P^0\xrightarrow{d^0}P^1\rightarrow \cdots$$ of finitely generated projective $A$-modules is said to be \emph{totally acyclic} provided it is acyclic and the Hom complex $\Hom_A(P^\bullet,A)$ is also acyclic \cite[Page 7]{AM}.
An $A$-module $M$ is said to be (finitely generated) \emph{Gorenstein projective} provided that there is a totally acyclic complex $P^\bullet$ of projective $A$-modules such that $M\cong \Ker d^0$; see \cite[Definition 10.2.1]{EJ}. We denote by $\Gproj A$ the full subcategory of $\mod A$ consisting of Gorenstein projective modules.

Let $\cx$ be a subcategory of $\mod A$. Define $^\bot\cx:=\{M\mid\Ext^i(M,X)=0, \mbox{ for all } X\in\cx, i\geq1\}$. Dually, we can define $\cx^\bot$. In particular, we denote by $^\bot A:={}^\bot
(\proj A)$.
The following lemma follows from the definition of Gorenstein projective modules.

\begin{lemma}
[\text{\cite[Proposition 10.2.6, Remark 10.2.9]{EJ}}]\label{lemma property of gorenstein projective modules}
Let $A$ be an algebra.
\begin{itemize}
\item[(i)]
For any $M\in\mod A$, $M$ is Gorenstein projective if and only if there exists an exact sequence $0\rightarrow M\rightarrow P^0\xrightarrow{d^0}P^1\xrightarrow{d^1}\cdots$, such that $P^i\in\proj A,\ker d^i\in{}^\bot A$, for any $i\geq0$.
\item[(ii)] If $M$ is Gorenstein projective, then $\Ext^i_A(M,L)=0$, for any $L$ of finite projective dimension or of finite injective dimension, and $i>0$.
\item[(iii)] If $P^\bullet$ is a totally acyclic complex, then all $\Im d^i$ are Gorenstein projective; and any
truncations
$$\cdots\rightarrow P^i\rightarrow\Im d^i\rightarrow0,\quad 0\rightarrow\Im d^i\rightarrow P^{i+1}\rightarrow\cdots$$
and
$$0\rightarrow\Im d^i\rightarrow P^{i+1}\rightarrow\cdots\rightarrow P^j\rightarrow \Im d^j\rightarrow0,\text{ }i<j$$
are $\Hom_A(-,\proj A)$-exact.
\end{itemize}
\end{lemma}

\begin{definition}
[\text{\cite[Page 150]{AR}, \cite[Page 389]{Ha3}}]
An algebra $A$ is called a Gorenstein (or Iwanaga-Gorenstein) algebra if $A$ satisfies $\id\,_A A<\infty$ and $\id A_A <\infty$.
\end{definition}
Then a $K$-algebra $A$ is Gorenstein if and only if $\id A_A<\infty$ and $\pd D(_AA)<\infty$; see  \cite[Page 389]{Ha3}.
Observe that for a Gorenstein algebra $A$, we have $\id _AA=\id A_A$; see \cite[Lemma A]{Za}, \cite[Lemma 6.9]{AR}. A Gorenstein algebra $A$ with $\id_AA$ at most $d$ is called \emph{$d$-Gorenstein}; see \cite[Definition 9.1.9]{EJ}. In the literature there is a different notion of $d$-Gorenstein algebra; see \cite[Page 11]{AR3}.% , and the common value is denoted by $\Gd A$. If $\Gd A\leq d$, we say that $A$ is a \emph{$d$-Gorenstein} algebra.

\begin{definition}
[\text{\cite[Definition 4.2.1]{Bu}; see also \cite[Page 223]{AR2}, \cite[Definition 3.2]{Be}}]
Let $A$ be a Gorenstein algebra. A finitely generated $A$-module $M$ is called (maximal) Cohen-Macaulay if
$$\Ext^i_A(M,A)=0\mbox{ for }i\neq0.$$
\end{definition}

%The full subcategory of Cohen-Macaulay modules in $\mod A$ is denoted by $\CM(A)$.

For a module $M$ take a short exact sequence $0\rightarrow\Omega(M)\rightarrow P\rightarrow M\rightarrow0$ with $P$ projective. The module $\Omega(M)$ is called a \emph{syzygy module} of $M$. Note that syzygy modules of $M$ are not uniquely determined, while they are isomorphic to each other in $\underline{\mod}A$. To avoid confusions, unless otherwise specified, $\Omega(M)$ always means the kernel of the projective cover $P\rightarrow M$, which is uniquely determined (up to isomorphism) in $\mod A$.

\begin{theorem}[\text{\cite[Theorem 3.2]{AM}}]\label{theorem characterize of gorenstein property}
Let $A$ be an algebra and let $d\geq0$. Then the following statements are equivalent:
\begin{itemize}
\item[(i)] the algebra $A$ is $d$-Gorenstein;

\item[(ii)] $\Gproj(A)=\Omega^d(\mod A)$.
\end{itemize}
In this case, A module $G$ is Gorenstein projective if and only if there is an exact sequence $0\rightarrow G\rightarrow P^0\rightarrow P^1\rightarrow \cdots$ with each
$P^i$ projective.
\end{theorem}

From Theorem \ref{theorem characterize of gorenstein property}, it is easy to see that for a Gorenstein algebra, the concept of Gorenstein projective modules coincides with that of Cohen-Macaulay modules.

For an algebra $A$ and $n\geq0$, denote by $\cp^{\leq n}(\mod A)$ the full subcategory of $\mod A$ consisting of modules having projective dimension at most $n$. Denote by $\cp^{<\infty}(\mod A)$ the union of these categories. It is well known that for a $d$-Gorenstein algebra $A$, we have $\cp^{<\infty}(\mod A)=\cp^{\leq d}(\mod A)$; see e.g. \cite[Proposition 9.1.2]{EJ}.

\begin{theorem}\label{theorem cotorsion pair}
Let $A$ be a $d$-Gorenstein algebra. Then we have the following.
\begin{itemize}
%\item[(i)] %\text{\cite[Proposition 9.1.2]{EJ}}
\item[(i)]\text{\cite[Corollary 5.10(1)]{AR}} $\Gproj A$ and $\cp^{\leq d}(\mod A)$ are functorially finite in $\mod A$;
\item[(ii)]\text{\cite[Theorem 1.1]{AB}} For any $M\in\mod A$, there exist the following exact sequences
\begin{align}
0\rightarrow Y_M\rightarrow X_M\rightarrow M\rightarrow0,\qquad
0\rightarrow M\rightarrow Y^M\rightarrow X^M\rightarrow0,
\end{align}
such that $X_M,X^M\in\Gproj A$, $Y_M\in\cp^{\leq d-1}(\mod A)$, $Y^M\in \cp^{\leq d}(\mod A)$.
\end{itemize}
\end{theorem}

%In fact, for a $d$-Gorenstein algebra $A$, $(\Gproj A,\cp^{\leq d}(\mod A))$ is a \emph{(complete) cotorsion pair} in $\mod A$. Here the notion of cotorsion pairs is defined in \cite{Sa}, for complete cotorsion pairs, see \cite{EJ,GT}.

Similarly, for graded algebras, one can define the notions of \emph{graded Gorenstein algebras}, \emph{graded Gorenstein projective modules}.
We denote by $\Gproj^\Z A$ the full subcategory of $\mod^\Z A$ formed by all $\Z$-graded Gorenstein projective modules.

For a positively graded algebra $A$, every (finitely generated) projective modules and injective modules are \emph{gradable}; see \cite[Corollary 3.4]{GG1}. Let $F:\mod^\Z A\rightarrow \mod A$ be the forgetful functor, then for any $M\in\mod^\Z A$, $M$ is graded projective (resp. graded injective) if and only if $F(M)$ is projective (resp. injective) as $A$-module; see \cite[Proposition 1.3, Proposition 1.4]{GG}. In fact,
one has $\gid_A A=\id_A A$ (see \cite[Lemma 3.4]{Le}), where $\gid_A M$ denotes the injective dimension of $M$ in $\mod^\Z A$. %, which is called the \emph{graded injective dimension }of $M$.
Therefore, the positively graded algebra $A$ is graded Gorenstein if and only if it is Gorenstein as an ungraded algebra.
%Certainly, for a finite-dimensional positively graded algebra $A=\bigoplus_{i\geq0} A_i$ which we mainly focus on, $A$ is $d$-graded Gorenstein if and only if $A$ is $d$-Gorenstein.
So we do not distinguish them in the following.

Let $A$ be a graded Gorenstein algebra. A graded $A$-module $M$ is called \emph{graded (maximal) Cohen-Macaulay} if $M$ satisfies that
\begin{align}
\Ext^i_{\mod^\Z A}(M,A(j))=0\mbox{ for }i\neq0,j\in \Z.
\end{align}
We denote by $\CM^\Z(A)$ the full subcategory of $\mod^\Z A$ formed by all $\Z$-graded Cohen-Macaulay modules.
Since for arbitrary graded algebra, every \emph{graded projective module} is a direct summand of a \emph{graded free module}, we get that $\CM^\Z(A)=\,^\bot\proj^\Z (A)$.
It is worth noting that the graded version of Theorem \ref{theorem characterize of gorenstein property} and Theorem \ref{theorem cotorsion pair} are also valid. In particular, $\Gproj^\Z A=\CM^\Z (A)$.%= \,^\bot\proj^\Z (A)$

\begin{lemma}\label{lemma forgetful functor}
Let $A=\bigoplus_{i\geq0} A_i$ be a graded algebra. Then the forgetful functor $F:\mod^\Z A\rightarrow \mod A$ induces a functor from $\Gproj^\Z A$ to $\Gproj A$, which is also denoted by $F$.
\end{lemma}
\begin{proof}
Since $F$ is exact and maps graded projective modules to projective modules, we only need to prove that $F$ induces a functor from $\,^\bot\proj^\Z (A)$ to $\,^\bot\proj A$.

For any $M,N\in \mod^\Z A$, take a projective resolution of $M$:
$$\cdots\xrightarrow{f_{i+1}} P_i\xrightarrow{f_i} \cdots \xrightarrow{f_2} P_1\xrightarrow{f_1}P_0\xrightarrow{f_0}M\rightarrow 0$$ in $\mod^\Z A$ with $P_i$ graded projective. Then for any $j\in\Z$, we get an exact sequence:
$$\Hom_{\mod^\Z A}(P_0,N(j))\rightarrow\Hom_{\mod^\Z A}( \ker f_{0},N(j))\rightarrow\Ext^1_{\mod^\Z A}(M,N(j))\rightarrow0.$$
Proposition \ref{proposition characterize of simple porjective injective modules of graded algebras} implies that $P_i$ is projective in $\mod A$ for any $i$, and then there exists an exact sequence:
$$\Hom_{A}(P_0,N(j))\rightarrow\Hom_{A}( \ker f_{0},N(j))\rightarrow\Ext^1_{A}(M,N(j))\rightarrow0.$$
From (\ref{equation morphism}), we get that
$\Hom_{A}(P_0,N(j))=\bigoplus_{i\in\Z}\Hom_{\mod^\Z A}(P_0,N(i+j))$ and
$$\Hom_{A}(\ker f_0,N(j))=\bigoplus_{i\in\Z}\Hom_{\mod^\Z A}(\ker f_0,N(i+j)).$$
Then
$$\Ext^1_{A}(M,N(j))\cong  \bigoplus_{i\in\Z}\Ext^1_{\mod^\Z A}(M,N(i+j)).$$
By induction, one can check that $$\Ext^k_{A}(M,N(j))\cong  \bigoplus_{i\in\Z}\Ext^k_{\mod^\Z A}(M,N(i+j)),\forall k.$$

Our desired result follows from the above computations.
\end{proof}

\begin{definition}[\text{\cite[Definition 1.2.2]{Bu}; see also \cite[Subsection 4.1]{Ha3}, \cite[Definition 1.8]{Or1}}]
Let $A=\bigoplus_{i\geq0}A_i$ be a graded algebra. The singularity category is defined to be the Verdier localization
$$D_{sg}(A):=D_{sg}(\mod A):=D^b(\mod A)/K^b(\proj A),$$
the $\Z$-graded singularity category is
$$D_{sg}(\mod^\Z A):=D^b(\mod^\Z A)/K^b(\proj^\Z A),$$
We denote by $\pi:D^b(\mod A)\rightarrow D_{sg}(\mod A)$ and $\pi^\Z:D^b(\mod^\Z A)\rightarrow D_{sg}(\mod ^\Z A)$
the localization functor.
\end{definition}

\begin{theorem}[Buchweitz's Theorem, \text{\cite[Theorem 4.4.1]{Bu}}] %see also \cite{KV2} for a more general version]
\label{theorem Buchweitz theorem}
Let $A$ be a graded algebra. Then
\begin{itemize}
\item[(i)] $\Gproj A$ and $\Gproj^\Z A$ are Frobenius categories with the projective modules and $\Z$-graded projective modules as the projective-injective objects respectively.

\item[(ii)] There is an exact embedding $\Phi:\underline{\Gproj}A\rightarrow D_{sg}(\mod A)$ given by $\Phi(M)=M$, where the second $M$ is the corresponding stalk complex at degree $0$, and $\Phi$ is an equivalence if and only if $A$ is Gorenstein.

\item[(iii)] There is an exact embedding $\Phi^\Z:\underline{\Gproj}^\Z A\rightarrow D_{sg}(\mod^\Z A)$ given by $\Phi^\Z(M)=M$, where the second $M$ is the corresponding stalk complex at degree $0$, and $\Phi^\Z$ is an equivalence if and only if $A$ is (graded) Gorenstein.
\end{itemize}
\end{theorem}

\section{Existence of silting objects in $D_{sg}(\mod^\Z A)$%for $1$-Gorenstein positively graded algebras
}\label{section 3}
In this section, we always assume that $A$ is a positively graded Gorenstein algebra.

First, we recall the definition of tilting objects and silting objects for a triangulated category.
Let $\ct$ be a triangulated category. An object $T\in\ct$ is called \emph{tilting} if it satisfies the following conditions:
\begin{itemize}
\item[(a)] $\Hom_\ct(T,T[i])=0$ for any $i\neq0$.
\item[(b)] $\ct=\thick_\ct T$.
\end{itemize}

Recall that an object $T$ of $\ct$ is called a \emph{silting object}  if it generates $\ct$ and
$\Hom_\ct(T,T[i])=0$ for any $i>0$; see \cite[Subsection 5.1]{KV} and \cite[Definition 2.1]{AI}.

Let $\ct$ be an algebraic triangulated Krull-Schmidt category. If $\ct$ has a tilting object $T$, then there exists a triangle equivalence $\ct\simeq K^b(\proj(\End_\ct(T)^{op}))$; see \cite[Theorem 8.7]{Ke1}.

In this section, we prove that there exists a silting object in $D_{sg}(\mod^\Z A)$ for a positively graded $1$-Gorenstein algebra $A$. First, similar to \cite[Theorem 3.1]{Ya}, we get the following lemma.
\begin{lemma}\label{lemma finite global dimension}
Let $A$ be a positively graded Gorenstein algebra.
If $D_{sg}(\mod^\Z A)$ has tilting objects, then $A_0$ has finite global dimension.
\end{lemma}
\begin{proof}
We assume that $A$ is a $d$-Gorenstein algebra.

Suppose for a contradiction that $D_{sg}(\mod^\Z A)$ has a tilting object $T$ and $\gldim A_0=\infty$.
As $A$ is positively graded, for any $A_0$-module $X$, the degree $0$ part of a projective resolution of $X$ in $\mod^\Z A$ gives a projective resolution of $X$ in $\mod A_0$.
So we have $$\Ext_{\mod^\Z A}^i(A_0/J_{A_0},A_0/J_{A_0})=\Ext_{A_0}^i(A_0/J_{A_0},A_0/J_{A_0})$$ for any $i>0$.
Because $\gldim A_0=\infty$, the $A_0$-module $A_0/J_{A_0}$ has infinite projective dimension. Thus we have
$$\Ext_{A_0}^i(A_0/J_{A_0},A_0/J_{A_0})\neq0\mbox{ for any }i>0.$$

Denote by $\Omega$ the syzygy translation of $\mod^\Z A$. As $A$ is a $d$-Gorenstein algebra, it follows from the graded version of Theorem \ref{theorem characterize of gorenstein property} that $\Omega^j( A_0/J_{A_0})$ is a graded Gorenstein projective $A$-module for any $j\geq d$. Let
$$\cdots\rightarrow Q_{d} \xrightarrow{f_d} Q_{d-1} \xrightarrow{f_{d-1}} \cdots\xrightarrow{f_1} Q_0\xrightarrow{f_0} A_0/J_{A_0}\rightarrow0$$
be a projective resolution of $A_0/J_{A_0}$ as graded $A$-module. Let $\Sigma$ be the shift functor of $D^b(\mod^\Z A)$, and $\langle1\rangle$ be the shift functor of $D_{sg}(\mod^\Z A)$ which is induced by $\Sigma$.
It is not hard to see that $A_0/J_{A_0}\simeq \Omega^j(A_0/J_{A_0})\langle j\rangle$ for any $j\geq0$ in $D_{sg}(\mod^\Z A)$. In particular, $A_0/J_{A_0}\simeq \Omega^d(A_0/J_{A_0})\langle d\rangle$. So
\begin{align*}
 \Hom_{D_{sg}(\mod^\Z A)}(A_0/J_{A_0}, (A_0/J_{A_0})\langle i\rangle)
=& \Hom_{D_{sg}(\mod^\Z A)}(\Omega^d(A_0/J_{A_0})\langle d\rangle, \Omega^d(A_0/J_{A_0})\langle d+i\rangle)\\
=& \Hom_{D_{sg}(\mod^\Z A)}(\Omega^d(A_0/J_{A_0}), \Omega^d(A_0/J_{A_0})\langle i\rangle),
\end{align*}
for any $i$.
Buchweitz's Theorem shows that $\underline{\Gproj}^\Z(A) \simeq D_{sg}(\mod^\Z A)$. As $\underline{\Gproj}^\Z(A)$ is a full subcategory of $\underline{\mod}^\Z(A)$, we have
\begin{align*}
\Hom_{D_{sg}(\mod^\Z A)}(\Omega^d(A_0/J_{A_0}), \Omega^d(A_0/J_{A_0})\langle i\rangle)
&=\Hom_{\underline{\Gproj}^\Z(A)}(\Omega^d(A_0/J_{A_0}), \Omega^d(A_0/J_{A_0})[i])\\
&=\Hom_{\underline{\mod}^\Z(A)}(\Omega^d(A_0/J_{A_0}), \Omega^d(A_0/J_{A_0})[i]),
\end{align*}
where $[1]$ is the shift functor of $\underline{\Gproj}^\Z(A)$ given by the cosyzygy functor.
Since $\Omega^d(A_0/J_{A_0})$ is a graded Gorenstein projective $A$-module, there exists an exact sequence
$$0\rightarrow \Omega^d(A_0/J_{A_0}) \xrightarrow{g_0} P_0 \xrightarrow{g_1} P_1\xrightarrow{g_2}\cdots\rightarrow P_i\rightarrow\cdots$$
with $P_i$ graded projective and $\coker (g_i)$ graded Gorenstein projective for any $i\geq0$.
From $$\Ext^i_{\mod^\Z A}(\Omega^d(A_0/J_{A_0}),P)=0$$ for any projective module $P$ and $i>0$, we deduce that
\begin{eqnarray*}
&&\Hom_{\underline{\mod}^\Z(A)}(\Omega^d(A_0/J_{A_0}), \Omega^d(A_0/J_{A_0})[i])=\Ext^i_{\mod^\Z A}(\Omega^d(A_0/J_{A_0}), \Omega^d(A_0/J_{A_0}) ).
\end{eqnarray*}

Similarly, for any $i>d$,
\begin{align*}
\Ext^i_{\mod^\Z A}(\Omega^d(A_0/J_{A_0}), \Omega^d(A_0/J_{A_0}) )
=&\Ext^{i-1}_{\mod^\Z A}(\Omega^d(A_0/J_{A_0}), \Omega^{d-1}(A_0/J_{A_0}) )\\
=&\Ext^{i-d}_{\mod^\Z A}(\Omega^d(A_0/J_{A_0}), A_0/J_{A_0})\\
=&\Ext^{i-d+1}_{\mod^\Z A}(\Omega^{d-1}(A_0/J_{A_0}), A_0/J_{A_0})\\
=&\Ext^{i}_{\mod^\Z A}(A_0/J_{A_0}, A_0/J_{A_0}).
\end{align*}
So
\begin{align*}
\Hom_{D_{sg}(\mod^\Z A)}(A_0/J_{A_0}, (A_0/J_{A_0})\langle i\rangle)
=&\Hom_{D_{sg}(\mod^\Z A)}(\Omega^d (A_0/J_{A_0})\langle d\rangle, \Omega^d(A_0/J_{A_0})\langle d+ i\rangle)\\
=&\Hom_{D_{sg}(\mod^\Z A)}(\Omega^d (A_0/J_{A_0}), \Omega^d(A_0/J_{A_0})\langle i\rangle)\\
=&\Hom_{ \underline{\Gproj}^\Z(A)}(\Omega^d (A_0/J_{A_0}), \Omega^d(A_0/J_{A_0})[i])\\
=&\Ext^{i}_{\mod^\Z A}(A_0/J_{A_0}, A_0/J_{A_0})\\
=&\Ext_{A_0}^i(A_0/J_{A_0},A_0/J_{A_0})\neq0
\end{align*}
for any $i>d$.
However, as $D_{sg}(\mod^\Z A)$ has a tilting object, \cite[Proposition 2.4]{AI} shows that
$\Hom_{D_{sg}(\mod^\Z A)}(X,Y\langle i\rangle)=0$ holds for any $X,Y\in D_{sg}(\mod^\Z A)$ and $|i|\gg0$. This gives a contradiction. Then  $A_0$ has finite global dimension.
\end{proof}

Let $T$ be the graded $A$-module $T:=\bigoplus_{i\geq0} A(i)_{\leq0}$ defined as in (\ref{def:T}).
Since $A(i)_{\leq 0}=A(i)$ is zero in $D_{sg}(\mod^\Z A)$ for sufficiently large $i$, we can regard $T$ as an object in $D_{sg}(\mod^\Z A)$.

\begin{lemma}\label{lemma T is a generator}
Let $A$ be a positively graded Gorenstein algebra.
If $A_0$ has finite global dimension, then we have $D_{sg}(\mod^\Z A)=\thick T$, where $T= \bigoplus_{i\geq0} A(i)_{\leq0}$.
\end{lemma}
\begin{proof}
The proof is based on that of \cite[Lemma 3.5]{Ya}.

For any object $N$ in $D_{sg}(\mod^\Z A)$, there exists a graded Gorenstein projective $A$-module such that it is isomorphic to $N$ in $D_{sg}(\mod^\Z A)$. So it is enough to show that all the graded $A$-modules are in the subcategory $\thick T$.

We regard $A_0$ as the natural $\Z$-graded factor $A$-module, i.e., $A_0(0)=A(0)_{\leq0}$. Any object in $\mod^\Z(A)$ has a finite filtration by simple objects in $\mod^\Z(A)$ which are given by simple $A_0$-modules concentrated in some degree
by Proposition \ref{proposition characterize of simple porjective injective modules of graded algebras}. Every short exact sequence in $\mod^\Z A$ gives a triangle in $D^b(\mod^\Z A)$, and then a triangle in $D_{sg}(\mod^\Z A)$, so we only need to check that $S(i)\in \thick T$ for any simple $A_0$-module $S$, and any $i\in\Z$.
As $A_0$ has finite global dimension, it is enough to show that $A_0(i)\in \thick T$ for any $i\in\Z$.
We divide the proof into two parts.

(i) We show that $A_0(i)\in \thick T$ for any $i\geq 0$ by induction on $i$. Obviously we have $A_0(0)=A(0)_{\leq0}\in \thick T$.
We assume $A_0(0),\dots,A_0(i-1)\in \thick T$. This implies that $S(0),\dots,S(i-1)\in \thick T$ for any simple $A_0$-module $S$ since $A_0$ has finite global dimension. For any graded $A$-module $N$, it is easy to see that if $N_j=0$ for any $j>0$ and $j<1-i$ then $N\in \thick T$.

Obviously, there exists a short exact sequence
\begin{equation}\label{eqn:short exact sequence}
0\rightarrow (A(i)_{\leq0})_{\geq 1-i} \rightarrow A(i)_{\leq0}\rightarrow A_0(i)\rightarrow0,
\end{equation}
in $\mod^\Z A$. By the above, we have $(A(i)_{\leq0})_{\geq 1-i}\in\thick T$, and then  $A_0(i)\in \thick T$ by (\ref{eqn:short exact sequence}).

(ii) We show that $A_0(-i)\in\thick T$ for any $i\geq1$. Assume $A_0(-j)\in\thick T$ for any $1\leq j<i$. Together with (i), we get that $S(-j)\in \thick T$ for any simple $A_0$-module $S$ and $j<i$ since $A_0$ has finite global dimension. It implies that for any $N\in\mod^\Z A$, if $N_{\geq i}=0$ then $N\in \thick T$.
Denote by $n:=\pd_{A_0^{op}} D(A_0)+1$.
There exists an exact sequence
$$0\rightarrow X\rightarrow Q^{n-1}\rightarrow \cdots \rightarrow Q^1\rightarrow Q^0
\rightarrow D(A_0)\rightarrow0$$
in $\mod^\Z A^{op}$ such that $Q^j$ is a $\Z$-graded projective $A^{op}$-modules for $0\leq j\leq n-1$, and $X_{\leq 0}=0$.
So there exists an exact sequence
$$0\rightarrow A_0\rightarrow D(Q^0)\rightarrow D(Q^1)\rightarrow\cdots\rightarrow D(Q^{n-1})\rightarrow D(X)\rightarrow0$$
in $\mod^\Z A$ such that $D(Q^j)$ is a $\Z$-graded injective $A$-module for $0\leq j\leq n-1$, and $(D(X))_{\geq0}=0$. Since $A$ is Gorenstein, we have $D(Q^j)\in K^b(\proj^\Z A)$ for $0\leq j\leq n-1$ and then $A_0(-i)=D(X)(-i)\langle-n\rangle$ in $D_{sg}(\mod^\Z A)$. It follows from $(D(X))_{\geq0}=0$  that $(D(X)(-i))_{\geq i}=0$ and then $D(X)(-i)\in \thick T$ by the inductive hypothesis. Therefore, $A_0(-i)\in \thick T$.
\end{proof}

Now we prove the main result of this section.
\begin{theorem}\label{theorem silting object}
Let $A$ be a positively graded $1$-Gorenstein algebra. If $A_0$ has finite global dimension, then $D_{sg}(\mod^\Z A)$ admits a silting object.
\end{theorem}
\begin{proof}

Let $T=\bigoplus_{i\geq0} A(i)_{\leq0}$, which is viewed as an object in $D_{sg}(\mod^\Z A)$. In the following, we prove that $T$ is a silting object in $D_{sg}(\mod^\Z A)$.

We take a minimal projective resolution
$$ \cdots \rightarrow P_2\rightarrow P_1\rightarrow A(i)\rightarrow A(i)_{\leq 0}\rightarrow0$$
of $A(i)_{\leq 0}$ in $\mod^\Z A$. Because $A$ is positively graded, we have $(P_j)_{\leq0}=0$ for $j>0$. Then
$(\Omega^j T)_{\leq0}=0$ holds for any $j>0$. It follows from $T=T_{\leq0}$ that $\Hom_{\mod^\Z A}(T,\Omega^j T)=0=\Hom_{\mod^\Z A}(\Omega^j T,T)$ for any $j>0$.

Let $0\rightarrow \Omega(T)\xrightarrow{f_0} P\xrightarrow{f_1} T\rightarrow0$
be the exact sequence.
Since $A$ is $1$-Gorenstein, $\Omega T\in \Gproj^\Z(A)$.
Similar to the proof of Lemma \ref{lemma finite global dimension}, for any $i>1$,
\begin{align*}
\Hom_{\underline{\Gproj}^\Z(A)}(\Omega T,\Omega T[i])
=&\Ext^i_{\mod^\Z A}(\Omega T,\Omega T)\\
=&\Ext^{i}_{\mod^\Z A}(T,T)\\
=&\Ext_{\mod^\Z A}^{i-1}(\Omega T,T)\\
=&\Ext^1_{\mod^\Z A}(\Omega^{i-1}T,T)\\
=&\Hom_{\underline{\mod}^\Z(A)}(\Omega^iT,T)=0.
\end{align*}
For $i=1$, $\Hom_{\underline{\Gproj}^\Z(A)}(\Omega T,\Omega T[1])= \Ext^1_{\mod^\Z A}(\Omega T,\Omega T)=\Hom_{\underline{\mod}^\Z(A)}(\Omega T,T)=0$ since $\Ext^1_{\mod^\Z A}(\Omega(T),P)=0$.
Combining the above, we have $\Hom_{\underline{\Gproj}^\Z(A)}(\Omega T,\Omega T[i])=0$ for any $i>0$.

Together with Lemma \ref{lemma T is a generator}, we deduce that $\Omega T$ is a silting object in $D_{sg}(\mod^\Z A)$. Because $T\cong \Omega T\langle 1\rangle$ in $D_{sg}(\mod^\Z A)$, $T$ is also a silting object in $D_{sg}(\mod^\Z A)$.
\end{proof}

At the end of this section, we consider the case when $T=\bigoplus_{i\geq0} A(i)_{\leq0}$ is a Gorenstein projective $A$-module. Let $\Gamma$ be the endomorphism ring  of $T$ in $\underline{\Gproj}^{\Z} A$.

\begin{proposition}\label{corollary for T CM}
Let $A$ be a positively graded Gorenstein algebra such that $A_0$ has finite global dimension. If $T= \bigoplus_{i\geq0} A(i)_{\leq0}$ is graded Gorenstein projective, then $T$ is a tilting object in $D_{sg}(\mod^\Z A)$.
\end{proposition}
\begin{proof}
As $A$ is a Gorenstein algebra, it follows that $\underline{\Gproj}^\Z(A)$ is equivalent to $D_{sg}(\mod^\Z A)$. Since $T$ is Gorenstein projective,
$$\Hom_{D_{sg}(\mod^\Z A)}(T,T\langle i\rangle)=\Hom_{\underline{\Gproj}^\Z(A)}(T,\Omega^{-i} T )=\Hom_{\underline{\Gproj}^\Z(A)}(\Omega^{i} T,T )$$ for any $i\in\Z$. From the proof of Theorem \ref{theorem silting object}, we have $$\Hom_{\mod^\Z A}(T,\Omega^j T)=0=\Hom_{\mod^\Z A}(\Omega^j T,T)$$ for any $j>0$. So $\Hom_{D_{sg}(\mod^\Z A)}(T,T\langle i\rangle)=0$ for any $i\neq0$, and then $ T$ is a tilting object in $D_{sg}(\mod^\Z A)$ by Lemma \ref{lemma T is a generator}.
\end{proof}

Let $A$ be a positively graded Gorenstein algebra. We say that $A$ has \emph{Gorenstein parameter} $l$ if $\soc A$ is contained in $A_l$.

\begin{theorem}[cf. \text{\cite[Proposition 3.6, Theorem 3.7]{Ya}}]\label{theorem endomorphism ring of gorenstein projective modules}
Let $A$ be a positively graded Gorenstein algebra of Gorenstein parameter $l$.  Assume that $T=\bigoplus_{i\geq0} A(i)_{\leq0}$ is a Gorenstein projective $A$-module.
Take a decomposition $T=\underline{T}\oplus P$ in $\Mod^\Z A$ where $\underline{T}$ is a direct sum of all indecomposable non-projective direct summand of $T$. Then
\begin{itemize}
\item[(i)] $\underline{T}$ is finitely generated, and is isomorphic to $T$ in $\underline{\Gproj}^{\Z} A$.

\item[(ii)] There exists an algebra isomorphism $\Gamma\simeq \End_A(\underline{T})_0$.

\item[(iii)] If $A_0$ has finite global dimension, then so does $\Gamma$.
\end{itemize}
\end{theorem}
\begin{proof}
The proof is the same to that of \cite[Proposition 3.6, Theorem 3.7]{Ya}, we omit it here.
\end{proof}

Similar to \cite[Section 3.2]{Ya}, we assume that $A$ is a positively graded Gorenstein algebra of Gorenstein parameter $l$. Let $U:=\bigoplus_{i=0}^{l-1} A(i)_{\leq0}$. Then
there exists an algebra isomorphism
\begin{equation}\label{equation 1}
\End_A(U)_0\simeq \left(\begin{array}{ccccccc} A_0&A_1&\cdots &A_{l-2} &A_{l-1} \\
&A_0& \cdots & A_{l-3} &A_{l-2}\\
&& \ddots &\vdots&\vdots\\
&&& A_0&A_1\\
&&&&A_0
\end{array}  \right).
\end{equation}
If we decompose $U=\underline{T}\oplus P'$ for some projective direct summand of $U$, then
\begin{equation}\label{equation 2}
\End_A(U)_0\simeq \left( \begin{array}{ccc} \End_A(P')_0 & \Hom_A(\underline{T},P')_0\\
0& \Gamma \end{array} \right).
\end{equation}
Then $\Gamma$ is a quotient algebra of $\End_A(U)_0$.

We give some examples about cluster-tilted algebras, which are $1$-Gorenstein by \cite[Proposition 2.1]{KR}.

\begin{example}
Let $A$ be a cluster-tilted algebra of type $\E_6$ with its quiver as Fig. 1. shows.

\begin{center}\setlength{\unitlength}{0.8mm}
 \begin{picture}(170,45)
 \put(0,0){\begin{picture}(80,40)
\put(10,20){\circle{1.5}}
\put(50,20){\circle{1.5}}
\put(20,30){\circle{1.5}}
\put(40,30){\circle{1.5}}
\put(20,10){\circle{1.5}}
\put(40,10){\circle{1.5}}

\put(12,20){\vector(1,0){36}}
\put(19,29){\vector(-1,-1){8}}
\put(19,11){\vector(-1,1){8}}
\put(38,30){\vector(-1,0){16}}
\put(38,10){\vector(-1,0){16}}

\put(49,21){\vector(-1,1){8}}
\put(49,19){\vector(-1,-1){8}}

\put(7,20){$_{1}$}
\put(19,30){$^{2}$}
\put(39,30){$^{3}$}
\put(19,7){$_{4}$}
\put(39,7){$_{5}$}
\put(51,20){$_{6}$}

\put(12,25){$^{\alpha_1}$}
\put(12,13){$_{\beta_1}$}

\put(28,30){$^{\alpha_2}$}
\put(28,5){$^{\beta_2}$}

\put(45,25){$^{\alpha_3}$}
\put(45,13){$_{\beta_3}$}

\put(28,15){$^{\gamma}$}
\put(0,-5){Fig. 1. A cluster-tilted algebra}
\put(17,-11){of type $\E_6$.}
\end{picture}}

 \put(60,0){\begin{picture}(80,40)
\put(10,20){\circle{1.5}}
\put(50,20){\circle{1.5}}
\put(20,30){\circle{1.5}}
\put(40,30){\circle{1.5}}
\put(20,10){\circle{1.5}}
\put(40,10){\circle{1.5}}
\qbezier[45](11,20)(30,20)(49,20)

\put(19,29){\vector(-1,-1){8}}
\put(19,11){\vector(-1,1){8}}
\put(38,30){\vector(-1,0){16}}
\put(38,10){\vector(-1,0){16}}

\put(49,21){\vector(-1,1){8}}
\put(49,19){\vector(-1,-1){8}}

\put(7,20){$_{1}$}
\put(19,30){$^{2}$}
\put(39,30){$^{3}$}
\put(19,7){$_{4}$}
\put(39,7){$_{5}$}
\put(51,20){$_{6}$}

\put(12,25){$^{\alpha_1}$}
\put(12,13){$_{\beta_1}$}

\put(28,30){$^{\alpha_2}$}
\put(28,5){$^{\beta_2}$}

\put(45,25){$^{\alpha_3}$}
\put(45,13){$_{\beta_3}$}
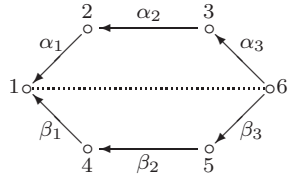
\put(10,-5){Fig. 2. Tilted algebra $B_1$}
\put(27,-11){of type $\E_6$.}
\end{picture}}

 \put(120,0){\begin{picture}(80,40)
\put(10,20){\circle{1.5}}
\put(50,20){\circle{1.5}}
\put(20,30){\circle{1.5}}
\put(40,30){\circle{1.5}}
\put(20,10){\circle{1.5}}
\put(40,10){\circle{1.5}}

\qbezier[25](21,30)(30,30)(39,30)
\qbezier[25](21,10)(30,10)(39,10)

\put(12,20){\vector(1,0){36}}
\put(19,29){\vector(-1,-1){8}}
\put(19,11){\vector(-1,1){8}}

\put(49,21){\vector(-1,1){8}}
\put(49,19){\vector(-1,-1){8}}

\put(7,20){$_{1}$}
\put(19,30){$^{2}$}
\put(39,30){$^{3}$}
\put(19,7){$_{4}$}
\put(39,7){$_{5}$}
\put(51,20){$_{6}$}

\put(12,25){$^{\alpha_1}$}
\put(12,13){$_{\beta_1}$}

\put(45,25){$^{\alpha_3}$}
\put(45,13){$_{\beta_3}$}

\put(28,15){$^{\gamma}$}
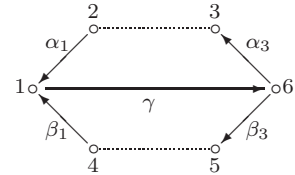
\put(10,-5){Fig. 3. Tilted algebra $B_2$}
\put(27,-11){of type $\E_6$.}
\end{picture}}
\end{picture}
\vspace{1cm}
\end{center}

(a) Let $B_1$ be the algebra where its quiver is as Fig. 2. shows with $\alpha_1\alpha_2\alpha_3-\beta_1\beta_2\beta_3=0$. Then $B_1$ is a tilted algebra, and $A\cong B_1\ltimes \Ext^2_{B_1}(DB_1,B_1) $; see \cite[Theorem 1.1]{ABS1} and \cite[Proposition 3.1]{Zh}. So $A$ is a positively graded algebra with $A_0=B_1$ and $A_1=\Ext^2_{B_1}(DB_1,B_1)$. It is easy to check that $T=\oplus_{i\geq0} A(i)_{\leq0}\cong B_1$ is a tilting object in $D_{sg}(\mod^\Z A)$.

(b) Let $B_2$ be the algebra where its quiver is as Fig. 3. shows with $\alpha_3\gamma\alpha_1=0,\beta_3\gamma\beta_1=0$. Then $B_2$ is a tilted algebra, and $A\cong B_2\ltimes \Ext^2_{B_2}(DB_2,B_2) $. So $A$ is a positively graded algebra with $A_0=B_2$ and $A_1=\Ext^2_{B_2}(DB_2,B_2)$. It is easy to check that $B_2$ is a tilting object in $D_{sg}(\mod^\Z A)$.
\end{example}

However, in general case, $T$ is not a tilting object in $D_{sg}(\mod^\Z A)$ for a $1$-Gorenstein algebra $A$.

\begin{example}
Let $A=KQ/I$ be the cluster-tilted algebra of type $\D_6$ with $Q$ as the following figure shows. Then $A$ is a positively graded algebra by setting $\alpha_4,\beta_1,\gamma_2$ to be degree one, and other arrows to be degree zero. Let $T= \bigoplus_{i\geq0} A(i)_{\leq0}$. It is easy to see that $T$ is a silting object (which is not a tilting object) in $D_{sg}(\mod^\Z A)$.

\begin{center}\setlength{\unitlength}{0.8mm}
 \begin{picture}(60,45)

\put(10,20){\circle{1.6}}
\put(10,40){\circle{1.6}}
\put(30,20){\circle{1.6}}
\put(30,40){\circle{1.6}}
\put(20,10){\circle{1.6}}
\put(40,30){\circle{1.6}}
\put(10,38){\vector(0,-1){16}}

\put(28,40){\vector(-1,0){16}}
\put(12,20){\vector(1,0){16}}
\put(30,22){\vector(0,1){16}}

\put(31,39){\vector(1,-1){8}}

\put(39,29){\vector(-1,-1){8}}
\put(29,19){\vector(-1,-1){8}}

\put(19,11){\vector(-1,1){8}}

\put(5,29){$^{\alpha_1}$}
\put(12,13){$_{\beta_2}$}

\put(25,29){$^{\alpha_3}$}
\put(25,10){$^{\beta_1}$}

\put(36,34){$^{\gamma_1}$}
\put(36,24){$_{\gamma_2}$}

\put(18,15){$^{\alpha_2}$}

\put(18,40){$^{\alpha_4}$}

\put(-20,0){Fig. 4. A cluster-tilted algebra of type $\D_6$.}
\end{picture}
\end{center}

\end{example}

\section{Realising graded singularity categories of Gorenstein monomial algebras as derived categories}
\label{section 4}
In this section, we prove that the graded singularity category of a Gorenstein monomial algebra admits a tilting object, and it is triangulated equivalent to the derived category of a hereditary algebra of finite representation type.

\subsection{Gorenstein projective modules over monomial algebras}
In this subsection, we recall the description of Gorenstein projective modules over monomial algebras obtained in \cite[Section 3]{CSZ}.

Let $Q=(Q_0,Q_1,s,t)$ be a finite quiver. A \emph{path $p$ of length $n$} in $Q$ is a sequence $p=\alpha_n\cdots \alpha_2\alpha_1$ of arrows with $s(\alpha_i)=t(\alpha_{i-1})$ for $2\leq i\leq n$, its starting vertex is $s(p)=s(\alpha_1)$ and its ending vertex is $t(p)=t(\alpha_n)$. For each vertex $i$, we associate a {\em trivial} path $e_i$ of length zero with $s(e_i)=i=t(e_i)$. A nontrivial path is called an \emph{oriented cycle} if its starting vertex coincides with its ending vertex.

For two paths $p$ and $q$ with $s(p)=t(q)$, we write $pq$ for their concatenation. Let $p$ and $q$ be two paths. $p$ is called to be a subpath of $q$ if $q=p''pp'$ for some paths $p''$ and $p'$. Let $\cs$ be a set of paths in $Q$. A path $p$ in $\cs$ is \emph{left-minimal} in $\cs$ if there is no path $q$ such that $q\in \cs$ and $p=qp'$ for some nontrivial path $p'$. Dually one defines a \emph{right-minimal path} in $\cs$. A path $p$ in $\cs$ is \emph{minimal} in $\cs$ if there is no proper subpath $q$ of $p$ inside $\cs$.

Let $KQ$ be the path algebra of $Q$. An admissible ideal $I$ of $KQ$ is \emph{monomial} if it is generated by some paths of lengths at least two. In this case, the quotient algebra $A=KQ/I$ is called a \emph{monomial algebra}.

Let $A=KQ/I$ be a monomial algebra. Let $\bF$ be the set formed by all the minimal paths among the paths in $I$. A path is said to be a \emph{nonzero path in $A$} provided that $p$ does not belong to $I$, or equivalently, $p$ does not contain a subpath in $\bF$. Then the set of nonzero paths forms a $K$-basis of $A$.

For a nonzero path $p$, we consider the left ideal $Ap$ and the right ideal $pA$. Note that $Ap$ has a basis given by all nonzero paths $q$ such that $q=q'p$ for some path $q'$, and similar for $pA$. For a nontrivial path $p$ which is nonzero in $A$, we define $L(p)$ to be the set of right-minimal paths in the set $\{\text{nonzero paths } q\mid s(q)=t(p)\text{ and }qp=0\}$ and $R(p)$ the set of left-minimal paths in the set $\{\text{nonzero paths }q\mid t(q)=s(p)\text{ and }pq=0\}$.

\begin{definition}[\text{\cite[Definition 3.3]{CSZ}}]\label{definition of perfect pair}
Let $A=KQ/I$ be a monomial algebra. We call a pair $(p,q)$ of nonzero paths in $A$ perfect provided that the following conditions are satisfied:
\begin{itemize}
\item[(P1)] both of the nonzero paths $p,q$ are nontrivial, which satisfy $s(p)=t(q)$ and $pq=0$ in $A$;
\item[(P2)] if $pq'=0$ for a nonzero path $q'$ with $t(q')=s(p)$, then $q'=qq''$ for some path $q''$; in other words, $R(p)=\{q\}$;
\item[(P3)] if $p'q=0$ for a nonzero path $p'$ with $s(p')=t(q)$, then $p'=p''p$ for some path $p''$; in other words, $L(q)=\{p\}$.
\end{itemize}
\end{definition}
For a perfect pair $(p,q)$, from the first paragraph of \cite[Page 162]{Zi}, we have the following exact sequence of left $A$-modules
\begin{equation}\label{the syzygy functor of Gorenstein projective modules}
0\rightarrow Ap\xrightarrow{inc} A_{e_{t(q)}}\xrightarrow{\pi_q}Aq\rightarrow0.
\end{equation}
In particular, $\Omega(Aq)\simeq Ap$.

\begin{definition}[\text{\cite[Defition 3.7]{CSZ}}]\label{definition of perfect path}
Let $A=KQ/I$ be a monomial algebra. We call a nonzero path $p$ in $A$ a perfect path, if there exists a sequence
$$p=p_1,p_2,\dots,p_n,p_{n+1}=p$$
of nonzero paths such that $(p_i,p_{i+1})$ are perfect pairs for all $1\leq i\leq n$. If the given nonzero paths $p_i$ are pairwise distinct, we refer to the sequence $p=p_1,p_2,\dots,p_n,p_{n+1}=p$ as a relation-cycle for $p$.
\end{definition}

Note that a perfect path has a unique relation-cycle.
%By a \emph{basic $(n-)$cycle} $Z_n$, we mean a quiver consisting of $n$ vertices and $n$ arrows which form an oriented cycle.

The following theorem characterizes the indecomposable Gorenstein projective modules for a monomial algebra clearly.

\begin{theorem}[\text{\cite[Theorem 4.1]{CSZ}}]\label{theorem bijection of perfect path and Gorenstein projective modules}
Let $A$ be a monomial algebra. Then there is a bijection
\begin{equation}\label{equation characterization of gorenstein projective modules}
\{\text{perfect paths in }A\}\stackrel{1:1}{\longleftrightarrow}\Ind \underline{\Gproj}A
\end{equation}
sending a perfect path $p$ to the $A$-module $Ap$.
\end{theorem}

For a monomial algebra $A$, by \cite[Chapter III.2, Lemma 2.4, Lemma 2.6]{ASS}, its indecomposable projective modules (also Gorenstein projective modules) and injective modules are as Fig. 5. and Fig. 6. show respectively.

\setlength{\unitlength}{0.8mm}
\begin{center}
\begin{picture}(100,55)(0,-20)
\put(-10,0){\begin{picture}(50,50)
\put(10,30){\circle*{1.5}}
\put(0,20){\circle*{1.5}}
\put(20,20){\circle*{1.5}}

\put(-5,10){\circle*{1.5}}
\put(5,10){\circle*{1.5}}
\put(25,10){\circle*{1.5}}
\put(15,10){\circle*{1.5}}

\put(9,29){\vector(-1,-1){8}}
\put(11,29){\vector(1,-1){8}}

\put(-1,19){\vector(-1,-2){4}}
\put(1,19){\vector(1,-2){4}}

\put(19,19){\vector(-1,-2){4}}
\put(21,19){\vector(1,-2){4}}

\put(7,18.5){$\cdots$}
\put(-3,8.5){$\cdots$}
\put(17,8.5){$\cdots$}
\put(-6,2){$\vdots$}
\put(4,2){$\vdots$}
\put(14.5,2){$\vdots$}
\put(24.5,2){$\vdots$}

\put(-25,-7){Fig. 5. The structure of indecomposable}
\put(-9,-13){projective modules.}
\end{picture}}

\put(85,0){\begin{picture}(50,50)
\put(7,8.5){$\cdots$}

\put(-3,18.5){$\cdots$}
\put(17,18.5){$\cdots$}

\put(-6,22){$\vdots$}
\put(4,22){$\vdots$}
\put(14.5,22){$\vdots$}
\put(24.5,22){$\vdots$}

\put(-5,20){\circle*{1.5}}
\put(5,20){\circle*{1.5}}
\put(25,20){\circle*{1.5}}
\put(15,20){\circle*{1.5}}

\put(0,10){\circle*{1.5}}
\put(20,10){\circle*{1.5}}

\put(-4.5,19){\vector(1,-2){4}}
\put(4.5,19){\vector(-1,-2){4}}

\put(15.5,19){\vector(1,-2){4}}
\put(24.5,19){\vector(-1,-2){4}}

\put(19,9){\vector(-1,-1){8}}
\put(1,9){\vector(1,-1){8}}

\put(10,0){\circle*{1.5}}

\put(-25,-7){Fig. 6. The structure of indecomposable}
\put(-9,-13){injective modules.}\end{picture}}
\end{picture}

\end{center}

\subsection{Tilting objects in $D_{sg}(\mod^\Z A)$ of Gorenstein monomial algebras}
For a monomial algebra $A=KQ/I$, it is easy to see that $A$ is a positively graded algebra by setting each arrow to be degree one. In the following, we always consider $A$ to be positively graded in this way.

\begin{lemma}\label{lemma structure of Gorenstein projective module}
Let $A=KQ/I$ be a monomial algebra. Then the forgetful functor $F:\Gproj^\Z A\rightarrow \Gproj A$ is dense. In particular, for any indecomposable graded Gorenstein projective module $X$, we have that $\Top(X)$ is simple.
\end{lemma}
\begin{proof}
From Fig. 5., we have that any Gorenstein projective module is gradable, and $F:\Gproj^\Z A\rightarrow \Gproj A$ is dense. For the last statement, obviously, $X\cong Ap$ for some perfect path $p$, and then $\Top(X)$ is isomorphic to the simple module corresponding to the vertex $t(p)$.
\end{proof}

For an indecomposable Gorenstein projective $A$-module $X$, we always view it as a graded Gorenstein projective module with $\Top(X)$ concentrated in degree zero.
In this way, Lemma \ref{lemma structure of Gorenstein projective module}.
implies that $\Ind \underline{\Gproj}^\Z A=\bigcup_{i\in\Z}(\Ind\underline{\Gproj}A)(i)$.

\begin{theorem}\label{main theorem}
Let $A=KQ/I$ be a Gorenstein monomial algebra. Then $D_{sg}(\mod^\Z A)$ has a tilting object.
\end{theorem}

\begin{proof}
We prove it by constructing a tilting object in $D_{sg}(\mod^\Z A)$ similar to Theorem \ref{theorem silting object}.

Let $l$ be a positive integer such that $A=A_{\leq l}$. By abusing notations, we define a graded $A$-module
$T:=\bigoplus_{0\leq i\leq l} A(i)_{\leq0}.$
We shall show that $T$ is a tilting object in $D_{sg}(\mod^\Z A)$ in the following.
By noting that
$T\cong \bigoplus_{i\geq 0} A(i)_{\leq0}$ in $D_{sg}(\mod^\Z A)$, we have $\thick (T)=D_{sg}(\mod^\Z A)$ by Lemma \ref{lemma T is a generator} since $A_0$ is semisimple.

Denote by $\langle 1\rangle$ the shift functor in $D_{sg}(\mod^\Z A)$. It is enough to check the following equalities in $D_{sg}(\mod^\Z A)$:
\begin{align}
\label{eq:tiltingp<0}
&\Hom_{D_{sg}(\mod^\Z A)}(T,T\langle -p\rangle)=0;\\
\label{eq:tiltingp>0}
&\Hom_{D_{sg}(\mod^\Z A)}(T,T\langle p\rangle)=0
\end{align}
for any $p>0$. We shall use Theorem \ref{theorem cotorsion pair} (ii) and Buchweitz's Theorem to consider them in $\underline{\Gproj}^\Z A$, (in fact, this is the main technique to prove homological properties in singularity categories in this paper).

We take a minimal projective resolution
$$ \cdots \rightarrow P_2\rightarrow P_1\rightarrow A(i)\rightarrow A(i)_{\leq 0}\rightarrow0$$
of $A(i)_{\leq 0}$ in $\mod^\Z A$. Since $A$ is positively graded, we have $(P_j)_{\leq0}=0$ for $j>0$. Thus
$(\Omega^j T)_{\leq0}=0$ holds for any $j>0$.
Since $A$ is Gorenstein, by Theorem \ref{theorem cotorsion pair}, there is an exact sequence
$$0\rightarrow A_T\xrightarrow{a} N_T\xrightarrow{b} T\rightarrow0$$
where $N_T$ is graded Gorenstein projective $A$-module, and $A_T$ is graded $A$-module of finite projective dimension.
Let $P_{A_T}$ be the projective cover of $A_T$. Then we get the following commutative diagram with each row and each column short exact:
\begin{align}\label{xymatrix:1}
\xymatrix{ L\ar[r] \ar[d]^{c} & P_{A_T} \ar[r]^e \ar[d] & A_T\ar[d]^{a}\\
M\ar[r] \ar[d]^{d} & P_{A_T}\oplus (\bigoplus_{0\leq i\leq l}A(i)) \ar[d]\ar[r] &N_T\ar[d]^{b}\\
\Omega T\ar[r]&  \bigoplus_{0\leq i\leq l}A(i) \ar[r]& T. }
\end{align}
Then $L$ is of finite projective dimension, and $M$ is Gorenstein projective. In particular, $\Omega T\cong M$ in $D_{sg}(\mod^\Z A)$.
Assume that $M=\bigoplus_{i=1}^m M_i$ with $M_i$ indecomposable for $1\leq i\leq m$. Then $M_i$ is Gorenstein projective, which implies that $\pd_A M_i=0\mbox{ or }\infty$ for each $i$.
We assume that $M_i$ is non-projective for $1\leq i\leq t$, and $M_i$ is projective for $t+1\leq i\leq m$. In this way, $d$ is of the form
$$(d_1,\dots, d_m): M=\bigoplus_{i=1}^m M_i \rightarrow \Omega(T)$$
with $d_i:M_i\rightarrow \Omega(T)$. Then $d_i\neq 0$ for $1\leq i\leq t $. In fact, if $d_i=0$ for some $1\leq i\leq t$, then from the short exact sequence
$$0\rightarrow L\xrightarrow{c} M\xrightarrow{d} \Omega(T)\rightarrow0,$$ we get that
$M_i$ is a direct summand of $L$ which implies that $M_i$ has finite projective dimension, and then it is projective, giving a contradiction.

From the definition of $T$, easily, $\Top(\Omega(T))$ is homogeneous of degree one, and $(\Omega(T))_{\geq1}=\Omega(T)$.
Lemma \ref{lemma structure of Gorenstein projective module} shows that $\Top(M_i)$ is simple for $1\leq i\leq m$, and then $\deg(\Top(M_i))\geq1$ for $1\leq i\leq t$ since there is a nonzero morphism $d_i:M_i\rightarrow \Omega(T)$. Furthermore, because $A$ is positively graded, we have $(\Omega^{j}(M))_{\geq1}=\Omega^j(M)$ for any $j\geq0$. Here $\Omega^0(M)=M$ for $j=0$.

Similarly, assume that $N_T=\bigoplus_{i=1}^n N_i$, where $N_i$ is indecomposable for $1\leq i\leq n$. We  also assume that $N_i$ is non-projective for $1\leq i\leq s$, $N_i$ is projective for $s+1\leq i\leq n$.
Then $\deg(\Top(N_i))\leq 0$ for $1\leq i\leq s$ since $(T)_{\leq0}=T$. In fact, $s=t$ since $\Omega(N_T)\cong M$ in $\underline{\Gproj}^\Z A$.

%Denote by $\langle 1\rangle$ the shift functor in $D_{sg}(\mod^\Z A)$. Let $p\in \Z^+$. Then
For \eqref{eq:tiltingp<0}, we have
\begin{align*}
\Hom_{D_{sg}(\mod^\Z A)}(T,T\langle -p\rangle)=&\Hom_{D_{sg}(\mod^\Z A)}(N_T,N_T\langle -p\rangle)\\
=&\Hom_{\underline{\Gproj}^\Z A}(N_T,\Omega^p(N_T))\\
=&\Hom_{\underline{\Gproj}^\Z A}(\oplus_{i=1}^sN_i,\Omega^{p-1}(\oplus_{j=1}^s M_j )).
\end{align*}
For each $1\leq i\leq s$, since $\deg(\Top(N_i))\leq0$ and $(\Omega^{p-1}(\oplus_{j=1}^s M_j ))_{\geq1}=\Omega^{p-1}(\oplus_{j=1}^s M_j )$, we obtain that
$\Hom_{\mod^\Z A}(N_i, \Omega^{p-1}(\oplus_{j=1}^s M_j ))=0$, which yields that
$$\Hom_{\underline{\Gproj}^\Z A}(\oplus_{i=1}^sN_i,\Omega^{p-1}(\oplus_{i=1}^s M_i ))=0.$$
So
$\Hom_{D_{sg}(\mod^\Z A)}(T,T\langle -p\rangle)=0$.

On the other hand, for \eqref{eq:tiltingp>0},
\begin{align*}
\Hom_{D_{sg}(\mod^\Z A)}(T,T\langle p\rangle)=&\Hom_{D_{sg}(\mod^\Z A)}(T\langle -p\rangle,T)\\
=&\Hom_{D_{sg}(\mod^\Z A)}(N_T\langle -p\rangle,N_T)\\
=&\Hom_{\underline{\Gproj}^\Z A}(\Omega^p(N_T),N_T)\\
=&\Hom_{\underline{\Gproj}^\Z A}(\Omega^{p-1}(\oplus_{i=1}^s M_i ),\oplus_{i=1}^sN_i).
\end{align*}
For any morphism $f:\Omega^{p-1}(\oplus_{i=1}^s M_i )\rightarrow\oplus_{i=1}^sN_i$ in $\mod^\Z A$, since
$$(\Omega^{p-1}(\oplus_{j=1}^s M_j ))_{\geq1}=\Omega^{p-1}(\oplus_{j=1}^s M_j ),$$
we get that $(\Im f)_{\geq1}=\Im f$. The morphism $b:N_T\rightarrow T$ in (\ref{xymatrix:1}) induces a morphism
$$b'=(b_1,\dots,b_s):\oplus_{i=1}^s N_i\rightarrow T.$$
Using the fact $(T)_{\leq0}=T$, we find $b'f=0$. Extend $f$ to be a morphism
$$\left( \begin{array}{cc}f\\0 \end{array}\right): \Omega^{p-1}(\oplus_{i=1}^s M_i )\rightarrow (\oplus_{i=1}^sN_i)\bigoplus(\oplus_{i=s+1}^nN_i)=N_T.$$
It is easy to see that $\left( \begin{array}{cc}f\\0 \end{array}\right)$ factors through $a:A_T\rightarrow N_T$, and then it also factors through $e:P_{A_T}\rightarrow A_T$ by Lemma \ref{lemma property of gorenstein projective modules} since $\Omega^{p-1}(\oplus_{i=1}^s M_i )$ is Gorenstein projective and $L$ is of finite projective dimension. So $f=0$ in $\underline{\Gproj}^\Z A$.
As $f$ is arbitrary, we get that $$\Hom_{D_{sg}(\mod^\Z A)}(T,T\langle p\rangle)\cong \Hom_{\underline{\Gproj}^\Z A}(\Omega^{p-1}(\oplus_{i=1}^s M_i ),\oplus_{i=1}^sN_i)=0.$$

The proof is completed. %Together with Lemma \ref{lemma T is a generator}, $T$ is a tilting object in $D_{sg}(\mod^\Z A)$.
\end{proof}

\subsection{Realising $D_{sg}(\mod^\Z A)$ of Gorenstein monomial algebras as derived categories}

In the following, we prove that $D_{sg}(\mod^\Z A)$ is triangulated equivalent to $D^b(\mod H)$, where $H$ is some hereditary algebra of finite representation type. Before that, we give some lemmas. Following \cite[Example 8.4(2)]{Be}, an algebra $A$ is called to be \emph{CM-finite} if there are only finitely many indecomposable Gorenstein projective modules (up to isomorphism) over $A$.

\begin{lemma}\label{lemma CM-finite}
Let $A=KQ/I$ be a monomial algebra. Then $A$ is CM-finite. In particular, $\Gproj^\Z A$ has finitely many indecomposable objects up to isomorphism and degree shift.
\end{lemma}
\begin{proof}
Since $A$ is a finite-dimensional algebra over $K$, there are only finitely many perfect paths in $A$.
From (\ref{equation characterization of gorenstein projective modules}) in Theorem \ref{theorem bijection of perfect path and Gorenstein projective modules}
we get that there are only finitely many indecomposable Gorenstein projective $A$-modules up to isomorphism.
The forgetful functor $F:\Gproj^\Z A\rightarrow \Gproj A$ in Lemma \ref{lemma structure of Gorenstein projective module} shows that $\Gproj^\Z A$ has finitely many indecomposable objects up to isomorphism and degree shift. %the shift of degree.
\end{proof}

Recall that a hereditary path algebra $KQ$ is of \emph{finite representation type} if every connected component of $Q$ is of Dynkin type.

\begin{lemma}[see e.g. \text{\cite[Theorem A]{BGS}, \cite[Theorem 2.1]{V}}]\label{lemma derived finite category}
For a finite-dimensional algebra $B$ over $K$, if $D^b(\mod B)$ has only finitely many indecomposable objects up to shift of complexes, then there exists a hereditary algebra $H$ of finite representation type such that $D^b(\mod B)\simeq D^b(\mod H)$.
\end{lemma}

\begin{lemma}\label{lemmm algebra of finite global dimension}
Let $B$ be a finite-dimensional algebra over $K$, if $K^b(\proj B)$ has finitely many indecomposable objects up to isomorphism and shift of complexes, then $B$ is of finite global dimension. Furthermore, in this case, $D^b(\mod B)$ also has only finitely many indecomposable objects up to isomorphism and shift of complexes.
\end{lemma}
\begin{proof}
Suppose for a contradiction that $B$ has infinite global dimension. Because $\gldim B=\max\{\pd_B S\mid S\text{ is a simple module}\}$, there exists a simple module $S$ with infinite projective dimension. Let
%\begin{equation*}\label{projective resolution}
$\cdots \rightarrow P_n\rightarrow \cdots \to P_1\rightarrow P_0\rightarrow S\rightarrow0$
%\end{equation*}
be a minimal projective resolution of $S$. Then the complex $P_n\rightarrow \cdots P_1\rightarrow P_0$
is an indecomposable object in $K^b(\proj B)$ for any $n\geq0$. As $S$ has infinite projective dimension, we get that $P_i\neq0$ for any $i\geq0$, and then
the complexes $P_n\rightarrow \cdots \to P_1\rightarrow P_0$ are pairwise non-isomorphic even up to shift of complexes. So there are infinitely many indecomposable objects up to isomorphism and shift of complexes, giving a contradiction.
\end{proof}

\begin{proposition}\label{proposition equivalent to hereditary algebra}
Let $A$ be a Gorenstein monomial algebra. Then there exists a hereditary algebra $H$ of finite representation type such that $D_{sg}(\mod^\Z A)\simeq D^b(\mod H)$.
\end{proposition}
\begin{proof}
Since $D_{sg}(\mod^\Z A)\simeq \underline{\Gproj}^\Z A$, Theorem \ref{main theorem} implies that there is a tilting object $T$ in $\underline{\Gproj}^\Z A$. Without loss of generality, we assume that $T$ is basic. Denote by $B=\End_{\underline{\Gproj}^\Z A}(T)^{op}$. Then $\underline{\Gproj}^\Z A\simeq K^b(\proj B)$.

In the following we prove that $\Gproj^\Z A$ has only finitely many indecomposable objects up to isomorphism and the syzygy functor $\Omega$.
For any indecomposable object in $\underline{\Gproj}^\Z A$,
there exists a perfect path $p$ such that it is isomorphic to $Ap$ up to some certain shift of degree. By Definition \ref{definition of perfect path}, there exists a sequence
$$p=p_1,p_2,\dots,p_{n},p_{n+1}=p$$
of nonzero paths such that $(p_i,p_{i+1})$ are perfect pairs for all $1\leq i\leq n$.
Then $A p_i$ is Gorenstein projective for $1\leq i\leq n$, and $\Omega (Ap_{i+1})$ is isomorphic to $Ap_{i}(-m_{i})$ by (\ref{the syzygy functor of Gorenstein projective modules}) for some $m_{i}$, and then
$\Omega^{n}(Ap)\cong Ap(-m_p)$ for some $m_p$ in $\underline{\Gproj}^\Z A$. Obviously, $m_p>0$.

For any indecomposable object $X$ in $\underline{\Gproj}^\Z A$, $X\cong Ap(i)$ for some perfect path $p$ and $i\in\Z$. There exists an integer $0\leq j<m_p$ satisfying $i\equiv j(\mod m_p)$, and then $X$ is isomorphic to $Ap(-j)$ up to $\Omega$. Let $m=\max\{m_p\mid p\text{ is a perfect path in } A\}$. Then $m<\infty$ since there are only finitely many perfect paths in $A$. By viewing each indecomposable $X\in\Gproj A$ as a graded Gorenstein projective module with $\Top(X)$ degree zero, let $\cx=\bigcup_{0\leq i<m}\Ind\underline{\Gproj}(A)(-i)$. Then $\cx$ is a finite set, and for any indecomposable object $X$ in $\underline{\Gproj}^\Z A$, there exists an object $Y\in\cx$ such that $X$ is isomorphic to $Y$ up to  $\Omega$. So there are only finitely many indecomposable objects in $\underline{\Gproj}^\Z A$ up to  $\Omega$.

For the triangle equivalence $\underline{\Gproj}^\Z A\simeq K^b(\proj B)$, the shift of complexes in $K^b(\proj B)$ corresponds to the cosyzygy functor $\Omega^{-1}$ in $\underline{\Gproj}^\Z A$. So there are only finitely many indecomposable objects in $K^b(\proj B)$ up to shift of complexes. Lemma \ref{lemmm algebra of finite global dimension} implies $K^b(\proj B)\simeq D^b(\mod B)$, together with Lemma \ref{lemma derived finite category}, there exists a hereditary algebra $H$ of finite representation type such that $\underline{\Gproj}^\Z A\simeq D^b(\mod H)$. Then our desired result follows.
\end{proof}

Note that the hereditary algebra $H$ is not unique in Proposition \ref{proposition equivalent to hereditary algebra}, however, it is unique up to derived equivalences.

Recall that a monomial algebra $A=KQ/I$ is called to be \emph{quadratic monomial} provided that the ideal $I$ is generated by paths of length two.
\begin{corollary}
Let $A$ be a Gorenstein quadratic monomial algebra. Then there exists a hereditary algebra $B=KQ$ with $Q$ a union of finitely many quivers of type $\A_1$ such that $D_{sg}(\mod^\Z A)\simeq D^b(\mod B)$.
\end{corollary}
\begin{proof}
Similar to the proof of Theorem \ref{main theorem}, we take a positive integer $l$ such that $A=A_{\leq l}$ and define $T:=\bigoplus_{0\leq i\leq l} A(i)_{\leq0}.$
Then there exists a short exact sequence
$$0\rightarrow A_T\xrightarrow{a} N_T\xrightarrow{b} T\rightarrow0$$
where $N_T$ is a graded Gorenstein projective $A$-module, and $A_T$ is a graded $A$-module of finite projective dimension. By the proof of Theorem \ref{main theorem}, $N_T$ is a tilting object in $\underline{\Gproj}^\Z A$.

Without loss of generality, we assume that $N_T$ is basic, and $N_T=Ap_1(r_1)\oplus \cdots\oplus Ap_n(r_n)$. By \cite[Theorem 5.7]{CSZ}, we get that $\underline{\Gproj}A\simeq \ct_{d_1}\times \ct_{d_2}\times\cdots\times \ct_{d_m}$, where
$\ct_{d_i}=D^b(\mod K)/[d_i]$ is the triangulated orbit category in the sense of \cite{Ke}. So $\Hom_{\underline{\Gproj} A}(Ap_i,Ap_j)=0$ for any $i\neq j$.
From the fact that $\Hom_{\underline{\Gproj}^\Z A}(Ap_i(r_i),Ap_j(r_j))$ is a subset of $\Hom_{\underline{\Gproj} A}(Ap_i,Ap_j)$,
we find that
$$\End_{\underline{\Gproj}^\Z A}(N_T)\cong \overbrace{K\times K\times \cdots K}^{n}.$$
Therefore,
there exists a hereditary algebra $B=KQ$ with $Q$ a union of finitely many quivers of type $\A_1$ such that $D_{sg}(\mod^\Z A)\simeq\underline{\Gproj}^\Z A\simeq D^b(\mod B)$.
\end{proof}

\begin{corollary}
Let $A$ be a Gorenstein Nakayama algebra. Then there exists a hereditary algebra $B=KQ$ with $Q$ a union of finitely many quivers of type $\A$ such that $D_{sg}(\mod^\Z A)\simeq D^b(\mod B)$.
\end{corollary}
\begin{proof}
From \cite[Proposition 1]{Rin}, we get that $\underline{\Gproj} A$ is triangulated equivalent to the stable category of a self-injective Nakayama algebra, which is also triangulated equivalent to the triangulated orbit category $D^b(\mod H)/[n]$ ($n\geq1$), where $H$ is a hereditary algebra of type $\A$ and $[1]$ is the shift functor of $D^b(\mod H)$. Theorem \ref{main theorem} shows that $\underline{\Gproj}^\Z A$ is triangulated equivalent to a hereditary algebra of finite representation type.

Suppose for a contradiction that there exists a connected component $\cc$ of $\underline{\Gproj}^\Z A$ such that it is equivalent to the derived category of a hereditary algebra of type $\D$ or $\E$. Then there exists an almost split triangle $$L\rightarrow M_1\oplus M_2\oplus M_3\rightarrow N\rightarrow \Sigma L$$
in $\underline{\Gproj}^\Z A$ for some nonzero indecomposable modules $L,M_1,M_2,M_3,N\in{\Gproj}^\Z A$. So there exists a projective module $P$ such that $L\rightarrow M_1\oplus M_2\oplus M_3 \oplus P\rightarrow N$ is an almost split sequence. As the forgetful functor $F:\Gproj^\Z A\rightarrow \Gproj A$ is dense, it is easy to see that it preserves almost split sequences, and then
$F(L)\rightarrow F(M_1)\oplus F(M_2)\oplus F(M_3) \oplus F(P)\rightarrow F(N)$ is an almost split sequence in $\Gproj A$, which implies that $F(L)\rightarrow F(M_1)\oplus F(M_2)\oplus F(M_3)\rightarrow F(N)\rightarrow  F(L)[1]$ is an almost split triangle in $\underline{\Gproj}A$, giving a contradiction to the fact that $\underline{\Gproj} A$ is triangulated equivalent to the stable category of a self-injective Nakayama algebra. So every connected component of $\underline{\Gproj}^\Z A$ is equivalent to the derived category of a hereditary algebra of type $\A$.
Thus our desired result follows from Buchweitz's Theorem.
\end{proof}

Note that C. M. Ringel gives a description of $\underline{\Gproj} A$ for any Nakayama algebras \cite{Rin}, recently, D. Shen gives a characterization of Gorenstein Nakayama algebras \cite{Shen}.

Recall that $A$ is a \emph{self-injective Nakayama algebra} if and only if $A=K$ or there exists an oriented cycle $Z_n$ with the vertex set $\{1,2,\dots,n\}$ and the arrow set $\{\alpha_1,\dots,\alpha_n\}$, where $s(\alpha_i)=i$ and $t(\alpha_i)=i+1$, such that $A$ is isomorphic to $KZ_n/J^d$ for some $d\geq2$, where $J$ denotes the two-sided ideal of $KZ_n$ generated by arrows.

The following remark may be known to some experts.
\begin{remark}
Let $A=KQ/I$ be a monomial algebra, where $Q$ is connected. Then $A$ is self-injective if and only if $A$ is a self-injective Nakayama algebra.
\end{remark}
\begin{proof}
We only need to prove the ``only if'' part. It is enough to prove that $Q$ is an oriented cycle, and the proof is to analyse the structures of projective and injective modules.

 For any vertex $i\in Q$, its corresponding indecomposable projective module $P_i$ is as Fig. 5. shows, and its corresponding indecomposable projective module $I_i$ is as Fig. 6. shows. Since $A$ is self-injective, $P_i$ is an indecomposable injective module as Fig. 6. shows, which implies that $P_i$ (and also $I_i$) is a string module with its string of the form
$\cdot\rightarrow \cdot\rightarrow\cdots\rightarrow \cdot\rightarrow\cdot.$
If $A\neq K$, then for any vertex $i$, there is at most one arrow starting from $i$, and at most one arrow ending at $i$. Because $Q$ is connected and $A$ is self-injective which is not isomorphic to $K$, we get that there is no sink vertex and no source vertex in $Q$. So for any vertex $i$, there is only one arrow starting from $i$, and only one arrow ending at $i$, and then $Q\cong Z_n$ for some $n$.
Therefore, $A\cong KZ_n/I$ which is a Nakayama algebra, and then $A$ is a self-injective Nakayama algebra.
\end{proof}

\begin{remark}
For a Gorenstein monomial algebra $A=KQ/I$, there are many ways to make it to be a positively graded algebra, not only by setting each arrow to be degree one. For any positively grading on $A$ such that its zero part $A_0$ satisfying $\gldim A_0<\infty$, similarly one can show that $\underline{\Gproj}^\Z A$ admits a tilting object by the same construction in Theorem \ref{main theorem}, furthermore, all the results in this section also hold. %In other words, the results in this section do not depend on the grading.
\end{remark}

\section{Characterization of $1$-Gorenstein monomial algebras}\label{section 5}

\subsection{Characterization of $1$-Gorenstein monomial algebras}
In this subsection, we give a characterization of $1$-Gorenstein monomial algebras $KQ/I$ by using the minimial paths in $I$, which is a generalization of the characterization of $1$-Gorenstein gentle algebras; see \cite[Proposition 3.1]{CL}.

First, we fix some notations. Let $A=KQ/I$ be an algebra. For any vertex $i$ in $Q$, we denote by $P_i$ the corresponding indecomposable projective module and $S_i$ the corresponding simple module.

\begin{lemma}\label{lemma not perfect pair}
Let $A=KQ/I$ be a monomial algebra. Let $\bF$ be the set formed by all the minimal paths among the paths in $I$. If there exist nontrivial paths $p,q$ such that
$pq\in\bF$ and $p$ is not a perfect path, then $A$ is not $1$-Gorenstein.
\end{lemma}
\begin{proof}
Suppose for a contradiction that $A$ is $1$-Gorenstein.
For the indecomposable projective module $P_{s(q)}$, which has a basis formed by nonzero paths starting at $s(q)$, it has a submodule $M$ with a basis $\{u\mid uq\text{ is a nonzero path in } A\}$.
Obviously, $\Top(M)=S_{t(q)}$, and then $M$ is indecomposable. Furthermore, since $A$ is $1$-Gorenstein, it is well known that the Gorenstein pojective modules coincide with the
torsionless modules (see e.g. \cite[Remark 2.6]{CGLu}), and then $M$ is an indecomposable Gorenstein projective module.
In the following, we prove that $M$ can not be Gorenstein projective by using the notions of perfect pairs, perfect paths, and Theorem \ref{theorem bijection of perfect path and Gorenstein projective modules}, which gives a contradiction.

It is easy to see that the nonzero path $p\notin M$, and then $M$ is not projective.
Theorem \ref{theorem bijection of perfect path and Gorenstein projective modules} implies that there exists a perfect path $q'$ such that $M=Aq'$. Let $(p',q')$ be the perfect pair. Obviously, $s(p')=t(q')=t(q)$. It follows from $p\notin M$ that $pq'=0$ and then $p=p''p'$ for some nonzero path $p''$ by definition. Claim $p'=p$. If the claim holds, we have $(p,q')$ is a perfect pair. From the fact $q'$ is a perfect path, $p$ is also a perfect path, giving a contradiction. Therefore, $A$ is not $1$-Gorenstein.

For the claim, denote by $p=\alpha_n\cdots\alpha_1$ with $\alpha_i\in Q_1$ for $1\leq i\leq n$. By $pq=\alpha_n\cdots\alpha_1q\in \bF$, we find that $\alpha_{n-1}\cdots \alpha_1q$ (equals to $q$ if $n=1$) is nonzero, and then $\alpha_{n-1}\cdots \alpha_1$ is in $M$. So
$\alpha_{n-1}\cdots \alpha_1q'$ is nonzero. By noting that $pq'=\alpha_n\alpha_{n-1}\cdots \alpha_1q'$ is zero, we get that $p'=p$. The claim is proved.
\end{proof}

\begin{lemma}\label{lemma Gorenstein projective modules for 1-Gorenstein modules}
Let $A=KQ/I$ be a monomial algebra, $\bF$ be the set formed by all the minimal paths among the paths in $I$. Assume that every nonzero path $p$ with the property that there exists a nonzero path $q$ such that $qp\in\bF$ is a perfect path. Then $Ap$ is an indecomposable Gorenstein projective $A$-module for any nonzero path $p$.
\end{lemma}
\begin{proof}
We only consider the case for $Ap$ to be non-projective.

First, $Ap$ is indecomposable for any nonzero path $p$. If $Ap$ is not projective, i.e., $Ap$ is not isomorphic to $P_{t(p)}$, then there exists a nonzero path $q'$ such that $q'p\in I$, which implies that $L(p)$ is nonempty. Let $q$ be a path in $L(p)$.
Obviously, $R(q)$ is nonempty. As $qp\in I$, there exists $p'\in R(q)$ such that $p=p'p''$. In the following, we shall prove that
\begin{align}
\label{eq:perf}
qp'\in\bF,\qquad
Ap=Ap'
\end{align}
by using the notions of perfect pairs and perfect paths, and Theorem \ref{theorem bijection of perfect path and Gorenstein projective modules}.
If \eqref{eq:perf} holds, from the hypothesis, we get that $(q,p')$ is a perfect pair, and then $p'$ is a perfect path. So $Ap= Ap'$ is a non-projective Gorenstein projective module by Theorem \ref{theorem bijection of perfect path and Gorenstein projective modules}.

For the first formula in \eqref{eq:perf}, since $qp'\in I$ and $p',q$ are nonzero, there exist nonzero paths $q_1,q_2,p_1,p_2$ such that $q=q_2q_1$, $p'=p_1p_2$ and $q_1p_1\in\bF$.
Then $qp_1=q_2q_1p_1\in I$, which implies that there exists a path $p_1'\in R(q)$ such that $p_1=p_1'p_1''$ for some nonzero path $p_1''$. On the other hand, $p'=p_1p_2=p_1'p_1''p_2$, which is also in $R(q)$. So $p_1''$ and $p_2$ are trivial paths, and then $p'=p_1$.
Similarly, from $q_1p=q_1p_1p_2p''\in I$ and $q=q_2q_1\in L(p)$, we get that $q_2$ is trivial, and then $q=q_1$. Therefore, $qp'=q_1p_1\in\bF$.

For the second formula in \eqref{eq:perf}, $Ap$ (resp. $Ap'$) has a basis $\cs$ (resp. $\cs'$) given by all nonzero paths $q'$ such that $q'p\notin I$ (resp. $q'p'\notin I$). Since $p=p'p''$, it is easy to see that $\cs\subseteq \cs'$.

Conversely, suppose for a contradiction that there exists a nonzero path $q'$ satisfying that $q'p'\notin I$, and $q'p\in I$. Without loss of generality, assume that $q'\in L(p)$. Similar to the above, there exists a nonzero path $p_3$ such that $q'p_3\in\bF$ and $p=p_3p_4$ for some nonzero path $p_4$.
Together with $p=p'p''$, we get that either $l(p')\geq l(p_3)$ or $l(p')\leq l(p_3)$. If $l(p')\geq l(p_3)$, then $p'=p_3p_5$ for some nonzero path $p_5$.
So $q'p'=q'p_3p_5\in I$, giving a contradiction. If $l(p')\leq l(p_3)$, then $p_3=p'p_6$ for some nonzero path $p_6$, which yields that $qp_3=qp'p_6\in I$. From $q'p_3\in\bF$, we get that $(q',p_3)$ is a perfect pair, and then there exists a path $q_1$ such that $q=q_1q'$. On the other hand, $q',q\in L(p)$, which implies that $q_1$ is trivial and $q=q'$. Since both of $(q,p')$ and $(q',p_3)$ are perfect pairs, we get that $p'=p_3$.
Then $q'p'=qp'\in\bF$, giving a contradiction to that $q'p'\notin I$.

Therefore, $Ap$ is isomorphic $Ap'$ as $A$-modules, and then $Ap$ is a non-projective Gorenstein projective module.
%To sum up, $Ap$ is an indecomposable Gorenstein projective $A$-module for any nonzero path $p$.
\end{proof}

From Fig. 5., it is easy to get the following lemma, which is a special case of \cite[Theorem 2.2]{B}.

\begin{lemma}\label{lemma first sysyzy for monomial algebra}
Let $A=KQ/I$ be a monomial algebra. Then for any semisimple module $M$, the first syzygy module of $M$ is isomorphic to a direct sum $\oplus Ap^{(\Lambda(p))}$, where $p$ runs over all the nonzero paths in $A$ and each $\Lambda(p)$ is some index set.
\end{lemma}

\begin{theorem}\label{proposition characterize of 1 Gorenstein monomial algebras}
Let $A=KQ/I$ be a monomial algebra. Let $\bF$ be the set formed by all the minimal paths among the paths in $I$. Then $A$ is $1$-Gorenstein if and only if every nonzero path $p$ with the property that there exists a nonzero path $q$ such that $qp\in\bF$ is a perfect path.
\end{theorem}
\begin{proof}
It follows from Lemma \ref{lemma not perfect pair} that if $A$ is $1$-Gorenstein, then for any nonzero path $p$ with the property that there exists a nonzero path $q$ such that $qp\in\bF$, we have that $p$ is a perfect path.

Conversely, by Theorem \ref{theorem characterize of gorenstein property}, we only need to check that $\Omega(\mod A)=\Gproj A$.

First, by Lemma \ref{lemma first sysyzy for monomial algebra} and Lemma \ref{lemma Gorenstein projective modules for 1-Gorenstein modules}, for any semisimple module, its first syzygy is Gorenstein projective.
For any finite-dimensional module $M$, by induction, we assume that the first syzygy of $\rad(M)$ is Gorenstein projective. Then there are exact sequences
$$0\rightarrow \rad(M)\rightarrow M\rightarrow\Top(M)\rightarrow0,\quad 0\rightarrow N_2\rightarrow P_2\rightarrow \rad(M)\rightarrow0,$$
and $0\rightarrow N_1\rightarrow P_1\rightarrow \Top(M)\rightarrow0$
with $N_1,N_2$ Gorenstein projective, $P_1,P_2$ projective.
So we get the following commutative diagram with each row and column short exact:
\[\xymatrix{N_2\ar[r]\ar@{.>}[d] &P_2\ar[r]\ar[d] & \rad(M)\ar[d]\\
N\ar@{.>}[r]\ar@{.>}[d] & P_1\oplus P_2\ar[r] \ar[d]& M\ar[d]\\
N_1\ar[r] &P_1\ar[r]& \Top(M).
}\]
Since $\Gproj(A)$ is closed under taking extensions, by the short exact sequence in the first column, we get that $N\in \Gproj(A)$, and then
$\Omega(M)$ is Gorenstein projecitve. So $\Omega(\mod A)\subseteq \Gproj(A)$. On the other hand, it is obvious that $\Gproj(A)\subseteq \Omega(\mod A)$, and then $\Gproj(A)=\Omega(\mod A)$. So Theorem \ref{theorem characterize of gorenstein property} shows that
$A$ is $1$-Gorenstein.
\end{proof}

As a special class of monomial algebras, gentle algebras have some nice properties.

\begin{definition}[\text{\cite[Page 272]{AS}}]
The pair $(Q,I)$ is called gentle if it satisfies the following conditions.
\begin{itemize}
\item Each vertex of $Q$ is starting point of at most two arrows, and end point of at most two arrows.
\item For each arrow $\alpha$ in $Q$ there is at most one arrow $\beta$ such that $\alpha\beta\notin I$, and at most one arrow $\gamma$ such that $\gamma\alpha\notin I$.
\item The set $I$ is generated by zero-relations of length $2$.
\item For each arrow $\alpha$ in $Q$ there is at most one arrow $\beta$ with $t(\beta)=s(\alpha)$ such that $\alpha\beta\in I$, and at most one arrow $\gamma$ with $s(\gamma)=t(\alpha)$ such that $\gamma\alpha\in I$.
\end{itemize}
\end{definition}
A finite-dimensional algebra $A$ is called \emph{gentle}, if it has a presentation as $A=KQ/\langle I\rangle$ where $(Q,I)$ is gentle.
For a gentle algebra $\Lambda=KQ/\langle I\rangle$, we denote by $\cc(\Lambda)$ the set of equivalence classes (with respect to cyclic permutation) of \emph{repetition-free} cyclic paths $\alpha_1\dots\alpha_n$ in $Q$ such that $\alpha_i\alpha_{i+1}\in I$ for all $i$, where we set $n+1=1$.
Then Theorem \ref{proposition characterize of 1 Gorenstein monomial algebras} yields the following corollary.
\begin{corollary}[\text{\cite[Proposition 3.1]{CL}}]
Let $\Lambda=KQ/\langle I\rangle$ be a finite-dimensional gentle algebra. Then $\Lambda$ is $1$-Gorenstein if and only if for any arrows $\alpha,\beta$ in $Q$ with $s(\beta)=t(\alpha)$ and $\beta\alpha\in I$, there exists $c\in \cc(\Lambda)$ such that $\alpha,\beta\in c$.
\end{corollary}

The following result gives a description of the elements in $\bF$ for a $1$-Gorenstein monomial algebra.
\begin{proposition}
\label{proposition elements of ideal for 1-Gorenstein monomial algebras}
Let $A=KQ/I$ be a $1$-Gorenstein monomial algebra. Denote by $\bF$ the set formed by all the minimal paths among the paths in $I$. Assume that
$c=p_{n}\cdots p_2p_1$ is a repetition-free cycle such that $(p_{i+1},p_{i})$ are perfect pairs for all $1\leq i\leq n$. Let $l(p_i)=r_i$ for all $1\leq i\leq n$. Then $r_1+r_2=r_2+r_3=\cdots=r_n+r_1$ and for any subpath $p$ of $c$ with length $r=r_1+r_2$, we have  $p\in\bF$.
%\begin{itemize}
%\item $l(p_i)=r_i$ for all $1\leq i\leq n$;
%\item $p_i=\alpha_{\sum_{i=1}^j r_i}\cdots \alpha_{(\sum_{i=1}^{j-1} r_i)+1}$  for all $1\leq j\leq n$;
%\item $(p_{i+1},p_{i})$ are perfect pairs for all $1\leq i\leq n$.
%\end{itemize}
%Then
%\begin{itemize}
%\item[(i)] $r_1+r_2=r_2+r_3=\cdots=r_n+r_1$;
%\item[(ii)] $\alpha_{j+r-1}\cdots\alpha_{j}\in\bF$ for any $j\geq 1$, where $r=r_1+r_2$. Here we set $\sum_{i=1}^n r_i+1=1$.
%\end{itemize}
\end{proposition}

\begin{proof}
The proof is elementary, and uses only the definitions of perfect pairs and perfect paths.

Obviously, $p_{i+1}p_i\in\bF$ for all $1\leq i\leq n$. Without loss of generality, we assume that $p_2p_1$ is (one of) the longest path in $\{p_{i+1}p_i\mid 1\leq i\leq n\}$. Then $p_2p_1=\alpha_{r_1+r_2}\cdots\alpha_{r_1+1}\alpha_{r_1}\cdots \alpha_1\in\bF$. As $A$ is a $1$-Gorenstein monomial algebra, Theorem \ref{proposition characterize of 1 Gorenstein monomial algebras} shows that $\alpha_{r_1+r_2}\cdots\alpha_{r_1+1}\alpha_{r_1}\cdots \alpha_2$ is a perfect path. So there exists a nonzero path $q_1$ such that $(q_1,\alpha_{r_1+r_2}\cdots\alpha_{r_1+1}\alpha_{r_1}\cdots \alpha_2)$ is a perfect pair. Since $p_3\alpha_{r_1+r_2}\cdots\alpha_{r_1+1}\alpha_{r_1}\cdots \alpha_2=p_3p_2\alpha_{r_1}\cdots \alpha_2\in I$, we have $p_3=q_1'q_1$ for some path $q_1'$ by definition.
So $q_1=\alpha_{r_1+r_2+s_1}\cdots \alpha_{r_1+r_2+1}$ for some $1\leq s_1\leq r_3$.

Similarly, $\alpha_{r_1+r_2}\cdots\alpha_{r_1+1}\alpha_{r_1}\cdots \alpha_3$ is also a perfect path, and then there exists a nonzero path $q_2$ such that $(q_2,\alpha_{r_1+r_2}\cdots\alpha_{r_1+1}\alpha_{r_1}\cdots \alpha_3 )$ is a perfect pair. We have $p_3=q_2'q_2$ with $q_2= \alpha_{r_1+r_2+s_2}\cdots \alpha_{r_1+r_2+1}$ for some $1\leq s_2\leq r_3$. From the fact
$(q_1 ,\alpha_{r_1+r_2}\cdots\alpha_{r_1+1}\alpha_{r_1}\cdots \alpha_2)$ is a perfect pair, we find that $q_1\alpha_{r_1+r_2}\cdots\alpha_{r_1+1}\alpha_{r_1}\cdots \alpha_3\notin I$, and then $1\leq s_1<s_2\leq r_3$.

Inductively,
there exist $1\leq s_1<s_2\cdots<s_{r_1}\leq r_3$ such that
$$(\alpha_{r_1+r_2+s_{r_1}} \cdots \alpha_{r_1+r_2+s_2}\cdots \alpha_{r_1+r_2+s_1}\cdots \alpha_{r_1+r_2+1}, \alpha_{r_1+r_2}\cdots\alpha_{r_1+1})$$
is a perfect pair. Then $\alpha_{r_1+r_2+s_{r_1}} \cdots \alpha_{r_1+r_2+s_2}\cdots \alpha_{r_1+r_2+s_1}\cdots \alpha_{r_1+r_2+1}=\alpha_{r_1+r_2+r_3}\cdots \alpha_{r_1+r_2+1}$, and so $s_{r_1}=r_3$. By our assumption, it is easy to see that $r_1+r_2\geq r_2+r_3$, which implies that $r_1\geq r_3$. Together with $1\leq s_1<s_2\cdots<s_{r_1}=r_3$, one can see that $r_1=r_3$ and $s_i=i$ for $1\leq i\leq r_1$.

Inductively, we have $r_1+r_2=r_2+r_3=\cdots=r_n+r_1$.

From the above, we have also proved $\alpha_{j+r-1}\cdots\alpha_{j}\in\bF$ for any $j\geq 1$, where $r=r_1+r_2$.
\end{proof}

\begin{corollary}
Let $A=KQ/I$ be a Nakayama algebra with $Q$ an oriented cycle. If $A$ is $1$-Gorenstein, then $A$ is self-injective.
\end{corollary}
\begin{proof}
Assume that $Q=Z_n$ for some $n\geq1$ with $\alpha_i:i\rightarrow i+1$ being the arrow. Recall that $\bF$ is the set formed by all the minimal paths among the paths in $I$. Let $q\in\bF$ be one of the longest path. Denote by $m=l(q)$. Then Proposition \ref{proposition elements of ideal for 1-Gorenstein monomial algebras} shows every path of length $m$ is in $\bF$, and then  $I=J^m$, where $J$ is the two-sided ideal generated by arrows. So $A=KZ_n/J^m$ is self-injective.
\end{proof}

\subsection{Graded singularity categories for $1$-Gorenstein monomial algebras}
In the following, we give a characterization of $D_{sg}(\mod^\Z A)$ for any $1$-Gorenstein monomial algebra $A=KQ/I$.

\begin{lemma}\label{lemma: morphism of indecomposables}
Let $A=KQ/I$ be a $1$-Gorenstein monomial algebra. Let $Ap_i,Ap_j$ be two graded $A$-modules with $\Top(Ap_i)$ and $\Top(Ap_j)$ concentrated in degree zero for some nonzero paths $p_i,p_j$. Then the following hold.
\begin{itemize}
\item[(i)] any nonzero morphism $f\in \Hom_{\mod^\Z A}(Ap_i ,Ap_j )$ is surjective.
\item[(ii)] if $\Hom_{\mod^\Z A}(Ap_i ,Ap_j )\neq0$, then $t(p_i)=t(p_j)$.
\item[(iii)] $\dim_K\Hom_{\mod^\Z A}(Ap_i,Ap_j)\leq 1$.
\end{itemize}
\end{lemma}
\begin{proof}
(i) For any nonzero morphism $f\in \Hom_{\mod^\Z A}(Ap_i ,Ap_j )$, since $p_i$ is a generator of $Ap_i$ as left $A$-module, $f$ is uniquely determined by $f(p_i)$. By the fact $Ap_i$ has a basis given by all nonzero paths $q$ such that $q=q'p_i$ for some path $q'$, we obtain that
$(Ap_i)_0=\Span_K\{p_i\}$, Similarly, $(Ap_j)_0=\Span_K\{p_j\}$. Then $f(p_i)\in f((Ap_i)_0)\subseteq (Ap_j)_0=\Span_K\{p_j\}$, so $f(p_i)=kp_j$ for some $0\neq k\in K$.
As $p_j$ is a generator of $Ap_j$ as left $A$-module, it is easy to see that $f$ is surjective.

(ii) From the proof of (i), if $0\neq f\in \Hom_{\mod^\Z A}(Ap_i ,Ap_j )$, then $f(p_i)=kp_j$ for some $0\neq k\in K$. So
$f(p_i)=f(e_{t(p_i)}p_i)= e_{t(p_i)}f(p_i)=  k e_{t(p_i)}p_j$ which is nonzero. So $e_{t(p_i)}p_j\neq0$, which shows that $t(p_i)=t(p_j)$.

(iii) follows from the proof of (i).
\end{proof}

Now we get the main result in this subsection.

\begin{theorem}\label{singularity category of 1-Gorenstein algebras}
Let $A=KQ/I$ be a $1$-Gorenstein monomial algebra. Then there exists a hereditary algebra $B=KQ^B$ with $Q^B$ a union of finitely many quivers of type $\A$ such that $D_{sg}(\mod^\Z A)\simeq D^b(\mod B^{op})$.
\end{theorem}
\begin{proof}
Similar to proof of Theorem \ref{main theorem}, we take a positive integer $l$ such that $A=A_{\leq l}$ and define a $\Z$-graded $A$-module by
$T:=\bigoplus_{0\leq i\leq l} A(i)_{\leq0}.$
Then $T$ is a tilting object in $D_{sg}( \mod^\Z A)$.
Let $P_T\rightarrow T$ be the minimal graded projective cover. Then $0\rightarrow\Omega(T)\rightarrow P_T\rightarrow T\rightarrow0$ is a short exact sequence. Note that $\Omega(T)$ is a (graded) Gorenstein projective module since $A$ is $1$-Gorenstein. So $\Omega(T)$ is a tilting object in $\underline{\Gproj}^\Z A$. Let $B=\End_{\underline{\Gproj}^\Z A}(\Omega(T)(1))$. It is enough to prove that $B$ is a hereditary algebra of type $\A$ by Buchweitz's Theorem.

By definition, $\Top(\Omega(T))$ is homogeneous of degree 1. So every indecomposable summand of $\Omega(T)$ is of the form $Ap(-1)$ for some perfect path $p$. Here $Ap$ is viewed as a graded module such that $\Top(Ap)$ is concentrated in degree zero.

Without loss of generality, we assume that $\Omega(T)$ is basic, and then
$\Omega(T)\cong \oplus_{i=1}^n Ap_i(-1)$ for some pairwise different perfect paths $p_1,\dots,p_n$. Easily, $\Omega(T)(1)$ is also a tilting object of $\underline{\Gproj}^\Z A$. We can assume that $B=KQ^B/I^B$, where $(Q^B,I^B)$ is a bound quiver. We shall prove that $Q^B$ is of type $\A$ and $I^B=0$, by computing the irreducible morphisms in $\add\Omega(T)(1)$ in the following.

In fact, the vertex set of $Q^B$ is $\{Ap_1,\dots,Ap_n\}$.
By Lemma \ref{lemma: morphism of indecomposables}, we have the following:
\begin{itemize}
\item[(a)] there is no irreducible morphisms in $\End_{\mod^\Z A}(Ap_i)$, and then $Q_B$ has no loops;
\item[(b)] there exists at most one arrow between $Ap_i$ and $Ap_j$ for any $i\neq j$;
\item[(c)] $Q_B$ is acyclic.
\end{itemize}

For any nonzero morphism $f\in \Hom_{\mod^\Z A}(Ap_i ,Ap_j )$, if $f$ is zero in $\underline{\Gproj}^\Z A$, then $f$ factors through the projective cover $P_{t(p_j)} \xrightarrow{\alpha} Ap_j $ in $\mod^\Z A$, which is of the form $f=\beta\alpha$, where $\beta:Ap_i \rightarrow P_{t(p_j)} $. Because $t(p_j)=t(p_i)$ by Lemma \ref{lemma: morphism of indecomposables} (ii), there is a nonzero morphism $\beta:Ap_i \rightarrow P_{t(p_i)} $. Lemma \ref{lemma: morphism of indecomposables} (i) shows that $\beta$ is an isomorphism, and then $Ap_i$ is projective, giving a contradiction. So $f\neq0$ in $\underline{\Gproj}^\Z A$. Therefore, the set of arrows in $Q^B$ from $Ap_i$ to $Ap_j$ is nonempty if and only if there exists an irreducible morphism from $Ap_i$ to $Ap_j$ in the subcategory $\add \Omega(T)(1)\subseteq \Gproj^\Z A$. In this case, the arrow set of $Q^B$ from $Ap_i$ to $Ap_j$ is formed by an irreducible morphism from $Ap_i$ to $Ap_j$.

We define a partial order for the set $\{Ap_1,\dots, Ap_n\}$ as follows. We set $Ap_i\leq Ap_j$ if and only if $p_i,p_j$ have the same ending point, and $p_j=p_i p_{i,j}$ for some path $p_{i,j}$. One can check that it is well defined.
Obviously, we can assume that the partial order is $Ap_1\leq \cdots  \leq Ap_{m_1}$, $Ap_{m_1+1}\leq \cdots \leq Ap_{m_2}$, $\dots$, $Ap_{m_{r-1}+1}\leq \cdots\leq Ap_{m_r}$.
In the following, we prove that there is an irreducible morphism from $Ap_i$ to $Ap_j$ in $\add \Omega(T)$ if and only if $j=i+1$ and $Ap_i\leq Ap_{i+1}$.

Denote by $(q_i,p_i)$ the perfect pair for any $1\leq i\leq n$.
If there is a nonzero morphism $f$ from $Ap_i$ to $Ap_j$, then $t(p_i)=t(p_j)$ and  $f$ is surjective by Lemma \ref{lemma: morphism of indecomposables}.
We claim that $q_ip_j=0$ in $A$. Otherwise, $0\neq q_ip_j\in Ap_j=\Im f$. Without loss of generality, we can assume that $f(p_i)=p_j$ by the proof of Lemma  \ref{lemma: morphism of indecomposables} (i). Then there exists $(\sum_k a_k t_k)p_i\in A$ with $a_k\in K$, and $t_k$ pairwise different nonzero paths in $A$ such that $f((\sum_k a_k t_k)p_i) =q_ip_j$. We can assume that $a_k\neq0 $, and $t_kp_i\neq 0$ for each $k$. Then $q_ip_j=f((\sum_k a_k t_k)p_i) =(\sum_k a_k t_k)f(p_i)=\sum_k a_k t_kp_j$ in $A$. As all nonzero paths form a basis of $A$, there exists some $t_{k}$ such that $t_kp_j=q_ip_j$. Then $t_k=q_i$, giving a contradiction to $t_kp_i\neq0$.
The claim holds. It follows from Definition \ref{definition of perfect pair} that $p_j=p_i p_{i,j}$ for some path $p_{i,j}$, and then $Ap_i\leq Ap_j$.

Conversely, if $Ap_i\leq Ap_j$, then $p_j=p_i p_{i,j}$, and there is a nonzero morphism $f_{ij}:Ap_i\rightarrow Ap_j$ induced by the multiplication of $p_{i,j}$ for any $i,j$. In particular, Lemma \ref{lemma: morphism of indecomposables} (iii) implies that $f_{ij}=f_{j-1,j}\cdots f_{i+1,i+2}f_{i,i+1}$ (up to a scalar), and $\Hom_{\underline{\Gproj}^\Z A}(Ap_i,Ap_j)$ is spanned by $f_{ij}$ if $Ap_i\leq Ap_j$.
So if there is an irreducible morphism $f:Ap_i\rightarrow Ap_j$ in $\add \Omega(T)$, then $j=i+1$.
In fact, one can check that $f_{i,i+1}:Ap_i\rightarrow Ap_{i+1}$ is irreducible in $\add \Omega(T)$ when $Ap_i\leq Ap_{i+1}$ by noting that any morphism from $Ap_i$ to $Ap_j$ with $j\neq i$ factors through $f_{i,i+1}$.
Then
$$Q^B=\bigcup_{l=1}^r (Ap_{m_{l-1}+1}\rightarrow Ap_{m_{l-1}+2}\rightarrow\cdots\rightarrow Ap_{m_l} ).$$
Here $m_0=0$. Because each irreducible morphism $f_{i,i+1}: Ap_i\rightarrow Ap_{i+1}$ is surjective, $I^B$ must be $0$.

Therefore, there exists a hereditary algebra $B=KQ^B$ with $Q^B$ a union of finitely many quivers of type $\A$ such that $D_{sg}(\mod^\Z A)\simeq\underline{\Gproj}^\Z A\simeq D^b(\mod B^{op})$.
\end{proof}

\section{Singularity categories of $1$-Gorenstein monomial algebras}\label{section 6}

In this section, our aim is to characterize singularity categories for $1$-Gorenstein monomial algebras. Before that, we give a precise definition of the \emph{gluing of algebras}, which is defined in \cite[Subsection 2.1]{Br}. After that, we prove that singularity category is an invariance under taking this kind of gluing. Finally, we use it to characterize the singularity categories of $1$-Gorenstein monomial algebras.

\subsection{Gluing of algebras}

Let $\Lambda=KQ/I$ be a finite-dimensional quiver algebra (not necessarily monomial), where $Q=(Q_0,Q_1,s,t)$. We do not assume $Q$ to be connected. Let $E$ be an \emph{involution} on the set of vertices of $Q$. The following precedure associates a new quiver $Q(E)$ to $Q$ by gluing together each pair $x\neq E(x)$ to one vertex. Precisely, for each $x\in Q_0$, we define $\bar{x}=\{x,E(x)\}$. The quiver $Q(E)=(Q(E)_0,Q(E)_1,s(E),t(E))$ is then defined as follows:
\begin{itemize}
\item $Q(E)_0=\{\bar{x}:x\in Q_0\}$,
\item $Q(E)_1=Q_1$,
\item $s(E)(\alpha)=\overline{s(\alpha)}$ and $t(E)(\alpha)=\overline{t(\alpha)}$ for any $\alpha\in Q_1$.
\end{itemize}
From the definition it follows that any path in $Q$ is also a path in $Q(E)$, hence we can regard $I$ as a subset of $KQ(E)$. Let $I(E)$ be the ideal of $KQ(E)$ that is generated by $I$. Set $\Lambda_E:=KQ(E)/I(E)$, which is called the \emph{Br\"{u}stle's gluing algebra of $\Lambda$ by gluing the vertices along $E$}.

If $\Lambda_E$ is finite-dimensional, then there is no nonzero path from $x$ to $E(x)$ for any $x\neq E(x)$ in $\Lambda $. In fact, we have the following lemma.

\begin{lemma}\label{lemma finite dimension of gluing algebras}
Keep the notations as above. Then $\Lambda_E$ is finite-dimensional if and only if there is no nonzero paths $p_m,\dots,p_1$ in $\Lambda $ such that
$t(p_i)=E(s(p_{i+1}))$ and $t(p_i)\neq s(p_{i+1}) $ for any $i\in\Z/m\Z$.
\end{lemma}
\begin{proof}
The following proof is elementary, and uses only the definition of Br\"{u}stle's gluing algebras.

Suppose for a contradiction that there exist nonzero paths $p_m,\dots,p_1$ in $\Lambda $ such that
$t(p_i)=E(s(p_{i+1}))$ and $t(p_i)\neq E(t(p_i))$ for any $i\in\Z/m\Z$. Then $(p_m\cdots p_1)^l$ is a path in $Q(E)$ for any $l>0$. From the definition of $I(E)$, it is easy to see that $(p_m\cdots p_1)^l$ is nonzero in $\Lambda_E$, which implies that $\Lambda_E$ is infinite-dimensional, giving a contradiction.

Conversely, we assume that $(x_1,E(x_1)),\dots (x_n,E(x_n))$ are the pairs of vertices with $x_i\neq E(x_i)$.
For $1\leq j\leq n$, define the involution $E_j$ on the set of vertices of $Q$ such that $E_j(x_i)=E(x_i)$ for any $1\leq i\leq j$, and $E(x)=x$ otherwise. It is easy to see that $E_n=E$. Then we get a series of algebras $\Lambda ,\Lambda _{E_1},\dots,\Lambda _{E_n}=\Lambda_E$. We prove that all of these algebras are finite-dimensional recursively.

Denote by $N_0$ the length of the longest path in $Q$ which is nonzero in $\Lambda $. Let $q$ be a nonzero path in $\Lambda _{E_1}$. If $q$ is path in $\Lambda $, then it is also nonzero in $\Lambda $, and so the length of $q$ is less than $N_0$.
Otherwise, $q$ is of the form $q=q_r\cdots q_1$, where $\{t(q_i),s(q_{i+1})\}=\{x_1,E(x_1)\}$ for any $1\leq i<r$, and $q_1,\dots ,q_r$ are nonzero paths in $\Lambda $.
We claim that if $q$ is nonzero in $\Lambda _{E_1}$, then $r\leq 3$.

Suppose for a contradiction that $r\geq4$.
We assume that $t(p_1)=x_1$. Then $s(p_2)=E(x_1)$.  As $\Lambda _{E_1}$ is finite-dimensional and $p_2$ is nonzero in $\Lambda $, we have $t(p_2)\neq E(s(p_2))$, so $t(p_2)=E(x_1)$ by noting that $\{t(p_2),s(p_3)\}=\{x_1,E(x_1)\}$.
It follows that $s(p_3)=x_1$. Similarly, $t(p_3)=x_1=E(s(p_2))$. Then $p_3,p_2$ are two nonzero paths in $\Lambda $ such that $E(t(p_2))=s(p_3)$ and $t(p_3)=E(s(p_2))$, giving a contradiction to our assumption. Therefore, the length of $q$ is less than $3N_0$, and then $\Lambda _{E_1}$ is finite-dimensional. As $q$ is arbitrary, we get that $\Lambda _{E_1}$ is finite-dimensional.

We can view $\Lambda _{E_2}$ as a Br\"{u}stle's gluing algebra of $\Lambda _{E_1}$ by identifying $x_2$ and $E(x_2)$. In order to prove that $\Lambda _{E_2}$ is finite-dimensional, we only need to check that $\Lambda _{E_2}$ satisfies the condition: there is no nonzero paths $p_m,\dots,p_1$ in $\Lambda _{E_1}$ such that
$\{t(p_i), s(p_{i+1})\}=\{x_2,E(x_2)\}$ for any $i\in\Z/m\Z$.
Otherwise, $p_i$ is of the form $p_i=p_{ir_i}\cdots p_{i1}$, where $\{t(p_{ij}),s(p_{i,j+1})\}=\{x_1,E(x_1)\}$ for any $1\leq j<r_i$, and $p_{ir_i},\dots ,p_{i1}$ are nonzero paths in $\Lambda $.
So we get a series of nonzero paths in $\Lambda $:  $p_{mr_m},\dots,p_{m1}$, $\dots$, $p_{1r_1},\dots, p_{11}$, which gives a contradiction to our assumption on $\Lambda $.
\end{proof}

From the proof of Lemma \ref{lemma finite dimension of gluing algebras}, in order to prove properties for  Br\"{u}stle's gluing algebras, it usually reduces to prove the case that there is only one pair of vertices $(x,E(x))$ such that $x\neq E(x)$. This is one of the main technique used below.

In the following, we always assume that $\Lambda_E$ is finite-dimensional.

For $\Lambda =KQ/I$, and $E$ is an involution on the set of vertices of $Q$, we define another new quiver $\bar{Q}$ from $Q$ by adding an arrow $\alpha_{(x,E(x))}$ between $x$ and $E(x)$ (in either direction) if $x\neq E(x)$. Then $I$ can be viewed as a subset of $K\bar{Q}$. Let $\bar{I}$ be the ideal of $K\bar{Q}$ generated by $I$ and set $\bar{\Lambda }=K\bar{Q}/\bar{I}$. It is easy to see that if $\Lambda_E$ is finite-dimensional, then so is $\bar{\Lambda }$.

\begin{example}\label{example cover functor}
Let $\Lambda =KQ/I$ be the quiver algbra, where $Q$ is the quiver as Fig. 7. shows, and $I$ is generated by $\alpha_{i+2}\alpha_{i+1}\alpha_i$, for all $i\in\Z/6\Z$. Let
$E$ be the involution such that $E(3)=6$, and $E(i)=i$ otherwise. Then $\Lambda_E$ is the algebra $KQ(E)/I_E$ with $Q(E)$ as Fig. 8. shows,
and $I_E$ is generated by $\alpha_{i+2}\alpha_{i+1}\alpha_i$, for all $i\in\Z/6\Z$. $\bar{\Lambda }=K\bar{Q}/\bar{I}$ is the algebra with $\bar{Q}$ as Fig. 9. shows, and $\bar{I}$ is also
generated by $\alpha_{i+2}\alpha_{i+1}\alpha_i$, for all $i\in\Z/6\Z$.

\begin{center}\setlength{\unitlength}{0.6mm}
 \begin{picture}(180,40)
\put(70,0){\begin{picture}(60,60)
\put(0,0){\circle*{1.5}}
\put(0,20){\circle*{1.5}}
\put(20,10){\circle*{1.5}}
\put(40,0){\circle*{1.5}}
\put(40,20){\circle*{1.5}}

\put(0,2){\vector(0,1){16}}
\put(2,19){\vector(2,-1){16}}
\put(22,11){\vector(2,1){16}}
\put(40,18){\vector(0,-1){16}}
\put(38,1){\vector(-2,1){16}}
\put(18,9){\vector(-2,-1){16}}

\put(-2,-4){$_{\bar{1}}$}
\put(-2,23){$_{\bar{2}}$}
\put(18.5,6){$_{\bar{3}}$}
\put(39,23){$_{\bar{4}}$}
\put(39,-4){$_{\bar{5}}$}
\put(-6,6){$^{\alpha_1}$}
\put(9,14){$^{\alpha_2}$}
\put(25,14){$^{\alpha_3}$}
\put(41,6){$^{\alpha_4}$}
\put(25,0){$^{\alpha_5}$}
\put(9,0){$^{\alpha_6}$}

\put(-10,-20){Fig. 8. The quiver $Q(E)$}
\put(15,-28){in Example \ref{example cover functor}.}

\end{picture}}
\put(-15,0){\begin{picture}(60,60)
\put(0,0){\circle*{1.5}}
\put(0,20){\circle*{1.5}}
\put(20,0){\circle*{1.5}}
\put(20,20){\circle*{1.5}}
\put(40,20){\circle*{1.5}}
\put(40,0){\circle*{1.5}}

\put(0,2){\vector(0,1){16}}
\put(2,20){\vector(1,0){16}}
\put(22,20){\vector(1,0){16}}

\put(40,18){\vector(0,-1){16}}
\put(38,0){\vector(-1,0){16}}
\put(18,0){\vector(-1,0){16}}

\put(-2,-4){$_1$}
\put(-2,23){$_2$}
\put(19,-4){$_6$}
\put(19,23){$_3$}
\put(39,23){$_4$}
\put(39,-4){$_5$}
\put(-6,6){$^{\alpha_1}$}
\put(8,14){$^{\alpha_2}$}
\put(26,14){$^{\alpha_3}$}
\put(41,6){$^{\alpha_4}$}
\put(26,0){$^{\alpha_5}$}
\put(8,0){$^{\alpha_6}$}
\put(-10,-20){Fig. 7. The quiver $Q$}
\put(12,-28){in Example \ref{example cover functor}.}
\end{picture}}

\put(150,0){\begin{picture}(60,60)
\put(0,0){\circle*{1.5}}
\put(0,20){\circle*{1.5}}
\put(20,0){\circle*{1.5}}
\put(20,20){\circle*{1.5}}
\put(40,20){\circle*{1.5}}
\put(40,0){\circle*{1.5}}

\put(0,2){\vector(0,1){16}}
\put(2,20){\vector(1,0){16}}
\put(22,20){\vector(1,0){16}}

\put(40,18){\vector(0,-1){16}}
\put(38,0){\vector(-1,0){16}}
\put(18,0){\vector(-1,0){16}}

\put(20,18){\vector(0,-1){16}}
\put(21,7){$^{\alpha_{(3,6)}}$}
\put(-2,-4){$_1$}
\put(-2,23){$_2$}
\put(19,-4){$_6$}
\put(19,23){$_3$}
\put(39,23){$_4$}
\put(39,-4){$_5$}
\put(-6,6){$^{\alpha_1}$}
\put(8,14){$^{\alpha_2}$}
\put(26,14){$^{\alpha_3}$}
\put(41,6){$^{\alpha_4}$}
\put(26,0){$^{\alpha_5}$}
\put(8,0){$^{\alpha_6}$}
\put(-10,-20){Fig. 9. The quiver $\bar{Q}$}
\put(15,-28){in Example \ref{example cover functor}.}
\end{picture}}

\end{picture}
\vspace{1.8cm}
\end{center}

\end{example}

\subsection{Singularity categories of Br\"{u}stle's gluing algebras}

In this subsection, we prove that $D_{sg}(\bar{\Lambda })$, $D_{sg}(\Lambda_E)$ and $D_{sg}(\Lambda )$ are triangulated equivalent. For any two algebras $A$ and $B$, if $D_{sg}(A)$ and $D_{sg}(B)$ are triangulated equivalent, then we call $A$ and $B$ to be \emph{singularity equivalent}. %, similar to the definition of \emph{derived equivalent}.

Keep the notations as above.
W define a functor $\Phi: \mod \Lambda_E\rightarrow \mod \bar{\Lambda }$ as follows.
For any representation $M=(M_{\bar{i}},\phi_\alpha)_{\bar{i}\in Q(E)_0,\alpha\in Q(E)_1}$ of $(Q(E),I(E))$,
we define $\Phi(M)=(N_j,\psi_\beta)_{j\in \bar{Q}_0,\beta\in \bar{Q}_1}$ by $N_j=M_{\bar{i}}$ if $j\in \bar{i}=\{i,E(i)\}$, and $\psi_\beta=\phi_\beta$ for any $\beta\in Q(E)_1= Q_1$, and $\psi_\beta=\Id$ for any $\beta=\alpha_{x,E(x)}$. It is easy to see that $\Phi$ is an exact fully faithful functor, which induces that $\mod \Lambda_E$ is equivalent to the full subcategory of $\mod \bar{\Lambda }$ consisting of those modules $(N_j,\psi_\beta)_{j\in \bar{Q}_0,\beta\in \bar{Q}_1}$
with $\psi_\beta$ bijective whenever $\beta$ is any  added arrow $\alpha_{(x,E(x))}$.

Since $\Phi:\mod \Lambda_E\rightarrow \mod \bar{\Lambda }$ is exact, its derived functor is $D^b(\Phi): D^b(\mod \Lambda_E)\rightarrow D^b(\mod \bar{\Lambda })$. We shall prove that $D^b(\Phi)$ is fully faithful (not dense) in the following, which needs some lemmas.

\begin{lemma}\label{lemma exact sequence for Phi}
Keep the notations as above. Assume that there is only one pair of vertices $(x,E(x))$ such that $x\neq E(x)$, and the added arrow for $\bar{Q}$ is $\gamma=\alpha_{(x,E(x))}:x\rightarrow E(x)$. Then for any vertex $i\neq E(x)$ in $Q_0$, there are short exact sequences in $\mod \bar{\Lambda }$:
\begin{equation}\label{equation finite projective dimension of Phi 1}
0\rightarrow (\bar{\Lambda }(e_{E(x)}))^{\oplus t_i} \rightarrow \bar{\Lambda } e_i\oplus (\bar{\Lambda }e_x)^{\oplus t_i}\rightarrow \Phi(\Lambda_E e_{\bar{i}})\rightarrow0,
\end{equation}
and
\begin{equation}\label{equation finite projective dimension of Phi 2}
0\rightarrow \bar{\Lambda } e_{i}\rightarrow \Phi(\Lambda_E e_{\bar{i}})\rightarrow B_x^{\oplus t_i}\rightarrow0,
\end{equation}
for some $t_i$, where $B_x$ is the cokernel of the natural injective morphism $f_\gamma: \bar{\Lambda } e_{E(x)}\rightarrow  \bar{\Lambda } e_{x}$ induced by the right multiplication of the arrow $\gamma$.
\end{lemma}
\begin{proof}
Since $\Lambda_E$ is finite-dimensional, from Lemma \ref{lemma finite dimension of gluing algebras}, there is no nonzero path from $x$ to $E(x)$ or from $E(x)$ to $x$ in $\Lambda $.
In the following, we use the structure of projective modules and the definition of $\Phi$ to prove our desired results.

In order to describe the structure of projective modules, for simplicity, we assume that there is only one arrow $\alpha_1$ ending at $x$, only one arrow $\alpha_2$ starting at $x$, only one arrow $\alpha_3$ ending at $E(x)$, and only one arrow $\alpha_4$ starting at $E(x)$. Then the structure of the indecomposable projective module $\Lambda _Ee_{\bar{i}}$ is as Fig. 10. shows.
The structure of $\Phi(\Lambda _Ee_{\bar{i}})$ is as Fig. 11. shows. So  there exists a short exact sequence in $\mod \bar{\Lambda }$:
$$0\rightarrow \bar{\Lambda } e_{i}\rightarrow \Phi(\Lambda_E e_{\bar{i}})\rightarrow B_x^{\oplus t_i}\rightarrow0.$$
In fact, $t_i$ is the number of the arrows $\alpha_3$ in Fig. 10..

It is worth noting that $\Phi(\Lambda _Ee_{\bar{x}})=\bar{\Lambda }e_x$ by Fig. 11. since there is no nonzero path from $x$ to $E(x)$ in $\Lambda $.

\setlength{\unitlength}{0.8mm}
\begin{center}
\begin{picture}(100,65)(0,-20)
\put(-10,0){\begin{picture}(50,50)
\put(9,30){\small$\bar{i}$}
\put(9,29){\vector(-1,-1){6}}
\put(11,29){\vector(1,-1){6}}
\put(10,29){\vector(0,-1){6}}
\put(4,22){\small $\cdots$}
\put(11,22){\small $\cdots$}
\put(17,19){\small $\ddots$}
\put(9.5,18){\small $\vdots$}
\put(1.5,21){\small $\cdot$}
\put(0,19.5){\small $\cdot$}
\put(-1.5,18){\small $\cdot$}

\put(-2,18){\vector(-1,-1){6}}
\put(10,17){\vector(0,-1){6}}
\put(23,18){\vector(1,-1){6}}

\put(8.5,7.5){\small$\bar{x}$}
\put(-11,8){\small$\bar{x}$}
\put(29,8){\small$\bar{x}$}

\put(9,6){\vector(-1,-1){6}}
\put(11,6){\vector(1,-1){6}}

\put(-12,7){\vector(-1,-1){6}}
\put(-10,7){\vector(1,-1){6}}

\put(30,7){\vector(-1,-1){6}}
\put(32,7){\vector(1,-1){6}}
\put(-3,8){\small $\cdots$}
\put(17,8){\small $\cdots$}

\put(-9,16){\tiny $\alpha_3$}
\put(11,15){\tiny $\alpha_1$}
\put(26,16){\tiny $\alpha_3$}

\put(-19,5){\tiny $\alpha_2$}
\put(-7,5){\tiny $\alpha_4$}

\put(2,4){\tiny $\alpha_2$}
\put(13,4){\tiny $\alpha_4$}

\put(23,5){\tiny $\alpha_2$}
\put(34,5){\tiny $\alpha_4$}

\put(-18.5,-4){\small $\vdots$}

\put(-4.5,-4){\small $\vdots$}

\put(2,-5){\small $\vdots$}
\put(16,-5){\small $\vdots$}

\put(23.5,-4){\small $\vdots$}
\put(37.5,-4){\small $\vdots$}

\put(-30,-20){Fig. 10. The structure of indecomposable}
\put(-14,-26){projective module $\Lambda _Ee_{\bar{i}}$.}
\end{picture}}
\put(85,0){\begin{picture}(50,50)
\put(9,30){\small$i$}
\put(9,29){\vector(-1,-1){6}}
\put(11,29){\vector(1,-1){6}}
\put(10,29){\vector(0,-1){6}}
\put(4,22){\small $\cdots$}
\put(11,22){\small $\cdots$}
\put(17,19){\small $\ddots$}
\put(9.5,18){\small $\vdots$}
\put(1.5,21){\small $\cdot$}
\put(0,19.5){\small $\cdot$}
\put(-1.5,18){\small $\cdot$}

\put(-2,18){\vector(-1,-1){6}}
\put(10,17){\vector(0,-1){6}}
\put(23,18){\vector(1,-1){6}}

\put(8.5,7.5){\small$x$}
\put(-13,8){\tiny$E({x})$}
\put(27,8){\tiny$E(x)$}

\put(9,6){\vector(-1,-1){6}}
\put(11,6){\vector(1,-1){6}}

\put(-10,7){\vector(1,-1){6}}

\put(30,7){\vector(-1,-1){6}}
\put(39,18){\vector(-1,-1){6}}

\put(39,19){\small$x$}
\put(41,18){\vector(1,-1){6}}
\put(43,16){\tiny $\alpha_2$}
\put(47,7){\small $\vdots$}

\put(34,16){\tiny $\gamma$}

\put(-19,18){\vector(1,-1){6}}
\put(-21,19){\small $x$}
\put(-21,18){\vector(-1,-1){6}}

\put(-28,7){\small $\vdots$}
\put(-28,16){\tiny $\alpha_2$}
\put(-17,16){\tiny $\gamma$}

\put(-3,8){\small $\cdots$}
\put(17,8){\small $\cdots$}

\put(-9,16){\tiny $\alpha_3$}
\put(11,15){\tiny $\alpha_1$}
\put(26,16){\tiny $\alpha_3$}

\put(-7,5){\tiny $\alpha_4$}

\put(2,4){\tiny $\alpha_2$}
\put(13,4){\tiny $\gamma$}

\put(13,-3){\tiny $E(x)$}
\put(16.5,-4){\vector(0,-1){6}}

\put(23,5){\tiny $\alpha_4$}

\put(-4.5,-4){\small $\vdots$}

\put(2,-5){\small $\vdots$}
\put(16,-15){\small $\vdots$}

\put(23.5,-4){\small $\vdots$}

\put(17,-7){\tiny $\alpha_4$}

\qbezier[40](40,28)(35,14)(40,0)
\qbezier[40](50,28)(55,14)(50,0)
\qbezier[15](40,28)(45,33)(50,28)
\qbezier[15](40,0)(45,-5)(50,0)
\put(43,22){$B_x$}

\qbezier[40](-29,28)(-34,14)(-29,0)
\qbezier[40](-19,28)(-14,14)(-19,0)
\qbezier[15](-29,28)(-24,33)(-19,28)
\qbezier[15](-29,0)(-24,-5)(-19,0)

\put(-28,22){$B_x$}

\qbezier[50](-9,32)(-20,11)(-9,-10)
\qbezier[50](30,32)(41,11)(30,-10)
\qbezier[50](-9,32)(10.5,47)(30,32)
\qbezier[50](-9,-10)(10.5,-25)(30,-10)

\put(-5,27){$ \bar{\Lambda }e_i$}

\put(-30,-26){Fig. 11. The structure of $\Phi(\Lambda _Ee_{\bar{i}})$.}
\end{picture}}
\end{picture}
\vspace{0.5cm}
\end{center}
By definition, there is a short exact sequence $0\rightarrow \bar{\Lambda }e_{E(x)} \xrightarrow{f_\gamma} \bar{\Lambda }e_x\rightarrow B_x\rightarrow0$. Then we get the following commutative diagram:

\[\xymatrix{ &(\bar{\Lambda } e_{E(x)})^{\oplus t_i} \ar@{=}[r] \ar[d] & (\bar{\Lambda } e_{E(x)})^{\oplus t_i}\ar[d]\\
\bar{\Lambda }e_i\ar[r]\ar@{=}[d] & \bar{\Lambda }e_i\oplus (\bar{\Lambda }e_x)^{\oplus t_i}\ar[r]\ar[d]& (\bar{\Lambda }e_x)^{\oplus t_i} \ar[d]\\
\bar{\Lambda }e_i\ar[r]& \Phi(\Lambda_E e_{\bar{i} })\ar[r]& B_x^{\oplus t_i}.
}\]
So the short exact sequence (\ref{equation finite projective dimension of Phi 1}) follows.
\end{proof}

\begin{lemma}\label{lemma homological property of Bx}
Keep the notations as above. Assume that there is only one pair of vertices $(x,E(x))$ such that $x\neq E(x)$, and the added arrow for $\bar{Q}$ is $\gamma=\alpha_{(x,E(x))}:x\rightarrow E(x)$. Let $B_x$ be the cokernel of the natural injective morphism $f_\gamma: \bar{\Lambda } e_{E(x)}\rightarrow  \bar{\Lambda } e_{x}$ induced by the right multiplication of the arrow $\gamma$. Then
$$\Hom_{\bar{\Lambda }}(B_x,\Phi(M))=0=\Ext^i_{\bar{\Lambda }}(B_x,\Phi(M))$$
for any $i\geq1$ and $M\in\mod \Lambda_E$. In particular, $\End_{\bar{\Lambda }}(B_x)\cong K$, i.e., $B_x$ is a brick.
\end{lemma}
\begin{proof}
For any morphism $f:B_x\rightarrow \Phi(M)$, by restricting to the full subquiver formed by the added arrow $\gamma$,  it is easy to see that $f_x: (B_x)_x\rightarrow (\Phi(M))_x$ is zero. Together with $\Top(B_x)=S_x$, we have $\Top(\pi): \Top(B_x)\rightarrow \Top(\Im f)$ is zero, where $\pi: B_x\rightarrow \Im f$ is the epimorphism induced by $f$. So $\pi=0$ and then $ f=0$. Therefore, $\Hom_{\bar{\Lambda }}(B_x,\Phi(M))=0$.

On the other hand, $\Hom_{\bar{\Lambda }}(\bar{\Lambda }e_{x},\Phi(M))\cong e_x\Phi(M)$, $\Hom_{\bar{\Lambda }}(\bar{\Lambda }e_{E(x)},\Phi(M))\cong e_{E(x)}\Phi(M)$. It follows from the definition of $\Phi$ that
$\dim_K\Hom_{\bar{\Lambda }}(\bar{\Lambda }e_{x},\Phi(M))=\dim_K \Hom_{\bar{\Lambda }}(\bar{\Lambda }e_{E(x)},\Phi(M))$. By applying $\Hom_{\bar{\Lambda }}(-,\Phi(M))$ to
$$0\rightarrow \bar{\Lambda }e_{E(x)} \xrightarrow{f_\gamma} \bar{\Lambda }e_x\rightarrow B_x\rightarrow0,$$
we have a long exact sequence, $$0=\Hom_{\bar{\Lambda }}(B_x,\Phi(M))\rightarrow  \Hom_{\bar{\Lambda }}(\bar{\Lambda }e_{x},\Phi(M)) \rightarrow\Hom_{\bar{\Lambda }}(\bar{\Lambda }e_{E(x)},\Phi(M))\rightarrow \Ext^1_{\bar{\Lambda }}(B_x,\Phi(M))\rightarrow0.$$
Then $\Ext^1_{\bar{\Lambda }}(B_x,\Phi(M))=0$. Our desired result follows by noting that $\pd_{\bar{\Lambda }}(B_x)\leq1$.

From the structure of $B_x$, it is easy to see that $\rad (B_x)\subseteq\Im\Phi$. By applying $\Hom_{\bar{\Lambda }}(B_x,-)$ to the short exact sequence
$$0\rightarrow \rad(B_x)\rightarrow B_x\rightarrow S_x\rightarrow0,$$
we have $\End_{\bar{\Lambda }}(B_x)\cong K$.
\end{proof}

\begin{proposition}
\label{prop:fully faithful}
Keep the notations as above. Assume that there is only one pair of vertices $(x,E(x))$ such that $x\neq E(x)$, and the added arrow for $\bar{Q}$ is $\gamma=\alpha_{(x,E(x))}:x\rightarrow E(x)$. Then the derived functor of $\Phi:D^b(\mod \Lambda_E)\rightarrow D^b(\mod \bar{\Lambda })$ is fully faithful.
\end{proposition}

\begin{proof}
Since $\Phi:\mod \Lambda_E\rightarrow \mod \bar{\Lambda }$ is exact, its derived functor is $D^b(\Phi): D^b(\mod \Lambda_E)\rightarrow D^b(\mod \bar{\Lambda })$. In order to prove  that $D^b(\Phi)$ is fully faithful, it is enough to prove that $\Phi$ induces
\begin{align}
\Ext^i_{\bar{\Lambda }}(\Phi(M),\Phi(N))\cong \Ext^i_{\Lambda_E}(M,N)\text{ for any } M,N\in\mod \Lambda_E, \text{ and }i>0.
\end{align}

Take a projective resolution of $M$
$$\cdots \xrightarrow{f^{n+1}} U^n\xrightarrow{f^n}\cdots \xrightarrow{f^1}U^0\xrightarrow{f^0}M\rightarrow0.$$
Denote by $M^i=\Im f^i$ for any $i\geq0$. Note that $M^i$ is a $i$-th syzygy of $M$, and $M^0=M$.

First, we check that $\Phi$ induces that $\Ext^1_{\bar{\Lambda }}(\Phi(M),\Phi(N))\cong \Ext^1_{\Lambda_E}(M,N)$. In fact, we get a short exact sequence
\begin{equation}\label{equation 2}
0\rightarrow \Phi(M^1)\rightarrow\Phi(U^0)\xrightarrow{\Phi(f^0)} \Phi(M)\rightarrow0.
\end{equation}
From (\ref{equation finite projective dimension of Phi 2}), there is a short exact sequence
\begin{equation}\label{eqn:3}
0\rightarrow Q^0\rightarrow\Phi(U^0)\rightarrow B_x^{\oplus t_0}\rightarrow0
\end{equation}
for some projective $\bar{\Lambda }$-module $Q^0$.
By applying $\Hom_{\bar{\Lambda }}(-,\Phi(N))$ to (\ref{eqn:3}), from Lemma \ref{lemma homological property of Bx}, we get that
$$\Ext^i_{\bar{\Lambda }}(\Phi(U^0),\Phi(N))=0,\text{ for any }i>0.$$
Applying $\Hom_{\bar{\Lambda }}(-,\Phi(N))$ to (\ref{equation 2}) yields the following long exact sequence
$$\Hom_{\bar{\Lambda }}(\Phi(U^0),\Phi(N))\rightarrow \Hom_{\bar{\Lambda }}(\Phi(M^1),\Phi(N))\rightarrow \Ext^1_{\bar{\Lambda }}(\Phi(M),\Phi(N))\rightarrow0.$$
Because $\Phi$ is fully faithful, it is easy to see that $\Ext^1_{\bar{\Lambda }}(\Phi(M),\Phi(N))\cong \Ext^1_{\Lambda_E}(M,N)$.

For $i>1$, by induction, we get that $\Phi$ induces that $\Ext_{\bar{\Lambda }}^{i-1}(\Phi(M^1),\Phi(N))\cong \Ext^{i-1}_{\Lambda_E}(M^1,N)$, and then
$$\Ext^i_{\bar{\Lambda }}(\Phi(M),\Phi(N))\cong \Ext_{\bar{\Lambda }}^{i-1}(\Phi(M^1),\Phi(N))\cong \Ext^{i-1}_{\Lambda_E}(M^1,N)\cong \Ext^i_{\Lambda_E}(M,N).$$

From the above, we have that $D^b(\Phi)$ induces $$\Hom_{D^b(\bar{\Lambda })}(D^b(\Phi)(M^\bullet), D^b(\Phi)(N^\bullet))=\Hom_{D^b(\Lambda_E)}(M^\bullet, N^\bullet)$$ for any stalk complexes $M^\bullet$, $N^\bullet$.
Since bounded derived category is generated by stalk complexes, we have that $D^b(\Phi)$ is fully faithful.
\end{proof}

By abusing notations, we use $\Phi$ to denote its derived functor $D^b(\Phi)$ in the following.

\begin{lemma}[\text{\cite[Lemma 4.7.1]{Kra}}]\label{lemma fully faithful of functors in quotient categories}
Let $\ct$ be a triangulated category with two full triangulated subcategories $\ct'$ and $\cs$. Set $\cs'=\cs\cap \ct'$. Then the natural inclusion $\ct'\hookrightarrow \ct$ induces an exact functor $J:\ct'/\cs'\rightarrow \ct/\cs$. Suppose that either
\begin{itemize}
\item[(a)] every morphism from an object in $\cs$ to an object in $\ct'$ factors through some object in $\cs'$, or
\item[(b)] every morphism from an object in $\ct'$ to an object in $\cs$ factors through some object in $\cs'$.
\end{itemize}
Then the induced functor $J:\ct'/\cs'\rightarrow \ct/\cs$ is fully faithful.
\end{lemma}

With the help of Lemma \ref{lemma fully faithful of functors in quotient categories} and Proposition \ref{prop:fully faithful}, we shall prove that $\Phi:\mod \Lambda_E\rightarrow \mod \bar{\Lambda }$ induces a triangle equivalence $\widetilde{\Phi}:D_{sg}(\Lambda_E)\rightarrow D_{sg}( \bar{\Lambda })$ in the following.

\begin{proposition}\label{proposition singularity equivalence of SE barS}
Keep the notations as above. Assume that there is only one pair of vertices $(x,E(x))$ such that $x\neq E(x)$, and the added arrow for $\bar{Q}$ is $\gamma=\alpha_{(x,E(x))}:x\rightarrow E(x)$. Then the exact functor $\Phi:\mod \Lambda_E\rightarrow \mod \bar{\Lambda }$ induces a triangle equivalence $\widetilde{\Phi}:D_{sg}(\Lambda_E)\rightarrow D_{sg}( \bar{\Lambda })$.
\end{proposition}
\begin{proof}
From (\ref{equation finite projective dimension of Phi 1}),
we get that $\pd_{\bar{\Lambda }}\Phi(M)< \infty$ if $\pd_{\Lambda_E} M<\infty$ for any finitely generated $\Lambda_E$-module $M$ since $\Phi$ is exact.
So $\Phi(K^b(\proj \Lambda_E)) \subseteq K^b(\proj \bar{\Lambda })$, which implies that $\Phi$ induces a triangulated  functor $\widetilde{\Phi}:D_{sg}(\Lambda_E)\rightarrow D_{sg}( \bar{\Lambda })$.
In the following, on the basis of Proposition \ref{prop:fully faithful}, we use Lemma \ref{lemma fully faithful of functors in quotient categories} to prove that $\widetilde{\Phi}$ is fully faithful.

First, we check that $\Phi(K^b(\proj \Lambda_E)) =K^b(\proj \bar{\Lambda })\cap \Phi(D^b(\mod \Lambda _E))$. From the above, we get that $\Phi(K^b(\proj \Lambda _E)) \subseteq K^b(\proj \bar{\Lambda })\cap \Phi(D^b(\mod \Lambda _E))$. Conversely, for any object in $K^b(\proj \bar{\Lambda })\cap \Phi(D^b(\mod \Lambda _E))$, it is of the form
$$\Phi(X^\bullet)=\cdots \xrightarrow{\Phi(d^{i-1})} \Phi(X^{i})\xrightarrow{\Phi(d^{i})} \Phi(X^{i+1})\xrightarrow{\Phi(d^{i+1})}\Phi(X^{i+2})\xrightarrow{\Phi(d^{i+2})} \cdots$$
for some bounded complex
$$X^\bullet=\cdots \xrightarrow{d^{i-1}} X^{i}\xrightarrow{d^{i}} X^{i+1}\xrightarrow{d^{i+1}} X^{i+2}\xrightarrow{d^{i+2}} \cdots.$$
Since $\Phi(X^\bullet)\in K^b(\proj \bar{\Lambda })$, from the dual of \cite[Lemma 4.1 b)]{Ke0}, there is an epimorphism of complexes $f^\bullet:Q^\bullet\rightarrow \Phi(X^\bullet)$ such that $f^\bullet$ is a quasi-isomorphism, where
$Q^\bullet= \cdots \xrightarrow{e^{i-1}} Q^{i}\xrightarrow{e^{i}} Q^{i+1}\xrightarrow{e^{i+1}} Q^{i+2}\xrightarrow{e^{i+2}} \cdots$ is a bounded complex with $Q^i\in \proj \bar{\Lambda }$ for all $i$. Here $f^\bullet=(f^i:Q^i\rightarrow \Phi(X^i))_i$.
For any $Q^i$, by (\ref{equation finite projective dimension of Phi 2}), there exists
$P^i\in\proj \Lambda_E$ such that
\begin{equation}\label{equation projective resolution of Bx}
0\rightarrow Q^i\xrightarrow{g^i}\Phi(P^i) \rightarrow B_x^{\oplus s_i}\rightarrow0
\end{equation}
is exact for some $s_i>0$. By applying $\Hom_{\bar{\Lambda}}(-, \Phi(X^i))$ to (\ref{equation projective resolution of Bx}), using
Lemma \ref{lemma homological property of Bx}, we have that $f^i$ factors through $g^i$, and then there exists $h^i:P^i\rightarrow X^i$ such that $\Phi(h^i)g^i=f^i$ for any $i$ since $\Phi$ is fully faithful.

Similar to the above, by applying $\Hom_{\bar{\Lambda }}(-,\Phi(P^{i+1}))$ to (\ref{equation projective resolution of Bx}), there exists $p^i: P^i\rightarrow P^{i+1}$ such that the following diagram is commutative for each $i$ since $\Phi$ is fully faithful:
\[\xymatrix{ \Phi(P^i)\ar[rr]^{\Phi(p^i)}  && \Phi(P^{i+1}) \\
Q^i\ar[rr]^{e^i} \ar@{>->}[u]^{g^i} && Q^{i+1}.\ar@{>->}[u]^{g^{i+1}}   }\]
It follows from $\Phi(p^{i+1}p^i)g^i=g^{i+2}e^{i+1}e^i=0$ that $\Phi(p^{i+1}p^i)$ factors through $B_x^{\oplus s_i}$, which implies that
$\Phi(p^{i+1}p^i)=0$ since $\Hom_{\bar{\Lambda }}(B_x, \Phi(P_{i+2}))=0$ by Lemma \ref{lemma homological property of Bx}. So $p^{i+1}p^i=0$ for any $i$, and then $P^\bullet=(P^i,p^i)_i$ is a complex in $K^b(\proj \Lambda_E)$. In particular $g^\bullet=(g^i: Q^i\rightarrow \Phi(P^i))_i$ is a morphism of complexes.

We have that $h^\bullet=(h^i: P^i\rightarrow X^i)_i$ is a morphism of complexes. In fact, $$\Phi(h^{i+1} p^i) g^i= \Phi(h^{i+1}) g^{i+1} e^i =f^{i+1}e^i=\Phi(d^i) f^i= \Phi(d^ih^i)g^i,$$
so $\Phi(d^ih^i-h^{i+1}p^i)g^i=0$, which implies that $\Phi(d^ih^i-h^{i+1}p^i)$ factors through $B_x^{\oplus s_i}$ and then $\Phi(d^ih^i-h^{i+1}p^i)=0$ similarly to the above. Since $\Phi$ is fully faithful, we have $h^{i+1}p^i=d^ih^i$ for any $i$.
%Then $h^\bullet=(h^i: P^i\rightarrow X^i)_i$ is a morphism of complexes.

We claim that $h^\bullet$ is a quasi-isomorphism. As $f^\bullet$ is a quasi-isomorphism and $\Phi$ is fully faithful and exact, $\Phi(f^\bullet)$ is also a quasi-isomorphism. From the fact $\Phi(h^\bullet) g^\bullet=\Phi(f^\bullet)$, we only need to check that $g^\bullet$ is a quasi-isomorphism.

Let $B^\bullet=(B^i,u_i)_i$ be the cokernel of $g^\bullet$. Then $B^i=B_x^{\oplus s_i}$. This gives a short exact sequence of complexes
$0\rightarrow Q^\bullet\xrightarrow{g^\bullet} \Phi(P^\bullet)\rightarrow B^\bullet\rightarrow0$, and then a long exact sequence of cohomology groups:
$$\cdots \to H^i(Q^\bullet) \rightarrow H^i(\Phi(P^\bullet))\rightarrow H^i(B^\bullet) \rightarrow H^{i+1}(Q^\bullet) \rightarrow\cdots. $$

Note that $\End_{\bar{\Lambda }}(B_x)=K$ by Lemma \ref{lemma homological property of Bx}. Then $H^i(B^\bullet)= B_x^{\oplus t_i}$ for some $t_i$.
On the other hand, $H^{i+1}(Q^\bullet)\cong H^{i+1}(\Phi(X^\bullet))=\Phi(H^{i+1}(X^\bullet))$, so $\Hom_{\bar{\Lambda }}( H^i(B^\bullet), H^{i+1}(Q^\bullet))=0$. Then $H^i(B^\bullet)$ is the cokernel of $H^i(Q^\bullet) \rightarrow H^i(\Phi(P^\bullet))$. Since both
$H^i(Q^\bullet)$ and $H^i(\Phi(P^\bullet))$ are in the image of $\Phi$, we get that $H^i(B^\bullet)\in\Im(\Phi)$. Together with $H^i(B^\bullet)= B_x^{\oplus t_i}$,  then $H^i(B^\bullet)=0$, and so $B^\bullet$ is acyclic. Therefore, $g^\bullet$ is a quasi-isomorphism. In conclusion, $h^\bullet$ is a quasi-isomorphism, and $X^\bullet\in K^b(\proj \Lambda_E)$ since $P^\bullet\in K^b(\proj \Lambda_E)$.

In  conclusion, we get that $\Phi(K^b(\proj \Lambda_E)) =K^b(\proj \bar{\Lambda })\cap \Phi(D^b(\mod \Lambda_E))$.

Second, we prove that $\Phi(K^b(\proj \Lambda_E))\subseteq \Phi(D^b(\mod \Lambda_E))$, $K^b(\proj \bar{\Lambda })\subseteq D^b(\mod \bar{\Lambda })$ satisfy Lemma \ref{lemma fully faithful of functors in quotient categories} (a). In fact, for any $Q^\bullet=(Q^i,e^i)_i\in K^b(\proj \bar{\Lambda })$, $X^\bullet=(X^i,d^i)_i\in D^b(\mod \Lambda_E)$, and any morphism $Q^\bullet\rightarrow \Phi(X^\bullet)$, from (\ref{equation projective resolution of Bx}) and similar to the above, there exists a complex $P^\bullet=(P^i,p^i)_i$ with $P^i \in\proj \Lambda_E$ such that there is a short exact sequences of bounded  complexes:
$0\rightarrow Q^\bullet\xrightarrow{g^\bullet} \Phi(P^\bullet)\rightarrow B^\bullet\rightarrow0$,
where $B^i=B_x^{\oplus s_i}$.
%there exists a short exact sequence for each $Q^i$:
%\begin{equation*}
%0\rightarrow Q^i\xrightarrow{g^i}\Phi(P^i) \rightarrow B_x^{\oplus s_i}\rightarrow0
%\end{equation*}
%where $P^i\in\proj \Lambda_E$.
%Similar to the above, there exists $p^i:P^i\rightarrow P^{i+1}$ such that the following diagram is commutative for each $i$:
%\[\xymatrix{ \Phi(P^i)\ar[rr]^{\Phi(p^i)}  && \Phi(P^{i+1}) \\
%Q^i\ar[rr]^{e^i} \ar@{>->}[u]^{g^i} && Q^{i+1},\ar@{>->}[u]^{g^{i+1}}   }\]
%nd $P^\bullet=(P^i,p^i)_i$ is a complex in $K^b(\proj \Lambda_E)$. Furthermore, similar to the above, $f^i: Q^i\rightarrow \Phi(X^i)$ factors through $g^i$ as $f^i=\Phi(h^i)g^i$ for some morphism $h^i:P^i\rightarrow X^i$ for each $i$, and $h^\bullet=(h^i:P^i\rightarrow X^i)_i: P^\bullet\rightarrow X^\bullet$ is a morphism of complexes.
So we get that $f^\bullet=\Phi(h^\bullet)g^\bullet$ factors through $\Phi(P^\bullet)$ which is in $\Phi(K^b(\proj \Lambda_E))$.
So $\Phi(K^b(\proj \Lambda_E))\subseteq \Phi(D^b(\mod \Lambda_E))$, $K^b(\proj \bar{\Lambda })\subseteq D^b(\mod \bar{\Lambda })$ satisfy Lemma \ref{lemma fully faithful of functors in quotient categories} (a), and then the induced functor $\widetilde{\Phi}: D_{sg}(\Lambda_E)\rightarrow D_{sg}(\bar{\Lambda })$ is fully faithful.

Finally, we check that $\widetilde{\Phi}$ is dense.

For any object in $D^b(\bar{\Lambda })$, we can assume that it is of the form $(Q^i,d^i)_i\in K^{-,b}(\proj \bar{\Lambda})$ by the equivalence $D^b(\bar{\Lambda})\simeq K^{-,b}(\proj \bar{\Lambda} )$. Similar to the above, there is a short exact sequences of bounded above complexes:
$0\rightarrow Q^\bullet\xrightarrow{g^\bullet} \Phi(P^\bullet)\rightarrow B^\bullet\rightarrow0$,
where $B^i=B_x^{\oplus s_i}$.
For $i$ sufficiently small, $H^i(Q^\bullet)=0$, and then
$H^i(\Phi(P^\bullet))\cong H^i(B^\bullet)$.
Since $H^i(\Phi(P^\bullet))\cong \Phi(H^i(P^\bullet))$ and $H^i(B^\bullet)=B_x^{\oplus t_i}$ for some $t_i$,
both of them are zero for $i$ sufficiently small.
Since $\pd_{\bar{\Lambda }}\Phi(P^i)\leq 1,\pd_{\bar{\Lambda }}B_x\leq1$ for any $i$ by Lemma \ref{lemma exact sequence for Phi}, we have $\Phi(P^\bullet),B^\bullet\in K^{-,b}(\proj \bar{\Lambda})$.

For $B^\bullet=(B^i,b^i)_i$, As $\Hom_{\bar{\Lambda }}(B_x,B_x)=K$, obviously, $\Im b^i\in \add B_x$. Denote by $n$ the number such that $H^i(B^\bullet)=0$ for any $i\leq n$. Then
there is a commutative diagram of complexes:
\[\xymatrix{ \cdots \ar[r]& B^{n-2} \ar@{=}[d] \ar[r]^{b^{n-2}} & B^{n-1} \ar[r] \ar@{=}[d] & \Im b^{n-1} \ar@{>->}[d] \ar[r] &0 \ar[r]\ar[d]&0\ar[r]\ar[d] &\cdots\\
\cdots \ar[r]& B^{n-2} \ar[d] \ar[r]^{b^{n-2}} & B^{n-1} \ar[r]^{b^{n-1}} \ar[d] & B^{n} \ar@{->>}[d] \ar[r]^{b^n} &B^{n+1} \ar[r]^{b^{n+1}}\ar@{=}[d]&B^{n+2}\ar[r]\ar@{=}[d] &\cdots\\
\cdots \ar[r]&0  \ar[r] & 0\ar[r] & B^{n}/\Im b^{n-1}  \ar[r] &B^{n+1} \ar[r]^{b^{n+1}}&B^{n+2}\ar[r]&\cdots.}\]
By the construction, the complex in the first row is acyclic, and $\Im b^{n-1},B^n/\Im b^{n-1}\in \add B_x$. Then $B^\bullet$ is in $K^b(\proj \bar{\Lambda })$ since it is isomorphic to the complex in the third row in $D^b(\mod \bar{\Lambda })$.
From $0\rightarrow Q^\bullet\xrightarrow{g^\bullet} \Phi(P^\bullet)\rightarrow B^\bullet\rightarrow0$, we get that $Q^\bullet\cong \Phi(P^\bullet)$ in $D_{sg}(\bar{\Lambda })$, and then $\widetilde{\Phi}$ is dense.

Therefore, $\widetilde{\Phi}$ is a triangulated equivalence.
\end{proof}

In the following, we prove that $D_{sg}(\bar{\Lambda })$ and $D_{sg}(\Lambda )$ are triangulated equivalent.
Note that $\Lambda $ is a subalgebra of $\bar{\Lambda }$, and then $\bar{\Lambda }$ can be viewed as a left (also right) $\Lambda $-module. In particular, $\bar{\Lambda }$ is projective as a left (also right) $\Lambda $-module.
Then there is an adjoint triple $(j_*,j^*,j_!)$,
where $j_*=\bar{\Lambda }\otimes_\Lambda -:\mod \Lambda \rightarrow \mod \bar{\Lambda }$, $j^*=\,_\Lambda \bar{\Lambda }\otimes _{\bar{\Lambda }}-:\mod \bar{\Lambda }\rightarrow \mod \Lambda $, and $j_!=\Hom_\Lambda (_\Lambda \bar{\Lambda },-):\mod \bar{\Lambda }\rightarrow \mod\Lambda$. All of these three functors are exact, and they induce triangulated functors on bounded derived categories, which are denoted by the same notations respectively. Note that $j^*$ is the restriction functor.

\begin{lemma}\label{lemma singularity equivalence of S and barS}
Keep the notations as above. Assume that there is only one pair of vertices $(x,E(x))$ such that $x\neq E(x)$, and the added arrow for $\bar{Q}$ is $\gamma=\alpha_{(x,E(x))}:x\rightarrow E(x)$. Then $\Lambda$ and $\bar{\Lambda }$ are singular equivalent.
\end{lemma}
\begin{proof}
Because $\bar{\Lambda }$ is projective as a left (also right) $\Lambda $-module, both of $j_*$ and $j^*$ preserve projective modules. Since $j_*$ and $j^*$ are exact, they induce triangulated functors on singularity categories, which are denoted by $\tilde{j}_*$ and $\tilde{j}^*$ respectively.
We prove that $\tilde{j}_*:D_{sg}(\Lambda)\rightarrow D_{sg}(\bar{\Lambda})$ is a triangulated equivalence in the following.

Let $S_i$ be the simple $\Lambda$ (resp. $\bar{\Lambda}$)-modules for $i\in Q_0$. Then
\begin{equation}\label{eqn:action of j on simple modules}
j_*(S_i)=\left\{ \begin{array}{cc} S_i &  \text{ if }i\neq x,\\
C_x& \text{ if }i=x. \end{array}\right.
\end{equation}
where $C_x$ is uniquely determined by the short exact sequence
\begin{equation}\label{eqn:short exact sequnce of Sx}
0\rightarrow \bar{\Lambda}e_{E(x)} \rightarrow C_x \xrightarrow{g} S_x\rightarrow0
\end{equation}
with the first morphism induced by the right multiplication of $\gamma$.
In particular, $j^*j_*(S_i)=S_i$ if $i\neq x$, and $j^*j_*(S_x)=S_x\oplus \Lambda e_{E(x)}$.

By induction on the length of modules, one can get that the adjunction $\Id\rightarrow j^*j_*$ induces that $M$ is a direct summand of $j^*j_*(M)$, and $j^*j_*(M)\cong M\oplus (\Lambda  e_{E(x)})^{\oplus \dim M_x}$
for any $\Lambda $-module $M=(M_i,\phi_\alpha)_{i\in Q_0,\alpha\in Q_1}$. Furthermore, for any $f:M\rightarrow N$ in $\mod\Lambda$,
$j^*j_*(f)$ is of the form
$$\left( \begin{array}{cc} f & 0\\ 0& g\end{array} \right): M\oplus (\Lambda  e_{E(x)})^{\oplus \dim M_x}\rightarrow N\oplus (\Lambda  e_{E(x)})^{\oplus \dim N_x}$$
where $g: (\Lambda  e_{E(x)})^{\oplus \dim M_x}\rightarrow (\Lambda  e_{E(x)})^{\oplus \dim N_x}$ can be represented by a $(\dim N_x)\times (\dim M_x)$-matrix over $K$.
In particular, we have $\tilde{j}^* \tilde{j}_* \simeq \Id_{D_{sg}(\Lambda )}$ since $\mod \Lambda$ generates $D_{sg}(\mod\Lambda)$. Therefore, $\tilde{j}_*$ is fully faithful.

On the other hand, from the above, we have $S_i\in \Im(\tilde{j}_*)$ for any $i\neq x$. Together with $S_x\cong C_x$ in $D_{sg}(\bar{\Lambda})$, then $S_i\in\Im(\tilde{j}_*)$ for any $i\in Q_0$. Thus, $\tilde{j}_*$ is dense.
Therefore, $\tilde{j}_*:D_{sg}(\Lambda)\rightarrow D_{sg}(\bar{\Lambda})$ is a triangulated equivalence.
\end{proof}

Now we get the following result.
\begin{theorem}\label{theorem singularity equivalence of gluing vertices}
Let $\Lambda =KQ/I$ be a finite-dimensional algebra, and $E$ an involution on the set of vertices of $Q$. Assume that the algebra $\Lambda_E$ constructed from $\Lambda $ by gluing the vertices along $E$ is finite-dimensional. Then $\Lambda $ and $\Lambda_E$ are singularity equivalent.
\end{theorem}
\begin{proof}
Similarly to the same induction as in the proof of Lemma \ref{lemma finite dimension of gluing algebras}, it is enough to prove the result for the case that there exists only one pair of vertices $(x,E(x))$ such that $x\neq E(x)$, which follows by combining Proposition \ref{proposition singularity equivalence of SE barS} and Lemma \ref{lemma singularity equivalence of S and barS}.
\end{proof}

At the end of this subsection, we consider the Gorenstein property of $\Lambda$ and $\Lambda_E$.
\begin{lemma}\label{lemma Gorenstein of S and barS}
Keep the notations as above. Assume that there is only one pair of vertices $(x,E(x))$ such that $x\neq E(x)$, and the added arrow for $\bar{Q}$ is $\gamma=\alpha_{(x,E(x))}:x\rightarrow E(x)$.
Then $\Lambda $ is Gorenstein if and only if $\bar{\Lambda }$ is Gorenstein.
\end{lemma}
\begin{proof}
Note that $\Lambda $ is a quotient algebra of $\bar{\Lambda }$, and then there is a fully embedding $i_*:\mod \Lambda \rightarrow \mod \bar{\Lambda }$.
We use the exact functors $i_*,j_*,j^*,j_!$ to prove this result.

For the ``if'' part, as $\bar{\Lambda }$ is Gorenstein, for any indecomposable injective $\bar{\Lambda }$-module $D(\bar{\Lambda })e_i$,  we  take
a minimal projective resolution of $D(\bar{\Lambda })e_i$:
\begin{equation}\label{eqn:projective resolution of injective module}
0\rightarrow P^m \rightarrow P^{m-1}\rightarrow\cdots\rightarrow P^0\rightarrow D(\bar{\Lambda })e_i\rightarrow0.
\end{equation}
By applying $j^*$ to (\ref{eqn:projective resolution of injective module}), we have the following exact sequence:
 $$0\rightarrow j^*(P^m) \rightarrow j^*(P^{m-1})\rightarrow\cdots\rightarrow j^*(P^0)\rightarrow j^*(D(\bar{\Lambda })e_i) \rightarrow0.$$
Then $j^*(D(\bar{\Lambda })e_i)= j^*j_!(D(\Lambda)e_i) = D(\Lambda )e_i\oplus (D(\Lambda )e_{x})^{\oplus s_i}$.
Because $j^*$ preserves projective modules,  we obtain $\pd _\Lambda  (D(\Lambda )e_i)<\infty$.
Similarly, one can check that $\id_\Lambda  (\Lambda e_i)<\infty$ by noting that $j^*$ also preserves injective modules, and $j^*(\bar{\Lambda }e_i)\cong \Lambda e_i\oplus (\Lambda e_{E(x)})^{\oplus r_i}$ for some $r_i$.
Therefore, $\Lambda $ is Gorenstein.

For the ``only if '' part, obviously, $j^*i_*\simeq \Id$. For any $M\in \mod\Lambda$, the adjunction yields that $\alpha:\Hom_{\Lambda}(j_*(M), i_*(M))\xrightarrow{\sim}\Hom(M,j^*i_*(M))=\Hom(M,M)$.
Let $f_M:j_*(M)\rightarrow i_*(M)$ be $\alpha^{-1}(\Id_M)$. It follows from (\ref{eqn:action of j on simple modules}) that $f_{S_i}$ is an automorphism of $S_i$ if $i\neq x$; and
$f_{S_x}$ equals to $g$ in (\ref{eqn:short exact sequnce of Sx}).
By induction on the length of modules, using (\ref{eqn:short exact sequnce of Sx}), there exists a short exact sequence:
\begin{equation}\label{equation kan extension}
0\rightarrow (\bar{\Lambda }e_{E(x)})^{\dim M_{x}} \rightarrow j_*(M) \xrightarrow{f_M}i_*(M)\rightarrow0.
\end{equation}
Dually, there is an exact sequence
\begin{equation}\label{equation kan extension 2}
0\rightarrow i_*(M) \xrightarrow{g_M}j_!(M)\rightarrow (D(\bar{\Lambda})e_x)^{\dim M_{E(x)}}\rightarrow0.
\end{equation}

For any projective $\Lambda $-module $U$, it follows from (\ref{equation kan extension}) that $\pd_{\bar{\Lambda }}(i_*(U))\leq 1$ since $j_*(U)$ is projective.
Similarly, for any injective $\Lambda $-module $V$, we have $\id_{\bar{\Lambda }}(i_*(V))\leq 1$.

Because $\Lambda_E$ is finite-dimensional, there is no nonzero paths from $x$ to $E(x)$, and no nonzero paths from $E(x)$ to $x$ in $\Lambda $. It follows from  \cite[Chapter III.2, Lemma 2.4, Lemma 2.6]{ASS}  that $\bar{\Lambda }e_{E(x)}\cong i_*(\Lambda e_{E(x)})$ and $D(\bar{\Lambda }e_x)\cong i_*(D(\Lambda )e_x)$. Since $\Lambda $ is Gorenstein, we have $\pd_\Lambda (D(\Lambda )e_x)<\infty$. By applying $i_*$ to the minimal projective resolution of $D(\Lambda )e_x$, it is easy to see that
$\pd_{\bar{\Lambda}}D(\bar{\Lambda })e_x= \pd_{\bar{\Lambda }} i_*(D(\Lambda )e_x)<\infty$ since $i_*$ is exact and $\pd_{\bar{\Lambda }}(i_*(U))\leq 1$ for any projective $\Lambda $-module $U$. Similarly, we can prove that
$\id_{\bar{\Lambda }}\bar{\Lambda }e_x<\infty$.

Note that $D(\bar{\Lambda })e_i=j_!(D(\Lambda )e_i)$ for any indecomposable injective $\bar{\Lambda }$-module $D(\bar{\Lambda })e_i$. Because $\Lambda $ is Gorenstein, $\pd_\Lambda  (D(\Lambda )e_i)<\infty$. By applying $j_*$ to the minimal projective resolution of $D(\Lambda )e_i$, we have $\pd_{\bar{\Lambda }} (j_*(D(\Lambda )e_i))<\infty$. So $\pd_{\bar{\Lambda }} (i_*(D(\Lambda )e_i))<\infty$ by (\ref{equation kan extension}). It follows from (\ref{equation kan extension 2}) that
$\pd_{\bar{\Lambda }} D(\bar{\Lambda })e_i=\pd_{\bar{\Lambda }}j_!(D(\Lambda )e_i)<\infty$ since $\pd_{\bar{\Lambda }} D(\bar{\Lambda })e_x<\infty$.

Dually, for any indecomposable projective $\bar{\Lambda }$-module $\bar{\Lambda }e_i$, we can get that $\id_{\bar{\Lambda }} \bar{\Lambda }e_i<\infty$. Therefore, $\bar{\Lambda }$ is Gorenstein.
\end{proof}

Now, we can prove that Gorenstein algebras are closed under taking Br\"{u}stle's gluing.

\begin{proposition}\label{proposition of gorenstein properties of gluiing algebras}
Let $\Lambda =KQ/I$ be a finite-dimensional algebra, and $E$ an involution on the set of vertices of $Q$. Assume that the algebra $\Lambda_E$ constructed from $\Lambda $ by gluing the vertices along $E$ is finite-dimensional. Then $\Lambda_E$ is Gorenstein if and only if $\Lambda $ is Gorenstein.
\end{proposition}
\begin{proof}
By applying the same induction as in the proof of Lemma \ref{lemma finite dimension of gluing algebras},
we can assume that there is only one pair of vertices $(x,E(x))$ such that $x\neq E(x)$, and the added arrow for $\bar{Q}$ is $\gamma=\alpha_{(x,E(x))}:x\rightarrow E(x)$.

For the ``if'' part, it follows from Lemma \ref{lemma Gorenstein of S and barS} that $\bar{\Lambda }$ is Gorenstein. For any indecomposable injective $\Lambda_E$-module $D(\Lambda_E)e_{\bar{i}}$, $i\neq x$, the dual of (\ref{equation finite projective dimension of Phi 1}) implies that there is a short exact sequence in $\mod\bar{\Lambda }$:
$$ 0\rightarrow \Phi(D(\Lambda_E)e_{\bar{i}})\rightarrow D(\bar{\Lambda })e_i\oplus (D(\bar{\Lambda })e_{E(x)})^{\oplus s_i}\rightarrow (D(\bar{\Lambda })e_{x})^{\oplus s_i}\rightarrow0,$$
for some $s_i$.
So $\id_{\bar{\Lambda }} \Phi(D(\Lambda_E)e_{\bar{i}})\leq1$. Because $\bar{\Lambda }$ is Gorenstein, we deduce that
$$\pd_{\bar{\Lambda }} (\Phi(D(\Lambda_E)e_{\bar{i}}))<\infty,\text{ for any }\bar{i}.$$
From the proof of Proposition \ref{proposition singularity equivalence of SE barS}, we get that $\Phi(K^b(\proj \Lambda_E))= K^b(\proj \bar{\Lambda })\cap \Phi(D^b(\mod \Lambda_E))$,
So $\Phi(D(\Lambda_E)e_{\bar{i}})\in K^b(\proj \bar{\Lambda })\cap \Phi(D^b(\mod \Lambda _E))$, which implies that there is a bounded complex $P^\bullet$ in $K^b(\proj \Lambda _E)$, such that
$\Phi(P^\bullet)\cong \Phi(D(\Lambda_E)e_{\bar{i}})$ in $D^b(\mod \bar{\Lambda })$. Since $\Phi:D^b(\mod \Lambda_E)\rightarrow D^b(\mod \bar{\Lambda })$ is fully faithful, we get that
$D(\Lambda_E)e_{\bar{i}}\cong P^\bullet$ in $D^b(\mod \Lambda_E)$, which implies that $\pd_{\Lambda_E} (D(\Lambda_E)e_{\bar{i}})<\infty$.

By considering the opposite algebras, one can prove that $\id_{\Lambda_E} \Lambda _Ee_{\bar{i}}<\infty$ dually for any $\bar{i}$. Then $\Lambda_E$ is Gorenstein.

For the ``only if'' part, the composition of $j^*$ and $\Phi$ yields an exact functor $j^*\Phi:\mod \Lambda_E\rightarrow \mod \Lambda $.
By the definitions, $j^*\Phi$ preserves projective modules and injective modules.
In particular, $j^*\Phi(\Lambda _Ee_{\bar{i}})=\Lambda e_i\oplus (\Lambda e_{E(x)})^{\oplus t_i}$ for some $t_i$ if $i\neq x$; and $j^*\Phi(\Lambda _Ee_{\bar{x}})=\Lambda e_x\oplus \Lambda e_{E(x)}$.
Dually, $j^*\Phi(D(\Lambda_E)e_{\bar{i}})=D(\Lambda )e_i\oplus(D(\Lambda )e_{x})^{\oplus s_i}$ for some $s_i$ if $i\neq x$; and $j^*\Phi(D(\Lambda_E)e_{\bar{x}})=D(\Lambda )e_x\oplus D(\Lambda )e_{E(x)}$.

Because $\Lambda_E$ is Gorenstein, for any $i\neq E(x)$, it is easy to see that $\id_\Lambda  j^*\Phi(\Lambda _Ee_{\bar{i}})<\infty$, and then $\id_\Lambda  \Lambda e_i<\infty$. For $i=E(x)$, using the fact $j^*\Phi(\Lambda _Ee_{\bar{x}})=\Lambda e_x\oplus \Lambda e_{E(x)}$, we obtain $\id_\Lambda  Se_{E(x)}<\infty$. It is dual to prove that
$$\pd_\Lambda  D(\Lambda )e_i<\infty$$
for any $i\in Q_0$. Therefore, $\Lambda $ is Gorenstein.
\end{proof}

\subsection{Singularity categories of $1$-Gorenstein monomial algebras}
In this subsection, we describe the stable categories of Gorenstein projective modules (i.e., singularity categories) for $1$-Gorenstein monomial algebras.

We always assume that $A=KQ/I$ is a $1$-Gorenstein monomial algebra in this subsection.

Denote by $\cc(A)$ the set of equivalence classes (with respect to cyclic permutation) of \emph{repetition-free} cyclic paths $c=\alpha_n\cdots \alpha_1$ in $Q$ satisfies the following:
\begin{itemize}
\item there is a positive integer $r$ such that any oriented path $\alpha_{i+r-1}\cdots\alpha_i\in\bF$ for any $i\in\Z/n\Z$.
\end{itemize}
In this case, $n$ is called the \emph{length} of $c$, and $r$, which is unique, is called \emph{the length of relations} of $c$. For any $c\in\cc(A)$, let $Q_c$ be the subquiver formed by all arrows appearing in $c$.

Proposition \ref{proposition elements of ideal for 1-Gorenstein monomial algebras} implies that for any perfect path $p$, there is one and only one $c\in\cc(A)$ such that $p\in Q_c$.

The following lemma is obvious; see \cite[Page 1122]{CSZ}.
\begin{lemma}\label{lemma property of oriented cycle}
Let $A=KQ/I$ be a $1$-Gorenstein monomial algebra. Then for any arrow $\alpha$, there is at most one oriented cycle $c\in\cc(A)$ such that $\alpha$ is in $Q_c$.
\end{lemma}

\begin{lemma}\label{lemma direct sum}
Let $A=KQ/I$ be a $1$-Gorenstein monomial algebra. For any two perfect paths $p_1,q_1$, if $p_1,q_1$ are contained in $Q_{c_1}$ and $Q_{c_2}$ respectively for two non-equivalent oriented cycle $c_1,c_2\in \cc(A)$,
then $\Hom_{A}(Ap_1,Aq_1)=0$.
\end{lemma}
\begin{proof}
The proof uses only the definitions of perfect pairs and perfect paths, and morphisms of modules.

Let $(p_2,p_1)$ and $(q_2,q_1)$ be the perfect pairs. For any morphism $f:Ap_1\rightarrow Aq_1$, because the top of every indecomposable Gorenstein projective module is simple, if the induced morphism $\Top(f):\Top(Ap_1)\rightarrow\Top(Aq_1)$ is nonzero, then $f$ is surjective.
In this case, $t(p_1)=t(q_1)$. Then $p_2q_1\in I$ since $p_2p_1\in I$ and $f$ is surjective (see the proof of Theorem \ref{singularity category of 1-Gorenstein algebras}). Because $(p_2,p_1)$ is a perfect pair, there exists a nonzero path $q_1'$ such that $q_1=p_1q_1'$. So $p_1\in Q_{c_2}$.
In conclusion, $p_1\in Q_{c_1},Q_{c_2}$ and $c_1,c_2\in\cc(A)$ are non-equivalent. This gives a contradiction to Lemma \ref{lemma property of oriented cycle}.

So it only happens that $\Im(f)\subseteq\rad (Aq_1)$. Suppose $f\neq0$. Then $f(p_1)\neq0$. We can assume that $f(p_1)=\sum_{k=1}^m a_k u_kq_1$ with each $u_kq_1$ a nonzero path ending at $t(p_1)$ and $a_k\neq0$ for each $k$. Because $p_2p_1\in I$ and $f$ is a morphism of modules, we deduce that $\sum_{k=1}^m a_kp_2u_kq_1\in I$. Then $p_2u_kq_1\in I$ since $A$ is monomial. By considering $u_1$, there exists a path $v_1$ such that $v_1q_2=p_2u_1$ since $(q_2,q_1)$ is a perfect pair.
It follows from $q_2q_1=0$ and $u_1q_1\neq0$ that $l(q_2)> l(u_1)$. Then there exists a nonzero (nontrivial) path $p''$ such that $q_2=p''p'$, $p_2=v_1p''$. So the nontrivial path $p''$
is in $Q_{c_1}$ and $Q_{c_2}$, giving a contradiction to Lemma \ref{lemma property of oriented cycle}.
\end{proof}

Let $A=KQ/I$ be a $1$-Gorenstein monomial algebra. Recall that $Q_c$ is the subquiver of $Q$ for each $c\in\cc(A)$. Restricting $I$ to $KQ_c$, we can define an ideal $I_c$ of $KQ_c$, and then an (finite-dimensional) algebra $\Lambda_c:=KQ_c/I_c$. Note that $\Lambda_c$ is a subring (not subalgebra) of $A$.

\begin{lemma}\label{lem:1-Gorenstein of subalgebra}
Keep the notations as above. Then $\Lambda_c$ is a $1$-Gorenstein monomial algebra.
\end{lemma}
\begin{proof}
It follows from Theorem \ref{proposition characterize of 1 Gorenstein monomial algebras}.
\end{proof}

\begin{lemma}\label{lemma singularity category of 1-Gorenstein 1}
Keep the notations as above. Let $A=KQ/I$ be a $1$-Gorenstein monomial algebra. Then
$$D_{sg}(A)\simeq \coprod_{c\in \cc(A)} D_{sg}(\Lambda _{c}).$$
\end{lemma}
\begin{proof}
Let $e_c$ be the idempotent $\sum_{i\in Q_c} e_i$. Then $\Lambda_c$ is a subalgebra of $e_c A e_c$.
%Define $j_\lambda:=A e_c\otimes_{e_cAe_c}-:\mod e_cA e_c\rightarrow \mod A$ and $j_\mu:=e_c A\otimes_A-: \mod A \rightarrow \mod e_cA e_c$. Then $(j_\lambda,j_\mu)$ is an adjoint pair.
%Define $\iota_\lambda=e_cAe_c\otimes_{\Lambda _{c}}-:\mod \Lambda _{c}\rightarrow \mod e_cAe_c$ and $\iota_\mu:\mod e_cAe_c\rightarrow \mod \Lambda _{c}$ the restriction functor.
%Then $(\iota_\lambda,\iota_\mu)$ is an adjoint pair.
%Define $\iota_*:= j_\lambda\circ \iota_\lambda$ and $\iota^*:= \iota_\mu\circ j_\mu$. Then $(\iota_*,\iota^*)$ is an adjoint pair.
%In fact, $\iota^*:\mod A\rightarrow \mod\Lambda_c$ is the restriction functor,
%and $\iota_*= Ae_c\otimes_{\Lambda_c}-: \mod\Lambda_c\rightarrow\mod A$.
Then $Ae_c$ is naturally a right $e_cAe_c$-module, and then a right $\Lambda_c$-module. Denote by $\varphi^c: =Ae_c\otimes_{\Lambda_c}-: \mod\Lambda_c\rightarrow\mod A$. In the following, we prove that $\varphi^c$ is an exact functor, which induces a fully embedding $\tilde{\varphi}^c:\underline{\Gproj} (\Lambda_c)\rightarrow \underline{\Gproj} (A)$.

First, $\varphi^c$ preserves Gorenstein projective modules. In fact, easily, $\varphi^c$ preserves projective modules. For any indecomposable non-projective Goresntein projective $\Lambda_c$-module, it is of the form $\Lambda_cp$ for some perfect path in $\Lambda_c$ by Theorem \ref{theorem bijection of perfect path and Gorenstein projective modules}. Then $\varphi^c(\Lambda_cp)=Ap$. Obviously, $p$ is also a perfect path of $A$. So $\varphi^c$ preserves Gorenstein projective modules.

%Denote by $\iota=\iota_c:\Lambda _{c}\rightarrow A$ the natural embedding. Then $A$ is left (and also right) $\Lambda _{c}$-module. So we get an adjoint pair $(\iota_*,\iota^*)$, where $\iota_*=A\otimes_{\Lambda _{c}}-:\mod \Lambda _{c}\rightarrow \mod A$, and $\iota^*:\mod A\rightarrow \mod \Lambda _{c}$ is the restriction functor. First, we check that $A$ is projective as left (and also right) $\Lambda _{c}$-module.
%By viewing $c$ as a subquiver of $Q$, let $I_c$ be the restriction of $I$ to the subquiver $c$ for any $c\in\cc(A)$. Then $I_c$ is a two sided ideal of $kc$ and $\Lambda _{c}=kc/I_c$.

Second, $\varphi^c$ is an exact functor.
For any arrow $\alpha$, Lemma \ref{lemma property of oriented cycle} shows that there is at most one $c\in\cc(A)$ such that $\alpha\in Q_c$. It follows from Theorem \ref{proposition characterize of 1 Gorenstein monomial algebras} that $I=\langle I_c\mid c\in\cc(A) \rangle$. Fix a $c\in\cc(A)$.
Let $p_1=\alpha_r\cdots \alpha_1$, and $p_2=\beta_s\cdots \beta_1$ with $t(\beta_s)=s(\alpha_1)$ be two nonzero paths in $A$, where $\alpha_i,\beta_j\in Q_1$ for any $i,j$. If
$p_1\in Q_c$ and $\beta_s\notin Q_c$, then it is easy to see that $p_1p_2$ is nonzero in $A$; similarly, if $p_2\in Q_c$, and $\alpha_1\notin Q_c$, then $p_1p_2$ is nonzero in $A$.
As $A$ (resp. $\Lambda_c$) has a basis formed by all nonzero paths in $A$ (resp. $\Lambda_c$), it follows from \cite[Chapter III.2, Lemma 2.4]{ASS} that $e_cA$ is projective as a left $\Lambda _{c}$-module. In fact, for each $i\in Q_c$, $(e_cA)e_i$ is isomorphic to $\Lambda_ce_i\bigoplus ((\Lambda_ce_j)^{\oplus n_j})$, where $n_j$ is the number of nonzero paths $\gamma_m\cdots \gamma_2\gamma_1$ from $i$ to $j$ such that $\gamma_j\in Q_1$ for $1\leq j\leq m$ and $\gamma_m\notin Q_c$.
By considering the dual quiver, similarly, $Ae_c$ is projective as right $\Lambda _{c}$-module. Then $\varphi^c$ is an exact functor, which also preserves projective modules.

It follows that $\varphi^c$ induces an exact functor $\tilde{\varphi}^c:\underline{\Gproj} (\Lambda_c)\rightarrow \underline{\Gproj} (A)$.

Third, $\tilde{\varphi}^c$ is fully faithful. For any two perfect paths $p,q$ in $\Lambda_c$ (and then in $A$), \cite[Lemma 2.3, Lemma 3.6]{CSZ} shows that there exist $K$-linear isomorphisms:
$$\Hom_{\underline{\Gproj} (\Lambda_c)}(\Lambda_cp,\Lambda_cq)\cong \frac{p\Lambda_c\cap \Lambda_cq}{p\Lambda_cq},\text{ and }\Hom_{\underline{\Gproj} (A)}(Ap,Aq)\cong \frac{pA\cap Aq}{pAq}.$$
By viewing $Q_c$ as a subquiver of $Q$, there exists a $K$-linear injective map $\iota_c:(p\Lambda_c\cap \Lambda_cq)\hookrightarrow (pA\cap Aq)$. Easily,
$p\Lambda_cq= pAq\cap \Lambda_c=pAq\cap (p\Lambda_c\cap \Lambda_cq)$, then $\iota_c$ induces a $K$-linear injective map
$$\bar{\iota}: \frac{p\Lambda_c\cap \Lambda_cq}{p\Lambda_cq} \rightarrow \frac{pA\cap Aq}{pAq}.$$
For any nonzero path in $pA\cap Aq$ which is not in $p\Lambda_c\cap \Lambda_cq$, it is of the form $pt_1$ and $t_2q$ for some nonzero path $t_1,t_2\in A$ and $t_1,t_2\notin\Lambda_c$.
So $pt_1=t_2q$ is of the form $ptq\in pAq$ for some nonzero path $t\in A$. As nonzero paths form a basis of $A$, we obtain that $\bar{\iota}$ is epic. So
$\Hom_{\underline{\Gproj} (\Lambda_c)}(\Lambda_cp,\Lambda_cq)\cong \Hom_{\underline{\Gproj} (A)}(Ap,Aq)$.
Then $\tilde{\varphi}^c$ is fully faithful.

Lemma \ref{lemma direct sum} implies that
$$\Hom_{\underline{\Gproj} (A)}(\Im(\tilde{\varphi}^c),\Im(\tilde{\varphi}^{c'}))=0$$
for any nonequivalent $c, c'\in\cc(A)$.
So there is a fully faithful functor
$$F= (\tilde{\varphi}^c)_{c\in\cc(A)} :\coprod_{c\in \cc(A)} \underline{\Gproj}(\Lambda _{c})\rightarrow \underline{\Gproj}(A).$$
Furthermore, Theorem \ref{theorem bijection of perfect path and Gorenstein projective modules} shows that $F$ is dense. The desired result follows from Buchweitz's Theorem.
\end{proof}

\begin{lemma}\label{lemma singularity category of 1-Gorenstein 2}
Keep the notations as above. Let $A=KQ/I$ be a $1$-Gorenstein monomial algebra. Then for any $c\in\cc(A)$,
$$D_{sg}(\Lambda _{c})\simeq D^b(\A_{r-1})/[\tau^{n}],$$
where $n=l(c)$, $r$ is the length of relations for $c$, $D^b(\A_{r-1})$ is the derived category of type $\A_{r-1}$ with $\tau$ the Auslander-Reiten functor, and $D^b(\A_{r-1})/[\tau^{n}]$ denotes the triangulated orbit category in the sense of \cite{Ke}.
\end{lemma}
\begin{proof}
Let $Z_n$ be a basic $n$-cycle, that is an oriented cycle of $n$ vertices, $\Lambda =KZ_n/J^r$, where $J$ is the ideal of $KZ_n$ generated by arrows of $Z_n$. Then $\Lambda $ is a self-injective Nakayama algebra and $\underline{\mod}\Lambda \simeq  D^b(\A_{r-1})/[\tau^{n}]$.

For $c\in\cc(A)$, we claim that there is a series of quivers $Q^0=Z_n$, $Q^1$, $\dots$, $Q^m$, and involutions $E_0$, $E_1$, $\dots$, $E_m$ on the sets of vertices of $Q^0$, $Q^1$, $\dots$, $Q^m$ respectively, such that the following hold:
\begin{itemize}
\item $Q^{i+1}=Q^i(E_i)$ for $1\leq i<m$, and $Q^m(E_m)=Q_c$;
\item  $\Lambda _m=\Lambda _c$, where $\Lambda _i=Q^i(E_i)/I(E_i)$ for $0\leq i\leq m$.
\end{itemize}
In fact, let $c=\alpha_n\cdots\alpha_1\in\cc(A)$. Denote by $\mu$ the number of vertices in $Q_c$. Assume that $(Q_c)_0=\{1,2,\dots, \mu\}$. Then $\mu\leq n$. Let $m=n-\mu$. If $\mu=n$, then the claim is obvious.
Otherwise, there exists a vertex $a$ such that the number of arrows adjacent to $a$ is $>2$. Because $c$ is an oriented cycle, there exist at least two arrows $\alpha_i,\alpha_j$ ($i<j$) such that $t(\alpha_i)=a=t(\alpha_j)$.  Then define $Q^m$ to be the quiver as follows:
\begin{itemize}
\item the vertex set of $Q^m$ is $(Q_c)_0\cup \{\mu+1\}$;
\item the arrow set of $Q^m$ is $\{ \alpha'_k\mid \alpha_k \text{ if } k\notin\{j,j+1\}; \alpha_j':s(\alpha_j)\rightarrow \mu+1,\alpha'_{j+1}:\mu+1\rightarrow t(\alpha_{j+1})\}$.
\end{itemize}
So $c'=\alpha'_n\cdots \alpha_1'$ is a repetition-free cycle of $Q_m$. Denote by $r$ the length of relations for $c$. Let $I_m$ be the ideal of $KQ_m$ generated by $\alpha'_{k+r-1}\cdots \alpha'_{k}$ for all $k\in\Z/n\Z$.
It is easy to see that $\Lambda_c$ is a Br\"{u}stle's gluing algebra of $KQ_m/I_m$ by identifying the vertices $a$ and $\mu+1$.
Inductively, we can construct the desired series of quivers $Q^0=Z_n$, $Q^1$, $\dots$, $Q^m$.

Then Theorem \ref{theorem singularity equivalence of gluing vertices} implies that $\Lambda $, $\Lambda _1$, $\dots$, $\Lambda _m=\Lambda _c$ are singularity equivalent. So
$$D_{sg}(\Lambda _{c})\simeq D_{sg}(\Lambda )\simeq \underline{\mod}\Lambda \simeq D^b(\A_{r-1})/[\tau^{n}].$$
\end{proof}

By combining Lemma \ref{lemma singularity category of 1-Gorenstein 1} and Lemma \ref{lemma singularity category of 1-Gorenstein 2}, we get the following result.

\begin{theorem}\label{singularity category of 1-gorenstein monomial algebras}
Let $A$ be a $1$-Gorenstein monomial algebra. Then
$$\underline{\Gproj}A \simeq D_{sg} (A)\simeq \coprod_{c\in \cc(A)}  D^b(\A_{r_c-1})/[\tau^{n_c}],$$
where $n_c=l(c)$, $r_c$ is the length of relations for $c$, $D^b(\A_{r_c-1})$ is the derived category of type $\A_{r_c-1}$ with $\tau$ the Auslander-Reiten functor, and $D^b(\A_{r_c-1})/[\tau^{n_c}]$ denotes the triangulated orbit category in the sense of \cite{Ke}.
\end{theorem}

\end{document}